\RequirePackage{fix-cm}   
\documentclass[11pt]{amsart}

\usepackage{amssymb}
\usepackage{mathtools}
\usepackage{bm}
\usepackage{enumitem}
\usepackage{xcolor}
\usepackage{hyperref}

\DeclareFontFamily{OT1}{cmrbigring}{}
\DeclareFontShape{OT1}{cmrbigring}{bx}{n}{<->s*[1.3]cmbx10}{}
\DeclareSymbolFont{thickringfont}{OT1}{cmrbigring}{bx}{n}
\SetSymbolFont{thickringfont}{bold}{OT1}{cmrbigring}{bx}{n}
\DeclareMathAccent{\thickmathring}{\mathalpha}{thickringfont}{"17}
\let\mathring\thickmathring

\definecolor{darkblue}{rgb}{0.0,0,0.7}
\newcommand{\darkblue}{\color{darkblue}}
\newcommand{\defn}[1]{\emph{\darkblue #1}}
\hypersetup{colorlinks=true, citecolor=darkblue, linkcolor=darkblue}

\setlist[enumerate]{
  label=\textnormal{({\roman*})},
  ref={\roman*}}

\makeatletter
\def\th@plain{%
  \thm@notefont{}%
  \itshape
}
\def\th@definition{%
  \thm@notefont{}%
  \normalfont
}
\makeatother

\theoremstyle{definition}
\newtheorem{definition}{Definition}[section]
\newtheorem{example}[definition]{Example}
\newtheorem{notation}[definition]{Notation}
\newtheorem{remark}[definition]{Remark}
\newtheorem{algorithm}[definition]{Algorithm}

\theoremstyle{plain}
\newtheorem{proposition}[definition]{Proposition}
\newtheorem{theorem}[definition]{Theorem}
\newtheorem{corollary}[definition]{Corollary}
\newtheorem{lemma}[definition]{Lemma}
\newtheorem{conjecture}[definition]{Conjecture}
\newtheorem{theoremA}{Theorem}

\newtheorem{theoremB}{Theorem}

\newtheorem{theoremC}{Theorem}

\numberwithin{figure}{section}
\numberwithin{equation}{section}

\newcommand{\ts}{\thinspace}
\newcommand{\Zge}{\mathbb{Z}_{\ge 0}}
\newcommand{\Rge}{\mathbb{R}_{\ge 0}}
\newcommand{\Rpos}{\mathbb{R}_{> 0}}
\newcommand{\BR}{\mathrm{BR}}
\newcommand{\trop}{\mathrm{trop}}

\newcommand{\heavy}[1]{%
  \pdfliteral direct{2 Tr 0.2 w}#1\pdfliteral direct{0 Tr}}
\let\bmorig\bm
\renewcommand{\bm}[1]{\heavy{\bmorig{#1}}}

\newcommand{\vx}{\bm{x}}
\newcommand{\va}{\bm{\alpha}}
\newcommand{\vg}{\bm{\gamma}}
\newcommand{\vb}{\bm{\beta}}
\newcommand{\vP}{\bm{P}}
\newcommand{\vd}{\bm{d}}
\newcommand{\vw}{\bm{w}}
\newcommand{\vv}{\bm{v}}
\newcommand{\ve}{\bm{e}}
\newcommand{\vrho}{\rho}

\title{Bounded ratios for Lorentzian polynomials}
\author[\ts Aayush Bathija]{Aayush Bathija}
\address[Aayush Bathija]{Oak Park High School, Oak Park, CA 91377.}
\email{\texttt{bathija.aayush@gmail.com}}

\author[\ts Prince Rohatgi]{Prince Rohatgi}
\address[Prince Rohatgi]{Oak Park High School, Oak Park, CA 91377.}
\email{\texttt{aditya.prince.rohatgi@gmail.com}}

\author[\ts Daniel Soskin]{Daniel Soskin}
\address[Daniel Soskin]{Department of Mathematics, UCLA, Los Angeles, CA 90095.}
\email{\texttt{dsoskin@math.ucla.edu}}

\date{}

\begin{document}
\begin{abstract}
We study multiplicative inequalities among the coefficients of Lorentzian
polynomials through the notion of \emph{bounded ratios}. Our main structural result
completely characterizes the cone of bounded ratios for Lorentzian
polynomials of degree $n$ in $k$ variables. We show that the dual of the cone
of bounded ratios is generated by equivalence classes of M-convex functions
modulo affine functions. For ternary Lorentzian forms of arbitrary degree
$n\ge3$, we show that the cone of bounded ratios is generated by triangular
ratios and determine the optimal bounding constant of every bounded ratio.
Furthermore, we characterize the pairs
$(n,k)$ for which the cone of bounded ratios can be computed by tropicalizing
products of $n$ linear forms in $k$ variables with nonnegative coefficients.
\end{abstract}

\maketitle
\vskip-.5cm

\section{Introduction}

A real symmetric matrix has \defn{Lorentzian signature} if it is nonsingular
and has exactly one positive eigenvalue.  Following
\cite[Definition~2.1]{BH20}, a homogeneous polynomial
$f\in\mathbb R[x_1,\dots,x_k]$ of degree $n$ is \defn{strictly Lorentzian} if
all its coefficients are positive and, for every sequence
$i_1,\dots,i_{n-2}\in\{1,\dots,k\}$, the Hessian of
$\partial_{i_1}\cdots\partial_{i_{n-2}}f$ has Lorentzian signature.  A
homogeneous polynomial of degree $n$ in $k$ variables is \defn{Lorentzian} if
it is a coefficientwise limit of strictly Lorentzian polynomials.

Fix integers $n,k\ge2$, and write $\vx=(x_1,\dots,x_k)$.  For
$\va=(\alpha_1,\dots,\alpha_k)\in\Zge^k$, set
\[
  |\va|_1=\alpha_1+\cdots+\alpha_k,
  \qquad
  \vx^{\va}=x_1^{\alpha_1}\cdots x_k^{\alpha_k},
  \qquad
  \va!=\alpha_1!\cdots\alpha_k!,
\]
and let
\[
  H(n,k)=\{\va\in\Zge^k:|\va|_1=n\}.
\]
Let $\ve_1,\dots,\ve_k$ be the standard basis of $\mathbb R^k$. 

Let $\mathring L(n,k)\subseteq\Rpos^{H(n,k)}$ be the set of coefficient
vectors $\vP=(P_{\va})_{\va\in H(n,k)}$ for which
\[
  \sum_{\va\in H(n,k)}P_{\va}\vx^{\va}
\]
is strictly Lorentzian.

We study multiplicative inequalities of the form
\[
  \prod_{\va\in H(n,k)}P_{\va}^{\,m_{\va}}
  \ \le\ C\prod_{\va\in H(n,k)}P_{\va}^{\,m'_{\va}}
  \qquad\text{for all }\vP\in\mathring L(n,k),
\]
where the exponents $m_{\va}$ and $m'_{\va}$ are nonnegative integers and
$C$ is a positive constant.

Mixed volumes provide a geometric source of such inequalities.  For convex
bodies $K_1,\dots,K_n$ in $\mathbb R^n$, their \defn{mixed volume} is
\[
  V(K_1,\dots,K_n)=\frac1{n!}
  \left.\partial_{t_1}\cdots\partial_{t_n}
  \operatorname{Vol}_n(t_1K_1+\cdots+t_nK_n)
  \right|_{t_1=\cdots=t_n=0}.
\]
For a collection $\mathbf K=(K_1,\dots,K_k)$ of convex bodies in
$\mathbb R^n$, set
\[
  V_{\va}(\mathbf K)=V\!\left(
  \underbrace{K_1,\dots,K_1}_{\alpha_1\text{ times}},\dots,
  \underbrace{K_k,\dots,K_k}_{\alpha_k\text{ times}}
  \right).
\]
By Minkowski's theorem \cite{Min03}, the normalized volume
polynomial has the expansion
\[
  f_{\mathbf K}(\vx)
  =\frac1{n!}\operatorname{Vol}_n\!\left(x_1K_1+\cdots+x_kK_k\right)
  =\sum_{\va\in H(n,k)}
  V_{\va}(\mathbf K)\frac{\vx^{\va}}{\va!}.
\]
For $\vb\in H(n-2,k)$, write
$\partial^{\vb}=\partial_1^{\beta_1}\cdots\partial_k^{\beta_k}$.  Then
\[
  \operatorname{Hess}(\partial^{\vb}f_{\mathbf K})
  =\left(V_{\vb+\ve_i+\ve_j}(\mathbf K)\right)_{i,j=1}^k
  =\left(
  V\!\Bigl(K_i,K_j,
  \underbrace{K_1,\dots,K_1}_{\beta_1\text{ times}},\dots,
  \underbrace{K_k,\dots,K_k}_{\beta_k\text{ times}}
  \Bigr)
  \right)_{i,j=1}^k.
\]
The Alexandrov--Fenchel inequality states that 
\[
  V_{\vb+2\ve_i}(\mathbf K)\cdot V_{\vb+2\ve_j}(\mathbf K)
  \le V_{\vb+\ve_i+\ve_j}(\mathbf K)^2.
\]

We study multiplicative inequalities through the notion of bounded ratios. This
approach has previously been applied to multiplicative inequalities for
totally positive matrices in \cite{BF08,SG25}, to the multiplicative inequalities on positive loci of
cluster varieties in \cite{GGS26}, and to multiplicative inequalities among
entries of Lorentzian matrices in \cite{HHSW}.

For an exponent 
$\vg=(\gamma_{\va})_{\va\in H(n,k)}\in\mathbb R^{H(n,k)}$ and a coefficient  
$\vP=(P_{\va})_{\va\in H(n,k)}\in\mathring  L(n,k)$ vectors, we define the Laurent
monomial
\[
  R_{\vg}(\vP)=\prod_{\va\in H(n,k)}P_{\va}^{\,\gamma_{\va}}.
\]
The exponent vector $\vg$ is a \defn{bounded ratio} if there is a constant
$C>0$ such that $R_{\vg}(\vP)\le C$ for every
$\vP\in\mathring L(n,k)$.  We denote the cone of all such exponent vectors by
\[
  \BR_{\mathring L}(n,k)=
  \left\{\vg\in\mathbb R^{H(n,k)}:
  R_{\vg}\text{ is bounded above on }\mathring L(n,k)\right\}.
\]

Any uniform upper bound $C$ for $R_{\vg}$ yields a multiplicative inequality of the
form
\[
  \prod_{\va\in H(n,k)}P_{\va}^{\,\max\{\gamma_{\va},0\}}
  \ \le\ C
  \prod_{\va\in H(n,k)}P_{\va}^{\,\max\{-\gamma_{\va},0\}}
\]
for every $\vP\in\mathring L(n,k)$. Because Lorentzian polynomials are
coefficientwise limits of strictly Lorentzian polynomials, the displayed
multiplicative inequality extends by continuity to every Lorentzian
polynomial. Br\"and\'en and Huh proved that $f_{\mathbf K}$ is Lorentzian
\cite[Theorem~4.1]{BH20}. Consequently, every bounded ratio yields a
multiplicative inequality for volume polynomials.

Coefficient inequalities for volume polynomials have a long history. The
Alexandrov--Fenchel inequality generalizes the isoperimetric and
Brunn--Minkowski inequalities and plays an important role in modern convex
geometry, see \cite{Huh26} and \cite{HMW}. The Khovanskii--Teissier
inequalities are used to prove log-concavity of combinatorial sequences, see
\cite{Huh12,HK12}.
Shenfeld and van Handel gave geometric characterizations of the equality
cases of the Alexandrov--Fenchel inequality for several classes of convex
bodies \cite{SvH23}. Chan and Pak showed
that these equality cases are not in the polynomial hierarchy unless that
hierarchy collapses \cite{CP23}.

Recently, Huang, Huh, Wang, and the third author characterized the cone of
bounded ratios $\BR_{\mathring L}(2,k)$ for Lorentzian polynomials of degree
two in $k$ variables \cite{HHSW}.  They also computed the optimal
upper bounds for quadratic Lorentzian polynomials in three variables.

To state our main structural theorem, we recall the following notion. Following
\cite{Mur03}, a function
\\
$\nu:H(n,k)\to\mathbb R$ is \defn{M-convex} if, for all
$\va,\vb\in H(n,k)$ and every $i$ with $\alpha_i>\beta_i$, there is a $j$ with
$\alpha_j<\beta_j$ such that
\[
  \nu(\va)+\nu(\vb)\ge
  \nu(\va-\ve_i+\ve_j)+\nu(\vb+\ve_i-\ve_j).
\]
For $\vg\in\mathbb R^{H(n,k)}$ and $\nu:H(n,k)\to\mathbb R$, we write
\[
  \langle\vg,\nu\rangle
  =\sum_{\va\in H(n,k)}\gamma_{\va}\nu(\va).
\]
An exponent vector $\vg$ is \defn{balanced} if
\[
  \sum_{\va\in H(n,k)}\gamma_{\va}\va=\bm0,
  \qquad\text{i.e.,}\qquad
  \sum_{\va\in H(n,k)}\gamma_{\va}\alpha_i=0
  \quad\text{for every }i\in\{1,\ldots,k\}.
\]
Let
\[
  V_{n,k}=
  \left\{\vg\in\mathbb R^{H(n,k)}:
  \sum_{\va\in H(n,k)}\gamma_{\va}\va=\bm0\right\}
\]
be the space of balanced exponent vectors.

Our main structural theorem completely characterizes the cone
$\BR_{\mathring L}(n,k)$ for every degree $n$ and every number of variables
$k$. We prove that this cone is a full-dimensional (in the balanced subspace
$V_{n,k}$) polyhedral cone and describe
its dual in terms of classes of M-convex functions modulo affine functions.

\begin{theoremA}[Main structural theorem]\label{thm:main-intro}
For any integers $n\ge2$ and $k\ge2$,
\[
  \BR_{\mathring L}(n,k)=
  \left\{\vg\in V_{n,k}:
  \langle\vg,\nu\rangle\ge0
  \text{ for every M-convex }\nu:H(n,k)\to\mathbb R\right\}.
\]
\end{theoremA}

\medskip

This result extends the dual description of $\BR_{\mathring L}(2,k)$ in
\cite[Theorem~B]{HHSW} from quadratic Lorentzian polynomials to strictly
Lorentzian polynomials of arbitrary degree.
\medskip

In three variables, we show that $\BR_{\mathring L}(n,3)$ is the cone spanned
by the exponent vectors of the following \defn{triangular ratios}:
\[
  R_{\vrho(\vb;i\mid jk)}(\vP)
  =
  \frac{P_{\vb+2\ve_i}\,P_{\vb+\ve_j+\ve_k}}
       {P_{\vb+\ve_i+\ve_j}\,P_{\vb+\ve_i+\ve_k}},
  \qquad
  \vb\in H(n-2,3),\quad \{i,j,k\}=\{1,2,3\}.
\]

For a bounded ratio $\vg$, the \defn{optimal bounding constant} is the least
uniform upper bound for $R_{\vg}$, namely,
\[
  f_{\mathring L}(\vg)
  \;\coloneqq\;
  \inf\bigl\{\, C>0 \;:\; R_{\vg}(\vP)\le C
  \ \text{ for every } \vP\in\mathring L(n,k) \,\bigr\}.
\]

We then compute the optimal upper bound for every element of
$\BR_{\mathring L}(n,3)$ for every $n\ge3$, the cone of bounded ratios for
Lorentzian forms of degree $n$ in three variables. This extends the optimal bounds obtained in
\cite[Theorem~C]{HHSW}.

For $a,b,c\ge0$, set
\[
 m\coloneqq a+b+c,
 \qquad
 \Delta(a,b,c)\coloneqq a^2+b^2+c^2-2ab-2ac-2bc.
\]
Define $\mathfrak m(a,b,c)\coloneqq 1$ when $\Delta(a,b,c)\le0$.  When
$\Delta(a,b,c)>0$, one entry is larger than the sum of the other two.  After
permuting the entries, assume that $a>b+c$ and set
\[
 \mathfrak m(a,b,c)
 \coloneqq 2^m\,
 \frac{a^a b^b c^c(a-b-c)^{a-b-c}}
 {(a-b+c)^{a-b+c}(a+b-c)^{a+b-c}},
 \qquad 0^0=1.
\]
This defines a symmetric function on $\Rge^3$.

For $n\ge3$ and $\vg\in\BR_{\mathring L}(n,3)$, let
\[
 \Lambda_n(\vg)
 \coloneqq
 \left\{\lambda=(\lambda_{\vb,i})\in\Rge^{\,3\binom n2}:\ 
 \vg=\sum_{\vb\in H(n-2,3)}\sum_{i=1}^3
 \lambda_{\vb,i}\cdot\vrho(\vb;i\mid jk)\right\},
\]
where, for each $i$, the indices $j,k$ are the two distinct elements of
$\{1,2,3\}\setminus\{i\}$, in either order.

The set $\Lambda_n(\vg)$ is nonempty for every
$\vg\in\BR_{\mathring L}(n,3)$ because $\BR_{\mathring L}(n,3)$ is
generated by the exponent vectors of triangular ratios. The array $\lambda$ has one
row of three nonnegative entries for each $\vb\in H(n-2,3)$, so there are
$3\binom n2$ entries in total. Thus $\Lambda_n(\vg)$ consists of all
representations of $\vg$ as a nonnegative combination of the exponent
vectors of triangular ratios.

\begin{theoremB}[Optimal bounding constants for ternary Lorentzian forms]
\label{thm:all-ternary-constants-intro}
For every $n\ge3$ and $\vg\in\BR_{\mathring L}(n,3)$,
\[
 f_{\mathring L}(\vg)
 =\sup_{\vP\in\mathring L(n,3)}R_{\vg}(\vP)
 =\left(\prod_{\va\in H(n,3)}(\va!)^{-\gamma_{\va}}\right)
 \min_{\lambda\in\Lambda_n(\vg)}
 \prod_{\vb\in H(n-2,3)}
 \mathfrak m(\lambda_{\vb,1},\lambda_{\vb,2},\lambda_{\vb,3}).
\]
If $\vg\ne0$, the supremum is not
attained on $\mathring L(n,3)$.
\end{theoremB}

Specializing to $n=3$ gives multiplicative inequalities of
various flavors, including symmetric and asymmetric inequalities and
inequalities on a fixed support. In the list below, we recover the reverse
Khovanskii--Teissier inequality (the first) and the Alexandrov--Fenchel
inequality (the third), together with a few of our other favorite inequalities.
Every ternary Lorentzian cubic satisfies the
following inequalities, and each displayed constant is optimal.
\[
\begin{gathered}
 P_{120}P_{201}\le P_{210}P_{111},
 \qquad
 P_{300}P_{111}\le\frac43 P_{210}P_{201},\\[3pt]
 P_{120}P_{102}\le\frac14 P_{111}^{2},
 \qquad
 P_{300}P_{120}\le\frac13 P_{210}^{2},\\[3pt]
 P_{012}P_{120}P_{201}\le\frac4{27}P_{111}^{3},
 \qquad
 P_{300}P_{030}P_{003}\le\frac1{216}P_{111}^{3},\\[8pt]
 (P_{300}P_{030}P_{003})^2
 \le\frac1{729}
 P_{210}P_{201}P_{120}P_{021}P_{102}P_{012},\\[8pt]
 P_{300}P_{030}^{2}P_{003}^{4}
 \le\frac1{2187}P_{120}P_{102}^{2}P_{012}^{4},\\[8pt]
 P_{030}P_{003}P_{210}P_{201}P_{021}
 \le\frac1{18}P_{120}^{2}P_{102}P_{012}P_{111}.\\[5pt]
\end{gathered}
\]

Although we have not computed the optimal bounding constant for every
element of $\BR_{\mathring L}(3,4)$,
Corollary~\ref{weighted-star-product-lorentzian} gives the following
inequality for every Lorentzian cubic in $k\ge4$ variables. For
$a,b,c\ge0$ and pairwise distinct $i,j,\ell,r\in\{1,\ldots,k\}$,
\[
\begin{gathered}
 P_{\ve_i+\ve_\ell+\ve_r}^{\,a}
 P_{\ve_i+\ve_j+\ve_r}^{\,b}
 P_{\ve_i+\ve_j+\ve_\ell}^{\,c}
 P_{3\ve_i}^{\,a+b+c}
 \le
 C\,P_{2\ve_i+\ve_j}^{\,b+c}
 P_{2\ve_i+\ve_\ell}^{\,a+c}
 P_{2\ve_i+\ve_r}^{\,a+b},\\[4pt]
 \text{where}\quad C=\left(\frac43\right)^{a+b+c}\cdot
 \frac{(a+b+c)^{a+b+c}a^ab^bc^c}
 {(a+b)^{a+b}(a+c)^{a+c}(b+c)^{b+c}}.
\end{gathered}
\]
Here $0^0=1$, and $C$ is optimal for each fixed $a,b,c$ and every $k\ge4$.

In particular, setting $a=b=c=1$ gives the sharp inequality
\[
 P_{\ve_i+\ve_\ell+\ve_r}
 P_{\ve_i+\ve_j+\ve_r}
 P_{\ve_i+\ve_j+\ve_\ell}
 P_{3\ve_i}^{3}
 \le
 P_{2\ve_i+\ve_j}^{2}
 P_{2\ve_i+\ve_\ell}^{2}
 P_{2\ve_i+\ve_r}^{2}.
\]

Finally, we characterize the pairs $(n,k)$ for which
$\BR_{\mathring L}(n,k)$ coincides with the cone of bounded ratios for volume
polynomials arising from families of Cartesian products of intervals. We
denote this cone by $\BR_{\mathring Q}(n,k)$. The corresponding volume
polynomials are products of $n$ linear forms in $k$ variables with
nonnegative coefficients.

\newpage
\begin{theoremC}
\label{thm:comparison-classification-intro}
For all $n,k\ge2$,
\[
 \BR_{\mathring L}(n,k)=\BR_{\mathring Q}(n,k)
\]
if and only if $n=2$, or $k=2$, or $k=3$ and $n\le5$. Equivalently, the
inclusion is strict precisely when $n\ge3$ and $k\ge4$, or $n\ge6$ and
$k\ge3$.
\end{theoremC}

The classification is summarized in the following table.
\[
\begin{gathered}
 \BR_{\mathring L}(n,k)\stackrel{?}{=}\BR_{\mathring Q}(n,k)
 \\[4pt]
\begin{array}{c|ccccccc}
 k\backslash n & 2 & 3 & 4 & 5 & 6 & 7 & \cdots \\ \hline
 2 & = & = & = & = & = & = & \cdots \\
 3 & = & = & = & = & \ne & \ne & \cdots \\
 4 & = & \ne & \ne & \ne & \ne & \ne & \cdots \\
 5 & = & \ne & \ne & \ne & \ne & \ne & \cdots \\
 6 & = & \ne & \ne & \ne & \ne & \ne & \cdots \\
 7 & = & \ne & \ne & \ne & \ne & \ne & \cdots \\
 \vdots & \vdots & \vdots & \vdots & \vdots & \vdots & \vdots & \ddots
\end{array}
\end{gathered}
\]

\subsection*{Concurrent work}

\begingroup
\hypersetup{urlcolor=black,citecolor=black}

Corollary~1.2 of Lorenzo Baldi and Mario Kummer's work~\cite{BK26}
implies Theorem~A of this paper. Theorem~1.1 of \cite{BK26} gives a
description of bounded ratios in a more general setting.

As we recently learned through private communication, Josephine Yu had
previously announced the result stated as Theorem~1.1 in~\cite{BK26}.
Abeer Al Ahmadieh, Felipe Rinc\'on, Cynthia Vinzant, and Josephine Yu plan to
include this result in the next version of their
paper~\cite{ahmadieh2025tropicalizingprincipalminorspositive}.

Dijia Chen, Bowen Gan, Ivy Liu, Zemeng Wang, and Chengzhi Wu have posted two
papers \cite{CGLWW26a}, \cite{CGLWW26b} that overlap substantially with our paper,
mostly with \cite{CGLWW26a}. 

More specifically, we show that for ternary Lorentzian polynomials the cone of bounded ratios is generated by triangular ratios and we provide the optimal upper bound for each such bounded ratio. In \cite{CGLWW26a} the authors solve the same two problems for ternary Lorentzian polynomials restricted to any M-convex supports. It is easy to show that the ratios from the bounded cones for any M-convex support share the same upper bound as if considered as ratios in full supported bounded cone. In \cite{CGLWW26a} the authors continue with the comparison of optimal upper bounds for all Lorentzian polynomials, volume polynomials and basis-generating polynomials of rank-three matroids.  

In \cite{CGLWW26b} the authors determine for which degree and number of variables the cone of bounded ratios for any M-convex set is generated by quadratic-slice ratios. The authors use cut directions to certify that certain bounded ratios cannot be generated by quadratic-slice bounded ratios. These cut directions are either our M-convex split functions, up to sign and an affine term, or the sum of two adjacent split functions. We use split functions to describe a family of the facets of the cone of bounded ratios, while in \cite{CGLWW26b} the author use cut directions for separation argument. Unlike in \cite{CGLWW26b}, we compare the cones of bounded ratios for all Lorentzian polynomials and for volume polynomials which are product of nonnegative linear forms. 

We acknowledge that all of the research groups mentioned above worked
independently.

\endgroup

\section{\texorpdfstring{Dual description of the cone of bounded ratios
$\BR_{\mathring{L}}(n,k)$}{Dual description of the cone of bounded ratios BRL(n,k)}}
\label{sec:BRL}

Fix integers $n\ge2$ and $k\ge2$, and write
$\vx=(x_1,\dots,x_k)$. For
$\va=(\alpha_1,\dots,\alpha_k)\in\Zge^k$, set
\[
  |\va|_1 \;=\; \alpha_1 + \cdots + \alpha_k,
  \qquad
  \vx^{\va} \;=\; x_1^{\alpha_1} \cdots x_k^{\alpha_k},
  \qquad
  \va! \;=\; \alpha_1! \cdots \alpha_k! ,
\]
and we let
\[
  H(n,k) \;=\; \bigl\{\, \va \in \Zge^{k} \;:\; |\va|_1 = n \,\bigr\}
\]
be the set of exponent vectors of the degree-$n$ monomials in
$x_1,\dots,x_k$. We write $\ve_1, \dots, \ve_k$ for the standard basis of
$\mathbb{R}^k$. We use bold font to denote tuples or vectors.

\subsection{Lorentzian polynomials}

A real symmetric matrix has \defn{Lorentzian signature} if it is nonsingular
and has exactly one positive eigenvalue, that is, if its signature is
$(+,-,\dots,-)$.

\begin{definition}[{\rm \defn{Lorentzian polynomials}}, {\cite[Definition~2.1]{BH20}}]\label{def:strictly-lorentzian}
A homogeneous polynomial $f \in \mathbb{R}[x_1, \dots, x_k]$ of degree
$n \ge 2$ is \defn{strictly Lorentzian} if
\begin{enumerate}
  \item all coefficients of $f$ are positive, and
  \item for every sequence $i_1, \dots, i_{n-2} \in \{1, \dots, k\}$, the
        Hessian of $\partial_{i_1} \cdots \partial_{i_{n-2}} f$ has Lorentzian
        signature.
\end{enumerate}
A homogeneous polynomial of degree $n$ is \defn{Lorentzian} if it is a limit of
strictly Lorentzian polynomials of degree $n$ in $k$ variables.
\end{definition}

The polynomial $\partial_{i_1} \cdots \partial_{i_{n-2}} f$ is a quadratic
form, so its Hessian is a matrix with constant entries.

\subsection{Ratios}

\begin{definition}[{\rm \defn{Ratios}}]\label{def:ratio}
For
$\vg=(\gamma_{\va})_{\va\in H(n,k)}\in\mathbb R^{H(n,k)}$ and
$\vP=(P_{\va})_{\va\in H(n,k)}$, define the associated Laurent monomial by
\[
  R_{\vg}
  \;=\;
  \prod_{\va \in H(n,k)} P_{\va}^{\,\gamma_{\va}} .
\]
By abuse of notation, we call both the exponent vector $\vg$ and the Laurent
monomial $R_{\vg}$ a \defn{ratio}. 
\end{definition}

\begin{definition}[{\rm \defn{Bounded ratios and optimal bounding constants}}]\label{def:BR-general}
Let $X \subseteq \Rpos^{H(n,k)}$ be a nonempty set of coefficient vectors. An
exponent vector $\vg=(\gamma_{\va})_{\va\in H(n,k)}$ is a \defn{bounded ratio}
on $X$ if it satisfies the following condition:
\begin{quote}
There is a positive constant $C$ such that
$\displaystyle\prod_{\va \in H(n,k)} P_{\va}^{\,\gamma_{\va}} \le C$
for every $\vP=(P_{\va})\in X$.
\end{quote}
We denote the set of bounded ratios on $X$ by $\BR(X)$.
The \defn{optimal bounding constant} of $\vg \in \BR(X)$, denoted
$f(\vg) = f_X(\vg)$, is the infimum of all constants $C$ satisfying this condition:
\[
  f_{X}(\vg)
  \;=\;
  \inf\bigl\{\, C>0 \;:\; R_{\vg}(\vP)\le C
  \ \text{ for every } \vP\in X \,\bigr\}.
\]
\end{definition}

\begin{definition}[{\rm \defn{The cone $\BR_{\mathring{L}}(n,k)$}}]
\label{def:BR-strict}
Let
\[
  \mathring{L}(n,k)  \;\coloneqq\;
  \Bigl\{\, (P_{\va})_{\va \in H(n,k)} \in \Rpos^{H(n,k)}
  \;:\; \sum_{\va \in H(n,k)} P_{\va}\, \vx^{\va}
  \ \text{is strictly Lorentzian} \,\Bigr\}
\]
be the set of coefficient vectors of strictly Lorentzian polynomials of
degree $n$ in $k$ variables, and put
\[
  \BR_{\mathring{L}}(n,k) \;\coloneqq\; \BR\bigl(\mathring{L}(n,k)\bigr),
  \qquad
  f_{\mathring{L}} \;\coloneqq\; f_{\mathring{L}(n,k)} .
\]
\end{definition}

\begin{remark}\label{rem:normalization}
Replacing the coefficients $P_{\va}$ by the normalized coefficients
$\va!\,P_{\va}$ multiplies $R_{\vg}$ by a positive constant and therefore
does not affect whether the ratio is bounded.
\end{remark}

\subsection{M-convex functions and affine quotients}

\begin{definition}[{\rm \defn{M-convex function}}, {\cite{Mur03}}]\label{def:Mconvex}
A function $\nu \colon H(n,k) \to \mathbb{R}$ is \defn{M-convex} if, for all
$\va, \vb \in H(n,k)$ and every $i$ with $\alpha_i > \beta_i$, there is a $j$
with $\alpha_j < \beta_j$ such that
\[
  \nu(\va) + \nu(\vb) \;\ge\;
  \nu(\va - \ve_i + \ve_j) + \nu(\vb + \ve_i - \ve_j).
\]
\end{definition}

\begin{notation}\label{not:BRL}
Let $\mathbb{R}^{H(n,k)}$ denote the space of tuples indexed by the finite
set $H(n,k)$. We identify a function $\nu \colon H(n,k) \to \mathbb{R}$ with
its tuple $(\nu(\va))_{\va}$. Under this identification, exponent vectors and
functions belong to the same space, equipped with the pairing
$\langle \vg, \nu \rangle = \sum_{\va} \gamma_{\va}\, \nu(\va)$. Let
\[
  V_{n,k} \;:=\; \Bigl\{ \vg \in \mathbb{R}^{H(n,k)} \;:\;
  \textstyle\sum_{\va} \gamma_{\va}\, \va = \bm{0} \Bigr\}
\]
be the space of \defn{balanced} exponent vectors. 

An \defn{affine-linear function} on $\mathbb{R}^{k}$ is a map
$\vx \mapsto c + \sum_i b_i x_i$, where $c, b_1,\dots,b_k \in \mathbb{R}$.
An \defn{affine function} on $H(n,k)$ is the restriction of such a map. Let
$\operatorname{Aff}(H(n,k))$ denote the space of these affine functions.

For an arbitrary function $f:H(n,k) \rightarrow \mathbb{R}$, write $[f]$ for
its class in the quotient vector space
$\mathbb{R}^{H(n,k)}/\operatorname{Aff}(H(n,k))$. Thus $[f] = [g]$ exactly
when $f - g \in \operatorname{Aff}(H(n,k))$. By
Lemma~\ref{lem:Maffine} below, adding an affine function preserves
M-convexity. Hence M-convexity is a property of the class $[\nu]$. Define
\[
  \mathcal{M}_{n,k} \;:=\; \overline{\operatorname{cone}}
  \bigl\{ [\nu] \;|\; \nu \colon H(n,k) \to \mathbb{R} \text{ is M-convex}
  \bigr\}
  \;\subseteq\;
  \mathbb{R}^{H(n,k)} / \operatorname{Aff}(H(n,k)) .
\]
\end{notation}

\begin{definition}[{\rm \defn{Local discrete Hessian and local metric field}}]
\label{def:local-hessian-field}
For $f \in \mathbb{R}^{H(n,k)}$ and $\vb \in H(n-2,k)$, let
$d^{\vb}(f)$ be the symmetric $k \times k$ matrix with entries
\[
  d^{\vb}_{ij}(f) \;=\;
  f(\vb + 2\ve_i) + f(\vb + 2\ve_j) - 2 f(\vb + \ve_i + \ve_j),
  \qquad 1 \le i,j \le k.
\]
We call $d^{\vb}(f)$ the \defn{local discrete Hessian} of $f$ at $\vb$. Its
diagonal vanishes. The family
\[
  (d^{\vb}(f))_{\vb \in H(n-2,k)}
\]
is the \defn{local metric field} of $f$.
\end{definition}

\begin{lemma}\label{lem:affine}
The space $\operatorname{Aff}(H(n,k))$ is spanned by the $k$ coordinate
functions $\va \mapsto \alpha_i$. Moreover,
\[
  \dim \operatorname{Aff}(H(n,k)) = k.
\]
Two pairs
$(c,\bm b)$ and $(c',\bm b')$ define the same affine function on $H(n,k)$ if and
only if
\[
  \bm b' = \bm b + t(1,\dots,1),
  \qquad
  c' = c - nt
  \qquad\text{for some } t \in \mathbb{R}.
\]
\end{lemma}

\begin{proof}
Since every $\va \in H(n,k)$ satisfies $\sum_i \alpha_i = n$, the constant
function with value $c$ can be written on $H(n,k)$ as
$c = (\frac{c}{n})\sum_i \alpha_i$. Thus,
\[
  c + \sum_i b_i \alpha_i
  \;=\; \sum_i \Bigl( b_i + \frac{c}{n} \Bigr) \alpha_i.
\]
Hence the coordinate functions span $\operatorname{Aff}(H(n,k))$. They are
linearly independent. Indeed, if
$\sum_i t_i \alpha_i = 0$ for all $\va \in H(n,k)$, evaluating at 
$\va = n\ve_j$ gives $n t_j = 0$, so $t_j = 0$ for every $j$. Hence the
dimension is $k$.

Suppose
$(c-c') + \sum_i (b_i - b'_i)\alpha_i = 0$ for all $\va \in H(n,k)$. Evaluating
at $\va = n \ve_j$ gives $(c-c') + n(b_j - b'_j) = 0$, that is
$b_j - b'_j = -(c-c')/n$, the same value for every $j$. Define
$t = (c-c')/n$. This gives the stated relation. Conversely, if
$\bm b' = \bm b + t(1,\dots,1)$
and $c' = c - nt$, then on $H(n,k)$
\[
  c' + \sum_i b'_i \alpha_i
  \;=\; c - nt + \sum_i b_i \alpha_i + t\sum_i \alpha_i
  \;=\; c + \sum_i b_i \alpha_i .
\]
\end{proof}
We next prove that affine functions have vanishing local discrete Hessians.
\begin{lemma}
\label{lem:affine-hessian}
Every $\ell \in \operatorname{Aff}(H(n,k))$ satisfies $d^{\vb}_{ij}(\ell) = 0$
for all $\vb \in H(n-2,k)$ and all $i,j$. Consequently, the local metric field
$(d^{\vb}(f))_{\vb}$ depends only on the class $[f]$.
\end{lemma}

\begin{proof}
Write $\ell(\va) = c + \sum_i b_i \alpha_i$ and put
$\langle \bm b, \vb \rangle = \sum_i b_i \beta_i$. Then
\[
  \ell(\vb + 2\ve_i) = c + \langle \bm b, \vb\rangle + 2b_i,
  \quad
  \ell(\vb + 2\ve_j) = c + \langle \bm b, \vb\rangle + 2b_j,
  \quad
  \ell(\vb + \ve_i + \ve_j) = c + \langle \bm b, \vb\rangle + b_i + b_j,
\]
so that
\[
  d^{\vb}_{ij}(\ell)
  = \bigl( 2c + 2\langle \bm b, \vb\rangle + 2b_i + 2b_j \bigr)
  - 2\bigl( c + \langle \bm b, \vb\rangle + b_i + b_j \bigr)
  = 0 .
\]
Since $d^{\vb}(f+l)=d^{\vb}(f)+d^{\vb}(l)$, $f$ and $f + \ell$ have the same
local metric field.
\end{proof}
We next record the corresponding orthogonality statement: balanced vectors
annihilate affine functions.
\begin{lemma}
\label{lem:balanced-affine}
Let $\vg \in V_{n,k}$. Thus $\sum_{\va} \gamma_{\va}\,\va = \bm 0$. Then
\[
  \sum_{\va} \gamma_{\va} = 0
  \qquad\text{and}\qquad
  \langle \vg, \ell \rangle = 0
  \quad\text{for every } \ell \in \operatorname{Aff}(H(n,k)).
\]
Consequently, every balanced pairing $\langle \vg, f \rangle$ depends only on
the class $[f]$.
\end{lemma}

\begin{proof}
The defining relation $\sum_{\va}\gamma_{\va}\va=\bm0$ gives the $k$ scalar
equations
\[
  \sum_{\va} \gamma_{\va}\,\alpha_i = 0
  \qquad (i = 1,\dots,k).
\]
Summing them over $i$ and exchanging the two finite sums gives
\[
  0 \;=\; \sum_{i=1}^{k} \sum_{\va} \gamma_{\va}\,\alpha_i
  \;=\; \sum_{\va} \gamma_{\va} \sum_{i=1}^{k} \alpha_i
  \;=\; \sum_{\va} \gamma_{\va}\,|\va|_1
  \;=\; n \sum_{\va} \gamma_{\va} ,
\]
because $|\va|_1 = n$ for every $\va \in H(n,k)$. Since $n \ge 2$, we have
$\sum_{\va} \gamma_{\va} = 0$.

Now for $\ell(\va) = c + \sum_i b_i \alpha_i$,
\[
  \langle \vg, \ell \rangle
  \;=\; c \sum_{\va} \gamma_{\va}
  \;+\; \sum_{i=1}^{k} b_i \sum_{\va} \gamma_{\va}\,\alpha_i
  \;=\; 0 ,
\]
because both terms vanish. Linearity of the pairing gives
$\langle \vg, f + \ell\rangle = \langle \vg, f\rangle$.
\end{proof}

\begin{lemma}[The dual of the affine quotient]\label{lem:quotientdual}
Recall that the coefficient pairing on $\mathbb R^{H(n,k)}$ is the bilinear
form
\[
 \langle\vg,\nu\rangle=\sum_{\va\in H(n,k)}\gamma_{\va}\,\nu(\va),
\]
the standard inner product in the basis indexed by $H(n,k)$. It is symmetric
and nondegenerate because $\langle\vg,\mathbf1_{\va}\rangle=\gamma_{\va}$ for
the indicator function $\mathbf1_{\va}$ of a point $\va\in H(n,k)$. Hence
$\vg\mapsto\langle\vg,\cdot\rangle$ is injective and, by equality of
dimensions, identifies $\mathbb R^{H(n,k)}$ with its own dual. For a subspace
$A$, recall that its annihilator is
$A^{\perp}=\{\vg:\langle\vg,\ell\rangle=0\ \text{for every }\ell\in A\}$.
Under this pairing, the dual of the quotient by affine functions is
canonically identified with the balanced space
\[
 \left(\mathbb R^{H(n,k)}/\operatorname{Aff}(H(n,k))\right)^*
 \;\cong\;\operatorname{Aff}(H(n,k))^\perp
 \;=\;V_{n,k}.
\]
Accordingly, all polar cones and supporting hyperplanes below are taken with
respect to the pairing between
\[
  \mathbb R^{H(n,k)}/\operatorname{Aff}(H(n,k))
  \qquad\text{and}\qquad
  V_{n,k}.
\]
Throughout, we use the positive-dual convention
\[
 C^\vee=\{y:\langle y,x\rangle\ge0\text{ for every }x\in C\}.
\]
\end{lemma}

\begin{proof}
We first recall the standard identification of the dual of a quotient with an
annihilator. Let \(E\) be a finite-dimensional vector space, let
\(A\subseteq E\) be a subspace, and let \(\pi\colon E\to E/A\) be the quotient
map. Composition with \(\pi\) gives the
\emph{pullback}
\[
 \pi^{*}\colon (E/A)^{*}\longrightarrow E^{*},
 \qquad
 \pi^{*}(\lambda)=\lambda\circ\pi ,
\]
which is linear. It is injective. Indeed, if \(\lambda\circ\pi=0\), then \(\lambda\)
vanishes on the image of \(\pi\), which is all of \(E/A\), so \(\lambda=0\). Its
image is contained in \(A^{\perp}\), since \(\pi\) kills \(A\) and therefore
\(\pi^{*}(\lambda)(a)=\lambda(\pi(a))=0\) for \(a\in A\). Conversely, if
\(\mu\in E^{*}\) vanishes on \(A\), then \(\mu\) is constant on the cosets of
\(A\), so \(\lambda(x+A)=\mu(x)\) is a well-defined linear functional on
\(E/A\) with \(\pi^{*}(\lambda)=\mu\). Hence \(\pi^{*}\) is an isomorphism
\((E/A)^{*}\cong A^{\perp}\).

Apply this with \(E=\mathbb R^{H(n,k)}\) and
\(A=\operatorname{Aff}(H(n,k))\) and use the coefficient pairing above to
identify \(E^{*}\) with \(E\). Under that identification \(A^{\perp}\) is the
annihilator as defined above.

It remains to identify the annihilator. If \(\vg\in V_{n,k}\), then
Lemma~\ref{lem:balanced-affine} gives \(\langle\vg,\ell\rangle=0\) for every affine
\(\ell\), so \(V_{n,k}\subseteq A^\perp\).  Conversely, if
\(\vg\in A^\perp\), pairing it with each coordinate function
\(\va\mapsto\alpha_i\) gives
\[
 0=\sum_{\va\in H(n,k)}\gamma_{\va}\alpha_i
 \qquad(i=1,\ldots,k).
\]
These are precisely the coordinates of
\(\sum_{\va}\gamma_{\va}\va=\bm0\). Hence \(\vg\in V_{n,k}\).
\end{proof}

\subsection{Semialgebraic sets and Puiseux series}

\begin{definition}[{\rm \defn{Semialgebraic sets and maps}}]\label{def:semialg}
A subset of $\mathbb{R}^{N}$ is \defn{semialgebraic} if it is a finite union of
sets defined by finitely many polynomial equations and strict inequalities with
real coefficients. A map between semialgebraic sets is \defn{semialgebraic}
if its graph is semialgebraic. The class of semialgebraic sets is closed under finite unions,
finite intersections, and complements. By the Tarski--Seidenberg theorem, it
is also closed under projection. A semialgebraic map is \defn{continuous
semialgebraic} if it is also continuous in the usual sense.
\end{definition}

\begin{definition}[{\rm \defn{Nash function}}, {\cite[Definition~2.9.3]{BCR98}}]\label{def:nash}
Let $U$ be an open semialgebraic subset of $\mathbb R^n$. A function
$f\colon U\to\mathbb R$ is a \defn{Nash function} if it is semialgebraic and
of class $C^\infty$.
\end{definition}

\begin{proposition}[{\cite[Proposition~8.1.8]{BCR98}}]
\label{prop:bcr-nash}
Let $U$ be an open semialgebraic subset of $\mathbb R^n$. A function
$f\colon U\to\mathbb R$ is Nash if and only if it is analytic algebraic on $U$.
\end{proposition}

We use Proposition~\ref{prop:bcr-nash} only to conclude that a Nash function
is real analytic. Together with the following reparametrization theorem, this
yields a convergent series representation.

\begin{proposition}[{\cite[Proposition~8.1.12]{BCR98}}]
\label{prop:bcr-parametrization}
Let $g\colon[0,\delta)\to\mathbb R$ be a continuous semialgebraic function. There
exist a positive integer $p$, an element $\epsilon\in\mathbb R$ with
$0<\epsilon\le\delta^{1/p}$, and a Nash function
$f\colon(-\epsilon,\epsilon)\to\mathbb R$ such that
$f(t)=g(t^p)$ for every $t\in[0,\epsilon)$.
\end{proposition}

\begin{definition}[{\rm \defn{Real convergent Puiseux series and germs}}]\label{def:puiseux}
A \defn{real Puiseux series} is a formal series
\[
 \varphi(t)=\sum_{l\ge l_0}c_l\,t^{\,l/m},
 \qquad
 c_l\in\mathbb R,\quad l_0\in\mathbb Z,\quad m\in\mathbb Z_{>0},
\]
that is, a Laurent series in a rational root of $t$. We work over the field of
real convergent Puiseux series
\[
 K=\bigcup_{m\ge1}\mathbb R((t^{1/m}))_{\mathrm{conv}},
\]
where the subscript means that the series converges for all sufficiently small
$t>0$. 

The field $K$ is real closed by \cite[Section~1.5]{Spe05}.
We order $K$ by declaring a nonzero series positive if its first nonzero
coefficient is positive. For example, $t$ is positive and smaller than every
positive real number.  If the first nonzero term of $\varphi$ is
$c_{l_0}t^{l_0/m}$,
write
\[
 \operatorname{val}\varphi=l_0/m\in\mathbb Q.
\]
Two functions, each defined on an interval $(0,\varepsilon)$, have the same
\defn{germ at $0^+$} if they agree on some interval $(0,\delta)$.
\end{definition}

\begin{lemma}\label{lem:semialg-puiseux}
Let $\mathcal G_{\mathrm{sa}}$ be the $\mathbb R$-algebra of germs at $0^+$
of continuous semialgebraic functions
$\varphi:[0,\varepsilon)\to\mathbb R$, where $\varepsilon>0$ may depend
on $\varphi$. For every $\xi=[\varphi]\in\mathcal G_{\mathrm{sa}}$, there
is a unique element $\widehat\xi\in K$ such that
\[
 \widehat\xi(t)=\varphi(t)
 \qquad\text{for all sufficiently small }t>0.
\]
The resulting map
\[
 \iota:\mathcal G_{\mathrm{sa}}\longrightarrow K,
 \qquad
 \xi\longmapsto\widehat\xi,
\]
is an injective $\mathbb R$-algebra homomorphism. Moreover, every nonzero
$\widehat\xi$ has nonnegative valuation, and if its leading term is $ct^q$,
then
\[
 q=\operatorname{val}(\widehat\xi)\in\mathbb Q_{\ge0},
 \qquad
 \varphi(t)=ct^q(1+o(1))
 \quad(t\to0^+).
\]
\end{lemma}

\begin{proof}
Fix $\xi=[\varphi]\in\mathcal G_{\mathrm{sa}}$.
Proposition~\ref{prop:bcr-parametrization}, applied with $g=\varphi$ and
$\delta=\varepsilon$, gives $p\in\mathbb Z_{>0}$,
$0<\eta\le\varepsilon^{1/p}$, and a Nash function
$h:(-\eta,\eta)\to\mathbb R$ such that
\[
 h(s)=\varphi(s^p)\qquad(0\le s<\eta).
\]
By Proposition~\ref{prop:bcr-nash}, the function $h$ is real analytic.
After decreasing $\eta$ if necessary, it therefore has the convergent Taylor
expansion
\[
 h(s)=\sum_{j\ge0}a_js^j \qquad (|s|<\eta).
\]
Consequently,
\[
 \varphi(t)=h(t^{1/p})=\sum_{j\ge0}a_jt^{j/p}
 \qquad(0\le t<\eta^p).
\]
The series on the right belongs to $K$, has nonnegative valuation when it is
nonzero, and represents $\xi$. This proves existence.

We now prove uniqueness. Suppose that $u,v\in K$ represent the same germ. Choose
$M\in\mathbb Z_{>0}$ such that
$u,v\in\mathbb R((t^{1/M}))_{\mathrm{conv}}$. If $u-v\ne0$ and its first
nonzero term is $ct^q$, then convergence gives
\[
 (u-v)(t)=ct^q(1+o(1))
 \qquad(t\to0^+).
\]
Thus $(u-v)(t)\ne0$ for all sufficiently small $t>0$, contradicting that
$u$ and $v$ represent the same germ. Hence $u=v$.

If $\xi=[\varphi]$ and $\zeta=[\psi]$, then
$\widehat\xi+\widehat\zeta$ and
$\widehat\xi\,\widehat\zeta$ represent the germs of
$\varphi+\psi$ and $\varphi\psi$, respectively. Uniqueness shows that
$\iota$ preserves addition and multiplication, and it plainly preserves
real constants. It is injective because $\widehat\xi=0$ precisely when
$\varphi$ vanishes on some interval $(0,\delta)$, that is, precisely when
$\xi$ is the zero germ. Finally, factoring the leading term of a nonzero
$\widehat\xi$ gives the asymptotic formula.
\end{proof}

\begin{definition}[{\rm \defn{Lorentzian polynomials over a real closed field}}]
\label{def:lorK}
Let $F$ be a real closed field. A subset $J\subseteq H(m,k)$ is
\defn{M-convex} if, for all $\va,\va'\in J$ and every $i$ with
$\alpha_i>\alpha'_i$, there is a $j$ with $\alpha_j<\alpha'_j$ such that
\[
 \va-\ve_i+\ve_j\in J
 \qquad\text{and}\qquad
 \va'+\ve_i-\ve_j\in J.
\]
We follow the convention that the empty set is M-convex. This is the symmetric
exchange definition used in \cite{BH20}, which refers to \cite{Mur03}.

For $m\ge0$, let $\mathrm{M}^{m}_{k}(F)$ be the set of homogeneous degree-$m$
polynomials in $F[x_1,\dots,x_k]$ with nonnegative coefficients and M-convex
support. Following \cite[Definition~3.18]{BH20} for the real Puiseux field and
\cite[Definition~8.1]{HLSV24} for arbitrary ordered fields, set
\[
 \mathrm{L}^{0}_{k}(F)=\mathrm{M}^{0}_{k}(F),
 \qquad
 \mathrm{L}^{1}_{k}(F)=\mathrm{M}^{1}_{k}(F),
\]
\[
 \mathrm{L}^{2}_{k}(F)
 =\left\{f\in\mathrm{M}^{2}_{k}(F):
   \operatorname{Hess}(f)\text{ has at most one eigenvalue in }F_{>0}
  \right\},
\]
and, for $m\ge3$,
\[
\mathrm{L}^{m}_{k}(F)
=\left\{f\in\mathrm{M}^{m}_{k}(F):
  \partial^{\vb}f\in\mathrm{L}^{2}_{k}(F)
  \text{ for every }\vb\in H(m-2,k)
 \right\}.
\]
With the empty-support convention above, the zero polynomial belongs to
$\mathrm{L}^{m}_{k}(F)$ for every $m$.
The elements of $\mathrm{L}^{m}_{k}(F)$ are the \defn{Lorentzian polynomials
over $F$}. When $F=\mathbb R$, this algebraic definition agrees with the
coefficientwise-closure definition in Definition~\ref{def:strictly-lorentzian}
by \cite[Theorem~2.25]{BH20}.
\end{definition}

From now on, $K$ denotes the real closed field of
Definition~\ref{def:puiseux}. For a finite index set $I$, let $K^{I}$ denote
the set of tuples $(y_i)_{i\in I}$ with entries in $K$. We write $K^{N}$ when
$I=\{1,\dots,N\}$. By the germ representation in
Lemma~\ref{lem:semialg-puiseux}, applied coordinatewise, every continuous
semialgebraic map
\[
 \varphi\colon[0,\varepsilon)\longrightarrow\mathbb R^{I}
\]
determines a point of $K^{I}$.

\begin{definition}[{\rm \defn{First-order formulas}}]\label{def:firstorder}
The \defn{language of ordered fields} has the symbols $+$, $\cdot$, $-$, $0$,
$1$, $=$, and $<$. An \defn{atomic formula} is an expression $p = q$ or
$p < q$, where $p$ and $q$ are polynomials in the variables with integer
coefficients. A \defn{first-order formula} is built from atomic formulas using
the connectives $\wedge$, $\vee$, and $\neg$ and the quantifiers $\forall x$
and $\exists x$, which range over elements of the field. Allowing real
coefficients yields formulas \defn{with parameters in $\mathbb{R}$}.
Quantifier-free formulas of this kind define precisely the semialgebraic sets
of Definition~\ref{def:semialg}. By the Tarski--Seidenberg theorem, quantifiers
can always be eliminated. Thus definable and semialgebraic sets coincide.
\end{definition}

\begin{remark}\label{rem:readings}
For a field $F$, we use a tuple in $F^{H(n,k)}$ in three equivalent ways: as
a \defn{function} $\nu$ on $H(n,k)$, as in Definition~\ref{def:Mconvex}, or
as an \defn{exponent vector} $\vg$, as in Definition~\ref{def:ratio}. Under the
linear isomorphism
\[
  F^{H(n,k)} \;\xrightarrow{\ \sim\ }\; F[x_1,\dots,x_k]_n,
  \qquad
  (P_{\va})_{\va} \;\longmapsto\; \sum_{\va \in H(n,k)} P_{\va}\,\vx^{\va}
\]
it may also be viewed as the \defn{coefficient vector} of a homogeneous
polynomial of degree $n$. This map is an isomorphism because the monomials
$\vx^{\va}$ with $\va \in H(n,k)$ form a basis of $F[x_1,\dots,x_k]_n$.
\end{remark}

The following specialization principle is stated in \cite[p.~860]{BH20}.

\begin{proposition}\label{prop:eventual-evaluation}
Let $\Phi(x_1,\ldots,x_m)$ be a first-order formula in the language of
ordered fields, and let $s_1(t),\ldots,s_m(t)\in K$. Then
\[
 \Phi(s_1(t),\ldots,s_m(t))\text{ holds in }K
\]
if and only if
\[
 \Phi(s_1(q),\ldots,s_m(q))\text{ holds in }\mathbb R
\]
for all sufficiently small positive real numbers $q$.
\end{proposition}

To apply Proposition~\ref{prop:eventual-evaluation} to Lorentzianity, we use
the following result of Br\'and\'en and Huh \cite{BH20}.

\begin{proposition}\label{prop:lor-firstorder}
For fixed integers $m\ge0$ and $k\ge1$, there is a first-order formula
\[
 \Lambda_{m,k}\bigl((c_{\va})_{\va\in H(m,k)}\bigr)
\]
in the language of ordered fields, without parameters, such that, for every
real closed field $F$ and every
$(c_{\va})_{\va\in H(m,k)}\in F^{H(m,k)}$, the formula
$\Lambda_{m,k}((c_{\va})_{\va\in H(m,k)})$ holds in $F$ if and only if
\[
 f=\sum_{\va\in H(m,k)}c_{\va}\vx^{\va}
\]
is Lorentzian over $F$ in the sense of Definition~\ref{def:lorK}.
\end{proposition}

Applied coordinatewise, Lemma~\ref{lem:semialg-puiseux} turns a continuous
semialgebraic curve into a point over $K$. Together with
Proposition~\ref{prop:eventual-evaluation}, this gives the following transfer
principle.

\begin{lemma}[Transfer]\label{lem:transfer}
Let $\Phi$ be a first-order formula in the language of ordered fields with
parameters in $\mathbb{R}$, and let
\[
 \varphi \colon [0,\varepsilon)\longrightarrow\mathbb{R}^{N}
\]
be continuous semialgebraic. Let $a\in K^N$ be the point determined by its
coordinate germs as in Lemma~\ref{lem:semialg-puiseux}. If
$\Phi(\varphi(t))$ holds for every $t\in(0,\varepsilon)$, then
$\Phi(a)$ holds in $K$.
\end{lemma}

\begin{proof}
Let $r_1,\ldots,r_\ell\in\mathbb R$ be the parameters occurring in $\Phi$.
The list may be empty. Write
\[
 \Phi(y)=\widetilde\Phi(y,r_1,\ldots,r_\ell),
\]
where $\widetilde\Phi(y,z_1,\ldots,z_\ell)$ is a formula without parameters,
and regard each $r_j$ as a constant Puiseux series in $K$.

Write $a=(a_1,\ldots,a_N)$. By
Lemma~\ref{lem:semialg-puiseux}, for each $i$ there is $\delta_i>0$ such that
\[
 a_i(q)=\varphi_i(q)
 \qquad(0<q<\delta_i).
\]
Since there are only finitely many coordinates, a single positive threshold
works for all of them. More explicitly, set
\[
 \delta=\min\{\varepsilon,\delta_1,\ldots,\delta_N\}>0.
\]
Then, for every $q\in(0,\delta)$,
\[
 \widetilde\Phi
 \bigl(a_1(q),\ldots,a_N(q),r_1,\ldots,r_\ell\bigr)
\]
holds in $\mathbb R$ by the hypothesis. Applying
Proposition~\ref{prop:eventual-evaluation} to
\[
 (a_1,\ldots,a_N,r_1,\ldots,r_\ell)\in K^{N+\ell}
\]
shows that
\[
 \widetilde\Phi(a_1,\ldots,a_N,r_1,\ldots,r_\ell)
\]
holds in $K$, which is precisely the assertion that $\Phi(a)$ holds in $K$.
\end{proof}

\begin{definition}[{\rm \defn{Tropicalization}}, {\cite[Definition~3.17]{BH20}}]
\label{def:tropK}
For
\[
 f=\sum_{\va\in H(n,k)}c_{\va}\vx^{\va}
 \in K_{\ge0}[x_1,\dots,x_k]_n,
\]
and $\va\in H(n,k)$, define
\[
 \operatorname{val}(c_{\va})=
 \begin{cases}
  q,
  &\text{if $c_{\va}\ne0$ has first nonzero term $a t^q$,
          with $a>0$ and $q\in\mathbb Q$},\\
  +\infty,
  &\text{if $c_{\va}=0$}.
 \end{cases}
\]
The \defn{tropicalization} of $f$ is
\[
 \trop(f)\colon H(n,k)\longrightarrow\mathbb Q\cup\{+\infty\},
 \qquad
 \trop(f)(\va)=\operatorname{val}(c_{\va}).
\]
\end{definition}

\subsection{Proof of the dual description}

\begin{lemma}\label{lem:Maffine}
Every affine function on $H(n,k)$ is M-convex. Its negative is also M-convex.
If $\nu$ is M-convex and $\ell$ is affine, then $\nu+\ell$ is M-convex.
\end{lemma}

\begin{proof}
Fix $\va \neq \vb\in H(n,k)$ and $i$ with $\alpha_i>\beta_i$.  Since
$|\va|_1=|\vb|_1$, there is at least one $j$ with $\alpha_j<\beta_j$.  For any such
$j$,
\[
 (\va-\ve_i+\ve_j)+(\vb+\ve_i-\ve_j)=\va+\vb.
\]
If $\ell(\vx)=c+\langle\bm b,\vx\rangle$, the sum
$\ell(\bm x)+\ell(\bm y)$ depends only on $\bm x+\bm y$.  Consequently
\[
 \ell(\va)+\ell(\vb)
 =\ell(\va-\ve_i+\ve_j)+\ell(\vb+\ve_i-\ve_j).
\]
Thus every exchange inequality for an affine function holds with equality.
Adding an affine function does not change any exchange inequality for
$\nu$.
\end{proof}

\begin{definition}[{\rm \defn{Rational polyhedral cones}}]\label{def:rational}
Let $E$ be a finite-dimensional real vector space. A \defn{rational structure}
on $E$ is a $\mathbb Q$-subspace $E_{\mathbb Q}$ such that
$E_{\mathbb Q}\otimes_{\mathbb Q}\mathbb R\cong E$. Its \defn{rational dual}
is
\[
 E_{\mathbb Q}^*=\{\lambda\in E^*:\lambda(E_{\mathbb Q})\subseteq\mathbb Q\}.
\]
A cone $C\subseteq E$ is \defn{rational polyhedral} if it has either of the
following equivalent forms:
\[
 C=\operatorname{cone}(v_1,\ldots,v_r),
 \qquad v_i\in E_{\mathbb Q},
\]
or
\[
 C=\{x\in E:\lambda_j(x)\ge0\text{ for }j=1,\ldots,s\},
 \qquad \lambda_j\in E_{\mathbb Q}^*.
\]
The equivalence of these descriptions is the rational
Farkas--Minkowski--Weyl theorem.

Finite Minkowski sums, rational linear images, and polars of rational
polyhedral cones are again rational polyhedral. Indeed,
if
\[
 C=\operatorname{cone}(v_1,\ldots,v_r),
 \qquad
 D=\operatorname{cone}(w_1,\ldots,w_s),
\]
then
\[
 C+D=\operatorname{cone}(v_1,\ldots,v_r,w_1,\ldots,w_s),
 \qquad
 T(C)=\operatorname{cone}(Tv_1,\ldots,Tv_r),
\]
for every rational linear map $T$, and
\[
 C^\vee=\{\lambda\in E^*:\lambda(v_i)\ge0\text{ for }i=1,\ldots,r\}.
\]
In particular, rational polyhedral cones are closed.

The quotient by
$\operatorname{Aff}(H(n,k))$ has the rational structure induced by the standard
rational structure on $\mathbb R^{H(n,k)}$ and the rational subspace
$\operatorname{Aff}(H(n,k))$.  A linear map between the spaces
$\mathbb{R}^{H(n,k)}$, their quotients, and their duals is \defn{rational} if its
matrix in rational bases has rational entries. By Lemma~\ref{lem:affine}, the
space $\operatorname{Aff}(H(n,k))$ is spanned by the integer vectors
$\va \mapsto \alpha_i$, so the quotient map
$\mathbb{R}^{H(n,k)} \to \mathbb{R}^{H(n,k)}/\operatorname{Aff}(H(n,k))$ is
rational.
\end{definition}

\begin{lemma}
\label{lem:Mpolyhedral}
$\mathcal{M}_{n,k}$ is a rational polyhedral cone. In particular, it is
closed and the closure in Notation~\ref{not:BRL} is redundant.  Every element
of $\mathcal{M}_{n,k}$ is a finite sum of classes of M-convex functions.
\end{lemma}

\begin{proof}
There are finitely many triples $(\va,\vb,i)$ with
$\va,\vb\in H(n,k)$ and $\alpha_i>\beta_i$.  For each such triple, let
\[
 J(\va,\vb,i)=\{j:\alpha_j<\beta_j\}.
\]
This set is nonempty.  Choose one element
$j(\va,\vb,i)\in J(\va,\vb,i)$ for every triple.  The functions satisfying
all the corresponding chosen exchange inequalities form a rational
polyhedral cone. A function is M-convex if and only if it belongs to at least one
of these cones.  Hence the set of M-convex functions is a finite union
\[
 C_1\cup\cdots\cup C_N
\]
of rational polyhedral cones.

The conical hull of this union is the Minkowski sum
$C_1+\cdots+C_N$. By Definition~\ref{def:rational}, a finite Minkowski sum of
rational polyhedral cones is again rational polyhedral and therefore closed.
Finally, the quotient map
\[
 \pi:\mathbb{R}^{H(n,k)}\longrightarrow
 \mathbb{R}^{H(n,k)}/\operatorname{Aff}(H(n,k))
\]
is rational by Lemma~\ref{lem:affine}, and the image of a rational
polyhedral cone under a rational linear map is rational polyhedral.  Thus
\[
 \mathcal M_{n,k}=\pi(C_1+\cdots+C_N)
\]
is rational polyhedral.  The description as a Minkowski sum also shows that
each of its elements is a finite sum of classes of M-convex functions.
\end{proof}

\begin{definition}[{\rm \defn{The polynomials $F_q^\nu$}}]\label{def:Fq}
For an M-convex function $\nu$ and $0 < q \le 1$, put
\[
  F^{\nu}_q \;=\; \sum_{\va \in H(n,k)}
  q^{\nu(\va)}\, \frac{\vx^{\va}}{\va!} .
\]
This polynomial has positive coefficients and is Lorentzian by
\cite[Theorem~3.14]{BH20}.
\end{definition}

\begin{theorem}[Necessary direction]\label{lem:duality-necessary}
If $\vg\in\BR_{\mathring L}(n,k)$, then $\vg\in V_{n,k}$ and
\[
  \langle\vg,\nu\rangle\ge0
\]
for every M-convex function $\nu:H(n,k)\to\mathbb R$.
\end{theorem}

\begin{proof}
Fix an M-convex function $\nu$. For each $0<q\le1$, the polynomial
\[
 F_q^{\nu}=\sum_{\va\in H(n,k)}q^{\nu(\va)}\frac{\vx^{\va}}{\va!}
\]
is Lorentzian by \cite[Theorem~3.14]{BH20}. By
Definition~\ref{def:strictly-lorentzian}, it is a
coefficientwise limit of strictly Lorentzian polynomials $f_1,f_2,\dots$.  Let
$C$ be a constant with $R_{\vg}\le C$ on $\mathring L(n,k)$. The coefficient
vector of each $f_m$ lies in $\mathring L(n,k)$ and
$R_{\vg}(f_m)\le C$. Since $\nu$ is real-valued, all coefficients of
$F_q^{\nu}$ are positive. Hence $R_{\vg}(F_q^{\nu})$ is defined, and
$R_{\vg}$ is continuous at its coefficient vector. Letting $m\to\infty$
therefore gives
$R_{\vg}(F_q^{\nu})\le C$.
We have
\[
 R_{\vg}(F_q^{\nu})
 =\left(\prod_{\va}(\va!)^{-\gamma_{\va}}\right)
 q^{\langle\vg,\nu\rangle}.
\]
If $\langle\vg,\nu\rangle<0$, the right-hand side tends to infinity as
$q\to0^+$, contradicting the uniform bound established above. Thus
$\langle\vg,\nu\rangle\ge0$.

By Lemma~\ref{lem:Maffine}, every affine function $\ell$ and its negative are
M-convex. Applying the preceding inequality to $\ell$ and $-\ell$ gives
$\langle\vg,\ell\rangle=0$ for every affine $\ell$.  In particular, it
vanishes on each coordinate function $\va\mapsto\alpha_i$, and hence
$\sum_{\va}\gamma_{\va}\va=0$.  Thus $\vg\in V_{n,k}$.
\end{proof}

For the reverse inclusion, we extract an M-convex valuation from an unbounded
ratio. As in the construction of \cite{SG25}, the divergence is realized along
a direction whose coordinates are exponents of a parameter $t>0$. In the
Lorentzian setting, this direction is encoded by a semialgebraic curve.

We begin by recalling the two semialgebraic results needed for this
construction.

\begin{lemma}[Curve selection lemma {\cite[Lemma~2.1]{LP22}}]
\label{lem:curve-selection}
Let $A\subseteq\mathbb R^N$ be a semialgebraic set, and let
$a\in\overline{A}\setminus A$. Then there exist $\varepsilon>0$ and a
continuous semialgebraic curve
\[
 \varphi\colon[0,\varepsilon)\longrightarrow\mathbb R^N
\]
such that $\varphi(0)=a$ and $\varphi(t)\in A$ for all
$t\in(0,\varepsilon)$.
\end{lemma}

\begin{lemma}[Growth dichotomy lemma {\cite[Lemma~2.2]{LP22}}]
\label{lem:growth-dichotomy}
Let $f\colon(0,\varepsilon)\to\mathbb R$ be a semialgebraic function with
$f(t)\ne0$ for all $t\in(0,\varepsilon)$. Then there exist constants
$a\ne0$ and $\alpha\in\mathbb Q$ such that
\[
 f(t)=at^\alpha+o(t^\alpha)
 \qquad\text{as }t\to0^+.
\]
\end{lemma}

\begin{lemma}[Semialgebraic escape curve]\label{lem:escape}
Let $\vg\in\mathbb Z^{H(n,k)}$.  If $R_{\vg}$ is unbounded on
$\mathring L(n,k)$, then there is a rational-valued M-convex function
$\nu:H(n,k)\to\mathbb Q$ such that
\[
  \langle\vg,\nu\rangle<0.
\]
\end{lemma}

\begin{proof}
For $i=1,\ldots,k$, put
\[
 s_i=\sum_{\va\in H(n,k)}\gamma_{\va}\alpha_i.
\]
If some $s_i$ is nonzero, then
\[
 \nu(\va)=-s_i\alpha_i
\]
is an integer-valued affine function and hence is M-convex by
Lemma~\ref{lem:Maffine}, while
\[
 \langle\vg,\nu\rangle
 =-s_i\sum_{\va}\gamma_{\va}\alpha_i
 =-s_i^2<0.
\]
We may therefore assume that $s_i=0$ for every $i$, that is,
$\vg\in V_{n,k}$.  Lemma~\ref{lem:balanced-affine} then gives
\[
 \sum_{\va\in H(n,k)}\gamma_{\va}=0.
\]
Consequently, positive scalar multiplication preserves both strict
Lorentzianity and the ratio
\[
 R_{\vg}(c\vP)=c^{\sum_{\va}\gamma_{\va}}R_{\vg}(\vP)
 =R_{\vg}(\vP)
 \qquad(c>0).
\]
The set $\mathring L(n,k)$ is semialgebraic. Positivity of the coefficients is
expressed by a finite collection of strict polynomial inequalities. The
condition that each of the finitely many Hessians
$\operatorname{Hess}(\partial^{\vb} f)$ has signature $(+,-,\ldots,-)$ is also
semialgebraic. By Sylvester's law of inertia, congruence classes of real
symmetric matrices are classified by their signature. Thus, with
$J=\operatorname{diag}(1,-1,\ldots,-1)$, a symmetric matrix $H$ has that
signature if and only if there is a matrix $A$ such that
\[
 (\det A)^{2}>0,\qquad A^{\mathsf T}HA=J.
\]
Here $(\det A)^{2}>0$ is used in place of the equivalent condition
$\det A\ne0$, so both displayed conditions are literally of the form admitted
by Definition~\ref{def:semialg}. They are polynomial conditions, and their
projection to the entries of $H$ is semialgebraic by Tarski--Seidenberg.

Let
\[
 \Delta=\left\{\vP\in\Rge^{H(n,k)}:
              \sum_{\va\in H(n,k)}P_{\va}=1\right\},
 \qquad
 \mathcal S=\mathring L(n,k)\cap\Delta.
\]
The set $\Delta$ is semialgebraic. Indeed, it is the finite union, over subsets
$Z\subseteq H(n,k)$, of the sets cut out by $P_{\va}=0$ for $\va\in Z$, by
$P_{\va}>0$ for $\va\notin Z$, and by $\sum_{\va}P_{\va}-1=0$. It is also
compact. By the scale invariance above, $R_{\vg}$ is
unbounded on $\mathcal S$. There is a sequence
$\vP^{(j)}\in\mathcal S$ such that $R_{\vg}(\vP^{(j)})\to\infty$.  Passing to a
subsequence, we may assume that $\vP^{(j)}\to\vP_*\in\Delta$.

Write
\[
 N(\vP)=\prod_{\gamma_{\va}>0}P_{\va}^{\gamma_{\va}},
 \qquad
 D(\vP)=\prod_{\gamma_{\va}<0}P_{\va}^{-\gamma_{\va}},
\]
where an empty product is $1$. Since $\vg$ is integral, $N$ and $D$ are
polynomials, and $R_{\vg}=N/D$ on $\mathcal S$. For all sufficiently large
$j$, the numbers
\[
 u_j=R_{\vg}(\vP^{(j)})^{-1}
\]
lie in $(0,1)$ and converge to zero. Define
\[
\mathcal Y=\left\{(u,\vP):0<u<1,\ \vP\in\mathcal S,\
                    uN(\vP)=D(\vP)\right\}
\]
Then $(u_j,\vP^{(j)})\in\mathcal Y$ for every such $j$. Moreover,
$\mathcal Y$ is semialgebraic. Indeed, the condition $0<u<1$ consists of two strict
polynomial inequalities, the condition $\vP\in\mathcal S$ defines the
semialgebraic set $\mathbb R\times\mathcal S$ in the variables $(u,\vP)$,
and $uN(\vP)-D(\vP)=0$ is a polynomial equation. Semialgebraic sets are
closed under finite intersections.
Therefore
\[
 (u_j,\vP^{(j)})\longrightarrow(0,\vP_*),
 \qquad
 (0,\vP_*)\in\overline{\mathcal Y}\setminus\mathcal Y.
\]
By Lemma~\ref{lem:curve-selection}, there are $\varepsilon>0$ and a continuous
semialgebraic curve
\[
 [0,\varepsilon)\longrightarrow\mathbb R\times\mathbb R^{H(n,k)},
 \qquad
 t\longmapsto\bigl(u(t),\vP(t)\bigr),
\]
such that $(u(0),\vP(0))=(0,\vP_*)$ and
$(u(t),\vP(t))\in\mathcal Y$ for every $t\in(0,\varepsilon)$.

The functions $u$ and $P_{\va}$ are semialgebraic and positive on
$(0,\varepsilon)$. Apply Lemma~\ref{lem:growth-dichotomy} to $u$ and to the
finitely many coordinate functions $P_{\va}$. After shrinking $\varepsilon$
if necessary, we obtain nonzero constants $c_0,c_{\va}$ and rational numbers
$r,\nu(\va)$ such that
\[
 u(t)=c_0t^r(1+o(1)),
 \qquad
 P_{\va}(t)=c_{\va}t^{\nu(\va)}(1+o(1))
 \quad(\va\in H(n,k)).
\]
Since the functions on the left are positive, their leading constants
$c_0,c_{\va}$ are positive. Moreover, since $u(t)\to0$, we have $r>0$.  The
exact identity
\[
 u(t)N(\vP(t))=D(\vP(t))
\]
holds for every $t\in(0,\varepsilon)$.  Comparing its leading powers of $t$
gives
\[
 r+\sum_{\gamma_{\va}>0}\gamma_{\va}\nu(\va)
 =\sum_{\gamma_{\va}<0}(-\gamma_{\va})\nu(\va),
\]
and therefore
\begin{equation}\label{eq:escape-negative}
 \langle\vg,\nu\rangle=-r<0.
\end{equation}

The growth dichotomy supplies the rational asymptotic exponents and the
negative pairing. We now apply the germ representation of
Lemma~\ref{lem:semialg-puiseux} coordinatewise to place the entire coefficient
curve over $K$, where Lemma~\ref{lem:transfer} applies.

For each $\va$, let $p_{\va}\in K$ be the unique convergent Puiseux series
representing the germ of $P_{\va}$ at $0^+$, as in
Lemma~\ref{lem:semialg-puiseux}.  The preceding expansion gives
\[
 \operatorname{val}p_{\va}=\nu(\va).
\]
Moreover, $p_{\va}>0$ in $K$ because $P_{\va}(t)>0$ for every $t>0$.
Since $H(n,k)$ is finite, all the $p_{\va}$ lie in a single field
$\mathbb R((t^{1/M}))_{\mathrm{conv}}\subseteq K$ for some
$M\in\mathbb Z_{>0}$.  Set
\[
  f_K=\sum_{\va\in H(n,k)}p_{\va}\vx^{\va}\in K[\vx].
\]
For every $t\in(0,\varepsilon)$, the coefficient tuple $\vP(t)$ belongs to
$\mathcal S$, so the corresponding polynomial is strictly Lorentzian and hence Lorentzian
over $\mathbb R$. It remains to verify that $f_K$ is Lorentzian over $K$.
Proposition~\ref{prop:lor-firstorder} expresses Lorentzianity by a first-order
formula, and Lemma~\ref{lem:transfer} transfers this formula from the real
coefficient tuples $\vP(t)$ to their coefficient-germ tuple. Thus, $f_K$ is
Lorentzian over $K$, and
Definition~\ref{def:tropK} gives
\[
 \trop(f_K)(\va)
 =\operatorname{val}p_{\va}=\nu(\va).
\]
Passing to normalized coefficients does not change this valuation because
$\operatorname{val}(\va!)=0$. Indeed, if
$\widehat p_{\va}:=\va!p_{\va}$ denotes the normalized coefficient, then
\[
 \operatorname{val}(\widehat p_{\va})
 =\operatorname{val}(\va!p_{\va})
 =\operatorname{val}(\va!)+\operatorname{val}(p_{\va})
 =\operatorname{val}(p_{\va}).
\]
The tropicalization theorem
\cite[Theorem~3.20]{BH20} now implies that $\nu$ is M-convex.
Together with \eqref{eq:escape-negative}, this completes the proof.
\end{proof}

\begin{lemma}\label{lem:BRcone}
$\BR_{\mathring L}(n,k)$ is a convex cone.  
\end{lemma}

\begin{proof}
For positive coefficient vectors,
\[
 R_{\sum_a c_a\vg_a}=\prod_a R_{\vg_a}^{c_a}.
\]
If $R_{\vg_a}\le C_a$ on $\mathring L(n,k)$, the right-hand side is at most
$\prod_a C_a^{c_a}$.  This proves the assertion.
\end{proof}

We can now complete the proof of the main structural theorem.

\begin{theoremA}[Main structural theorem]\label{thm:BRL-duality}
One has
\[
  \BR_{\mathring{L}}(n,k) \;=\; \mathcal{M}_{n,k}^{\vee}
  \;=\; \bigl\{ \vg \in V_{n,k} \;:\;
  \langle \vg, \nu \rangle \ge 0 \ \text{for every M-convex } \nu \bigr\} .
\]
\end{theoremA}

\begin{proof}
By Lemma~\ref{lem:quotientdual}, the dual of
$\mathbb R^{H(n,k)}/\operatorname{Aff}(H(n,k))$ is canonically identified with
$V_{n,k}$.  Since $\mathcal M_{n,k}$ is the closed conical hull of the classes
of M-convex functions, its polar is therefore
\[
 \mathcal M_{n,k}^{\vee}
 =\bigl\{\vg\in V_{n,k}:\langle\vg,\nu\rangle\ge0
       \text{ for every M-convex function }\nu\bigr\}.
\]
Here the pairing depends only on the class of $\nu$ by
Lemma~\ref{lem:balanced-affine}.  This proves the second equality.  We prove
the first equality by establishing its two inclusions.

\emph{First inclusion.}  Let
$\vg\in\BR_{\mathring L}(n,k)$.  The necessary direction,
Theorem~\ref{lem:duality-necessary}, states that $\vg\in V_{n,k}$ and
$\langle\vg,\nu\rangle\ge0$ for every M-convex function $\nu$.  Hence the
displayed description of the polar gives
$\vg\in\mathcal M_{n,k}^{\vee}$.  Thus
\[
 \BR_{\mathring L}(n,k)\subseteq\mathcal M_{n,k}^{\vee}.
\]

\emph{Reverse inclusion.}  First suppose that
$\vg\in\mathcal M_{n,k}^\vee\cap\mathbb Z^{H(n,k)}$.  If $R_{\vg}$ were
unbounded on $\mathring L(n,k)$, the semialgebraic escape curve lemma,
Lemma~\ref{lem:escape}, would produce a rational-valued M-convex function $\nu$
such that $\langle\vg,\nu\rangle<0$.  This contradicts
$\vg\in\mathcal M_{n,k}^\vee$.  Therefore
\[
 \mathcal M_{n,k}^{\vee}\cap\mathbb Z^{H(n,k)}
 \subseteq\BR_{\mathring L}(n,k).
\]

By Lemma~\ref{lem:Mpolyhedral}, $\mathcal M_{n,k}$ is rational polyhedral.
Definition~\ref{def:rational} therefore shows that its polar is rational
polyhedral as well.  After clearing denominators, there are integral vectors
$\vg_1,\ldots,\vg_r\in\mathcal M_{n,k}^{\vee}$ such that
\[
 \mathcal M_{n,k}^{\vee}=\operatorname{cone}(\vg_1,\ldots,\vg_r).
\]
The preceding paragraph gives
$\vg_1,\ldots,\vg_r\in\BR_{\mathring L}(n,k)$, and
Lemma~\ref{lem:BRcone} shows that $\BR_{\mathring L}(n,k)$ is a convex cone.
Thus $\mathcal M_{n,k}^{\vee}\subseteq\BR_{\mathring L}(n,k)$, completing
the proof.
\end{proof}

\section{The cone of bounded ratios and its facet consequences}\label{sec:facets}

We next derive the facet consequences of Theorem~\ref{thm:BRL-duality}.
Recall from Definition~\ref{def:strictly-lorentzian} that we write
$\partial_i=\partial/\partial x_i$, so that
$\partial_i\vx^{\va}=\alpha_i\vx^{\va-\ve_i}$. Our first ingredient is the
following Bernstein--B\'ezier derivative formula.

\begin{lemma}\label{lem:bernstein}
For a function $a:H(n,k)\to\mathbb R$, define its \defn{Bernstein polynomial} by
\[
 \mathcal B_a(\vx)=
 \sum_{\va\in H(n,k)}\binom{n}{\va} a(\va)\vx^{\va}.
\]
Then, for all $i,j$,
\begin{equation}\label{eq:bernstein}
 (\partial_i-\partial_j)^2\mathcal B_a
 =
 n(n-1)\sum_{\vb\in H(n-2,k)}
 \binom{n-2}{\vb}\,d^{\vb}_{ij}(a)\,\vx^{\vb} .
\end{equation}
\end{lemma}

\begin{proof}
Both sides of \eqref{eq:bernstein} are homogeneous polynomials of degree
$n-2$. It therefore suffices to compare the coefficient of
$\vx^{\vb}$ for an arbitrary
$\vb=(\beta_1,\ldots,\beta_k)\in H(n-2,k)$.

First suppose that $i=j$.  Then $\partial_i-\partial_j=0$, so the left-hand
side vanishes.  The definition of the local discrete Hessian gives
\[
 d_{ii}^{\vb}(a)
 =a(\vb+2\ve_i)+a(\vb+2\ve_i)-2a(\vb+2\ve_i)
 =0,
\]
so the right-hand side vanishes as well.  It remains to consider
$i\ne j$.  Since $|\vb|_1=n-2$, the three multi-indices
$\vb+2\ve_i$, $\vb+\ve_i+\ve_j$, and $\vb+2\ve_j$ all belong to
$H(n,k)$.  We use
\[
 \binom n{\va}=\frac{n!}{\va!},
 \qquad
 \vb!=\prod_{r=1}^k\beta_r!.
\]

For a polynomial $G$, write $[\vx^{\vb}]\,G$ for the coefficient of
$\vx^{\vb}$ in $G$.  Only the term of $\mathcal B_a$ indexed by
$\vb+2\ve_i$ can produce the monomial $\vx^{\vb}$ after applying
$\partial_i^2$, and
\[
 \partial_i^2\vx^{\vb+2\ve_i}
 =(\beta_i+2)(\beta_i+1)\vx^{\vb}.
\]
Therefore
\begin{align*}
 [\vx^{\vb}]\,\partial_i^2\mathcal B_a
 &=
 \binom n{\vb+2\ve_i}
 (\beta_i+2)(\beta_i+1)a(\vb+2\ve_i)\\
 &=
 \frac{n!}
 {(\beta_i+2)!\prod_{r\ne i}\beta_r!}
 (\beta_i+2)(\beta_i+1)a(\vb+2\ve_i)\\
 &=
 \frac{n!}{\vb!}\,a(\vb+2\ve_i).
\end{align*}
In the last equality, the factors produced by differentiation cancel the
corresponding factors in $(\beta_i+2)!$. Interchanging $i$ and $j$ gives
\[
 [\vx^{\vb}]\,\partial_j^2\mathcal B_a
 =\frac{n!}{\vb!}\,a(\vb+2\ve_j).
\]

For the mixed derivative, the unique contributing term is indexed by
$\vb+\ve_i+\ve_j$.  Since $i\ne j$,
\[
 \partial_i\partial_j\vx^{\vb+\ve_i+\ve_j}
 =(\beta_i+1)(\beta_j+1)\vx^{\vb}.
\]
Hence
\begin{align*}
 [\vx^{\vb}]\,\partial_i\partial_j\mathcal B_a
 &=
 \binom n{\vb+\ve_i+\ve_j}
 (\beta_i+1)(\beta_j+1)
 a(\vb+\ve_i+\ve_j)\\
 &=
 \frac{n!}
 {(\beta_i+1)!(\beta_j+1)!
  \prod_{r\ne i,j}\beta_r!}
 (\beta_i+1)(\beta_j+1)
 a(\vb+\ve_i+\ve_j)\\
 &=
 \frac{n!}{\vb!}\,a(\vb+\ve_i+\ve_j).
\end{align*}

\par\noindent
Again, the factors produced by differentiation cancel the factors
$(\beta_i+1)$ and $(\beta_j+1)$ in the denominator.

Finally, using
$(\partial_i-\partial_j)^2
=\partial_i^2-2\partial_i\partial_j+\partial_j^2$
and combining the three coefficient computations gives
\begin{align*}
 [\vx^{\vb}]\,(\partial_i-\partial_j)^2\mathcal B_a
 &=
 \frac{n!}{\vb!}
 \Bigl(
 a(\vb+2\ve_i)+a(\vb+2\ve_j)
 -2a(\vb+\ve_i+\ve_j)
 \Bigr)\\
 &=
 \frac{n!}{\vb!}\,d_{ij}^{\vb}(a).
\end{align*}
On the other hand, the coefficient of $\vx^{\vb}$ on the right-hand side of
\eqref{eq:bernstein} is
\[
 n(n-1)\binom{n-2}{\vb}d_{ij}^{\vb}(a)
 =
 n(n-1)\frac{(n-2)!}{\vb!}d_{ij}^{\vb}(a)
 =
 \frac{n!}{\vb!}d_{ij}^{\vb}(a).
\]
Thus the coefficients agree for every $\vb\in H(n-2,k)$, and the two
degree-$(n-2)$ polynomials are equal.
\end{proof}

\begin{lemma}\label{lem:hessiankernel}
For $a\colon H(n,k)\to\mathbb R$, the following are equivalent:
\begin{enumerate}
 \item[(i)] $d^{\vb}(a)=0$ for every $\vb\in H(n-2,k)$,
 \item[(ii)] $a\in\operatorname{Aff}(H(n,k))$.
\end{enumerate}
\end{lemma}

\begin{proof}
If $a\in\operatorname{Aff}(H(n,k))$, then $d^{\vb}(a)=0$ for every
$\vb\in H(n-2,k)$ by Lemma~\ref{lem:affine-hessian}.  We prove the converse.

Assume that $d^{\vb}(a)=0$ for every $\vb\in H(n-2,k)$.  For a constant
vector $z\in\mathbb R^k$, write
$D_z=\sum_{r=1}^k z_r\partial_r$ for the directional derivative in the
direction $z$. By Lemma~\ref{lem:bernstein}, for every
$p,q\in\{1,\ldots,k\}$ we have the polynomial identity
\[
 D_{\ve_p-\ve_q}^{2}\mathcal B_a
 =(\partial_p-\partial_q)^2\mathcal B_a=0.
\]

We first show that the Hessian of $\mathcal B_a$ vanishes in all directions
tangent to the hyperplanes $\sum_r x_r=\text{constant}$.  For
$i=1,\ldots,k-1$, put $v_i=\ve_i-\ve_k$.  Then
\[
 D_{v_i}=\partial_i-\partial_k,
 \qquad
 D_{v_i-v_j}=\partial_i-\partial_j.
\]
Hence the preceding identities give
\[
 D_{v_i}^{2}\mathcal B_a
 =D_{v_j}^{2}\mathcal B_a
 =D_{v_i-v_j}^{2}\mathcal B_a=0.
\]
Since constant-coefficient directional derivatives commute, polarization
gives
\[
 2D_{v_i}D_{v_j}
 =D_{v_i}^{2}+D_{v_j}^{2}-D_{v_i-v_j}^{2}.
\]
It follows that
\[
 D_{v_i}D_{v_j}\mathcal B_a=0
 \qquad(1\le i,j\le k-1).
\]
The vectors $v_1,\ldots,v_{k-1}$ form a basis of
\[
 T=\left\{z\in\mathbb R^k:\sum_{r=1}^k z_r=0\right\}.
\]
Indeed, if $z\in T$, then
$z=\sum_{i=1}^{k-1}z_i(\ve_i-\ve_k)$.  Since $D_uD_v\mathcal B_a$ is
bilinear in the two direction vectors $u$ and $v$, the preceding identities
therefore imply
\[
 D_uD_v\mathcal B_a=0\qquad\text{for every }u,v\in T.
\]
Equivalently, at every point $\vx\in\mathbb R^k$ and for every $u,v\in T$,
\[
 u^{\mathsf T}\operatorname{Hess}(\mathcal B_a)(\vx)v=0.
\]

Now consider the affine hyperplane
\[
 \Pi=\left\{\vx\in\mathbb R^k:x_1+\cdots+x_k=1\right\}
\]
and parametrize it by
\[
 \Phi(y_1,\ldots,y_{k-1})
 =\left(y_1,\ldots,y_{k-1},1-y_1-\cdots-y_{k-1}\right).
\]
Set $g=\mathcal B_a\circ\Phi$.  Since
$\partial\Phi/\partial y_i=v_i$, the chain rule gives
\[
 \frac{\partial^2g}{\partial y_i\partial y_j}(\bm{y})
 =D_{v_i}D_{v_j}\mathcal B_a\bigl(\Phi(\bm{y})\bigr)=0
 \qquad(1\le i,j\le k-1).
\]
Every second partial derivative of the polynomial $g$ therefore vanishes,
so $g$ has degree at most one.  Consequently, there are
$c,\lambda_1,\ldots,\lambda_{k-1}\in\mathbb R$ such that
\[
 g(\bm{y})=c+\sum_{i=1}^{k-1}\lambda_i y_i.
\]
Define the linear form
\[
 L(\vx)=c\sum_{r=1}^k x_r+\sum_{i=1}^{k-1}\lambda_i x_i
       =\sum_{r=1}^k b_rx_r.
\]
On $\Pi$ we have $\sum_r x_r=1$ and $y_i=x_i$, so
$\mathcal B_a|_{\Pi}=L|_{\Pi}$.

It remains to recover $\mathcal B_a$ away from $\Pi$.  Put
$s=x_1+\cdots+x_k$.  If $s\ne0$, then $\vx/s\in\Pi$.  Since
$\mathcal B_a$ is homogeneous of degree $n$ and $L$ is homogeneous of degree
one,
\[
 \mathcal B_a(\vx)
 =s^n\mathcal B_a(\vx/s)
 =s^nL(\vx/s)
 =s^{n-1}L(\vx).
\]
The set $\{\vx:s\ne0\}$ is dense in $\mathbb R^k$. Since both sides are
polynomials, the equality therefore extends to $s=0$. Consequently,
\[
 \mathcal B_a(\vx)
 =(x_1+\cdots+x_k)^{n-1}\sum_{r=1}^k b_rx_r.
\]

Finally, fix $\va=(\alpha_1,\ldots,\alpha_k)\in H(n,k)$.  By the definition
of the Bernstein polynomial,
\[
 [\vx^{\va}]\mathcal B_a=\binom n\va a(\va)
 =\frac{n!}{\va!}a(\va).
\]
On the other hand, the term $b_ix_i$ can contribute to the coefficient of
$\vx^{\va}$ only when $\alpha_i>0$, and in that case its contribution is
$b_i\binom{n-1}{\va-\ve_i}$.  Hence
\begin{align*}
 [\vx^{\va}]\,(x_1+\cdots+x_k)^{n-1}L(\vx)
 &=\sum_{i:\,\alpha_i>0}b_i\frac{(n-1)!}{(\va-\ve_i)!}\\
 &=\frac{(n-1)!}{\va!}\sum_{i=1}^k b_i\alpha_i,
\end{align*}
where $(\va-\ve_i)!=\va!/\alpha_i$ when $\alpha_i>0$.  Comparing the two
coefficients and cancelling $(n-1)!/\va!$ gives
\[
 a(\va)=\frac1n\sum_{i=1}^k b_i\alpha_i.
\]
Thus $a$ is the restriction to $H(n,k)$ of a linear, and hence affine,
function on $\mathbb R^k$.  Therefore
$a\in\operatorname{Aff}(H(n,k))$.
\end{proof}

\begin{lemma}\label{lem:Mpointed}
The cone $\mathcal M_{n,k}$ is pointed.
\end{lemma}

\begin{proof}
We first prove an elementary consequence of the M-convex exchange property.  If
$\nu$ is M-convex, $\vb\in H(n-2,k)$, and $i\ne j$, we apply the axiom to
$\vb+2\ve_i$ and $\vb+2\ve_j$.  The only possible exchange partner for $i$ is
$j$, so
\[
 \nu(\vb+2\ve_i)+\nu(\vb+2\ve_j)
 \ge2\nu(\vb+\ve_i+\ve_j).
\]
Thus every entry of every $d^{\vb}(\nu)$ is nonnegative.  The same is true for
any nonnegative sum of M-convex functions.

Suppose $[f]$ and $-[f]$ both lie in $\mathcal M_{n,k}$.  By
Lemma~\ref{lem:Mpolyhedral}, choose representatives $f$ and $g$, each a finite
sum of M-convex functions, for these two classes.  Since $[g]=-[f]$, there is
an affine function $\ell$ such that $g=-f+\ell$.  Local Hessians annihilate
affine functions, so
\[
 d^{\vb}(g)=-d^{\vb}(f)\qquad\text{for every }\vb.
\]
Since both $d^{\vb}(g)$ and $d^{\vb}(f)$ are entrywise nonnegative, the
preceding identity forces $d^{\vb}(f)=0$ for every $\vb$.
Lemma~\ref{lem:hessiankernel} now implies that
$f$ is affine and $[f]=0$.  Therefore
$\mathcal M_{n,k}\cap(-\mathcal M_{n,k})=\{0\}$.
\end{proof}

For $k\ge4$, M-convex functions are not closed under addition. Accordingly,
although every element of $\mathcal M_{n,k}$ is a finite sum of classes of
M-convex functions by Lemma~\ref{lem:Mpolyhedral}, it need not be represented
by a single M-convex function. On an extreme ray, however, one can always
choose such a representative.

\begin{lemma}
\label{lem:Mextremerep}
Every extreme ray of $\mathcal M_{n,k}$ contains the class of an M-convex
function.
\end{lemma}

\begin{proof}
In the notation of the proof of Lemma~\ref{lem:Mpolyhedral},
\[
 \mathcal M_{n,k}=\pi(C_1)+\cdots+\pi(C_N),
\]
where each $C_a$ consists entirely of M-convex functions. Let
$\mathbb R_{\ge0}x$ be an extreme ray and choose a nonzero $x$ on it.  Write
$x=x_1+\cdots+x_N$ with $x_a\in\pi(C_a)$.  Since every $x_a$ and
$x-x_a$ lie in $\mathcal M_{n,k}$, extremality forces each nonzero $x_a$ to
belong to $\mathbb R_{>0}x$. At least one $x_a$ is nonzero. Choosing such an
$x_a$ yields the class of an M-convex function on the ray.
\end{proof}

\begin{corollary}\label{cor:facets}
The following statements hold:
\begin{enumerate}
 \item Every supporting halfspace of $\BR_{\mathring{L}}(n,k)$ is of the form
 $\langle \vg,\nu\rangle\ge0$ for some
 $[\nu]\in\mathcal M_{n,k}$.
 \item For a nonzero class $[\nu]\in\mathcal M_{n,k}$, the hyperplane
 $\langle \vg,\nu\rangle=0$ defines a facet of
 $\BR_{\mathring{L}}(n,k)$ if and only if
 $\mathbb R_{\ge0}[\nu]$ is an extreme ray of $\mathcal M_{n,k}$.
 \item Every extreme ray of $\mathcal M_{n,k}$ has an M-convex
 representative.  Equivalently, modulo affine functions and positive
 scaling, the facet normals are precisely the \defn{convexly indecomposable}
 M-convex functions.
\end{enumerate}
\end{corollary}

\begin{proof}
\emph{For (i).}
Under the paired-space identification of
Lemma~\ref{lem:quotientdual}, Theorem~\ref{thm:BRL-duality} and the dual
theorem give
\[
 \BR_{\mathring L}(n,k)^\vee
 =\mathcal M_{n,k}^{\vee\vee}=\mathcal M_{n,k},
\]
because $\mathcal M_{n,k}$ is closed by Lemma~\ref{lem:Mpolyhedral}.  Thus
every linear functional defining a supporting halfspace of
$\BR_{\mathring L}(n,k)$ is represented by a class
$[\nu]\in\mathcal M_{n,k}$, which proves (i).

\par\medskip
\noindent\emph{For (ii).}
By Lemma~\ref{lem:Mpointed}, the standard face-duality correspondence applies.
Facets of $\mathcal M_{n,k}^\vee$ correspond exactly to extreme rays of
$\mathcal M_{n,k}$.  Under the paired-space identification, the face
corresponding to $\mathbb R_{\ge0}[\nu]$ is cut out by
$\langle\vg,\nu\rangle=0$.  This proves (ii).

\par\medskip
\noindent\emph{For (iii).}
Lemma~\ref{lem:Mextremerep} supplies an M-convex representative of every
extreme ray of $\mathcal M_{n,k}$.  Finally, extremality of
$\mathbb R_{\ge0}[\nu]$ says precisely that in any decomposition
$[\nu]=\sum_a[\nu_a]$ into M-convex classes, every nonzero summand is a
positive multiple of $[\nu]$.  Thus the M-convex representative is convexly
indecomposable exactly when its class spans an extreme ray.  Together with
(ii), this proves (iii).
\end{proof}

We now translate M-convexity---and hence the facet normals from
Corollary~\ref{cor:facets}---into the language of local tree metrics.

\begin{definition}[{\rm \defn{Tree metrics}}]\label{def:tree}
A symmetric function $d \colon \{1,\dots,k\}^{2} \to \mathbb{R}$ with
$d_{ii} = 0$ is called a
\defn{metric} if $d_{ij} \ge 0$ and $d_{ij} \le d_{il} + d_{lj}$ for all
$i,j,l$. We allow distinct indices to lie at distance zero. Thus, in the usual
terminology, this is a pseudometric. It is a \defn{tree metric} if there is a
tree with nonnegative edge lengths, together with a map from
$\{1,\dots,k\}$ to its vertices, such that $d_{ij}$ is the length of the path
joining the images of $i$ and $j$. By the four-point condition of Buneman
\cite{Bun74}, a metric
$d$ is a tree metric if and only if
\[
  d_{ij} + d_{lm} \;\le\; \max\{ d_{il} + d_{jm},\ d_{im} + d_{jl} \}
  \qquad (i,j,l,m \in \{1,\dots,k\}).
\]
In particular, the zero function is a tree metric.
\end{definition}

\begin{lemma}\label{lem:localslices}
Let $\nu\colon H(n,k)\to\mathbb R$.  For $\vb\in H(n-2,k)$, the
\defn{quadratic slice} of $\nu$ at $\vb$ is
\[
 \nu_{\vb}:H(2,k)\longrightarrow\mathbb R,
 \qquad \nu_{\vb}(\va)=\nu(\vb+\va).
\]
Then $\nu$ is M-convex on $H(n,k)$ if and only if $\nu_{\vb}$ is
M-convex on $H(2,k)$ for all $\vb\in H(n-2,k)$.
\end{lemma}

\begin{proof}
The full simplex $H(n,k)$ is an M-convex set: if $\alpha_i>\alpha'_i$, equality
of the coordinate sums supplies a $j$ with $\alpha_j<\alpha'_j$, and the two
exchanged points remain in $H(n,k)$. The local exchange theorem, stated in
\cite[Theorem~6]{MT03}, gives the following criterion. On an M-convex domain,
the exchange axiom in Definition~\ref{def:Mconvex} is equivalent to the
apparently weaker condition that, for every $\va,\va'$ with
$|\va-\va'|_1=4$, there exist
\[
 i\in\operatorname{supp}^+(\va-\va'),
 \qquad
 j\in\operatorname{supp}^-(\va-\va')
\]
such that the exchange inequality holds.  Here
\[
 \operatorname{supp}^+(\vv)=\{i:v_i>0\},
 \qquad
 \operatorname{supp}^-(\vv)=\{i:v_i<0\}
\]
are the \defn{positive support} and \defn{negative support}, respectively, of
$\vv\in\mathbb Z^{k}$. Thus,
$\operatorname{supp}^+(\va-\va')$ is the set of indices $i$ with
$\alpha_i>\alpha'_i$ and $\operatorname{supp}^-(\va-\va')$ the set of indices
$j$ with $\alpha_j<\alpha'_j$, exactly as in
Definition~\ref{def:Mconvex}. On the full simplex the condition can also be
written
\[
 \nu(\va)+\nu(\va')
 \ge
 \min_{\substack{i\in\operatorname{supp}^+(\va-\va')\\
                 j\in\operatorname{supp}^-(\va-\va')}}
 \bigl\{
 \nu(\va-\ve_i+\ve_j)+\nu(\va'+\ve_i-\ve_j)
 \bigr\}.
\]
If $\va,\va'\in H(n,k)$ have distance $4$, then their coordinatewise minimum
$\vb=\va\wedge\va'$ has degree $n-2$.  Consequently
\[
 \va=\vb+\mu,\qquad \va'=\vb+\mu'
 \quad\text{for some }\mu,\mu'\in H(2,k),
\]
and the local exchange inequality for $\va,\va'$ is exactly the corresponding
exchange inequality for $\nu_{\vb}$. Conversely, every nontrivial exchange
within a translated slice $\vb+H(2,k)$ arises in this way. The cited local
criterion therefore proves the equivalence.
\end{proof}

The following lemma and Proposition~\ref{prop:treemetric} specialize the
quadratic and local tree-metric characterization of Hirai and Murota
\cite{HM04}. We include the details needed here.

\begin{lemma}\label{lem:quadratic-tree}
Let $w:H(2,k)\to\mathbb R$, write
$w_{ij}=w(\ve_i+\ve_j)$ (also when $i=j$), and put
\[
 d_{ij}=w_{ii}+w_{jj}-2w_{ij}.
\]
Then $w$ is M-convex if and only if $d=(d_{ij})$ is a tree metric.
\end{lemma}

\begin{proof}
\emph{The possible exchanges.}
Every point of $H(2,k)$ is either $2\ve_i$ or $\ve_i+\ve_j$ with $i\ne j$.
For $\va,\vb\in H(2,k)$, the distance $|\va-\vb|_1$ is therefore $0$, $2$,
or $4$.  Indeed, because $\va$ and $\vb$ have the same coordinate sum, the
total surplus of $\va$ over $\vb$ equals the total deficit, and the
distance $|\va-\vb|_1$ is twice this common value.  In degree two, that value is
$0$, $1$, or $2$.  At distance $0$ the two points are equal and there is
nothing to check.  At distance $2$ there are unique distinct indices $p,q$
such that
$\va-\vb=\ve_p-\ve_q$.  The exchange in Definition~\ref{def:Mconvex}
replaces $\va$ by $\vb$ and $\vb$ by $\va$, so its inequality holds with
equality.

It remains to consider pairs at distance $4$. Here the total surplus is two,
which is equivalent to the two points having disjoint supports. Up to
interchanging them, there are exactly three possibilities.

First, consider $2\ve_i$ and $2\ve_j$, where $i\ne j$.  The only possible
exchange produces two copies of $\ve_i+\ve_j$, so the required inequality is
\begin{equation}\label{eq:qtree-double}
 w_{ii}+w_{jj}\ge2w_{ij}.
\end{equation}

Second, consider $2\ve_i$ and $\ve_j+\ve_l$, where $i,j,l$ are distinct.
For this ordering, $i$ is the only coordinate in which the first point
exceeds the second.  Its exchange partner can be either $j$ or $l$, but both
choices produce the same two points $\ve_i+\ve_j$ and $\ve_i+\ve_l$.
Consequently, the required inequality is
\begin{equation}\label{eq:qtree-triangle}
 w_{ii}+w_{jl}\ge w_{ij}+w_{il}.
\end{equation}
In the reverse ordering, either prescribed surplus coordinate, $j$ or $l$,
has the unique exchange partner $i$, again yielding
\eqref{eq:qtree-triangle}. Thus this inequality verifies the full exchange
condition for the pair in either ordering.

Third, consider $\ve_i+\ve_j$ and $\ve_l+\ve_m$, where $i,j,l,m$ are
pairwise distinct.  If the prescribed surplus coordinate is $i$, exchanging
it with $l$ or $m$ gives, respectively, the two sums
\[
 w_{jl}+w_{im}
 \qquad\text{and}\qquad
 w_{jm}+w_{il}.
\]
If the prescribed surplus coordinate is $j$, the two choices give the same
two sums in the opposite order.  The same is true after interchanging the two
original points.  Hence an admissible exchange exists for every prescribed
surplus coordinate if and only if
\begin{equation}\label{eq:qtree-fourpoint}
 w_{ij}+w_{lm}
 \ge\min\{w_{il}+w_{jm},\,w_{im}+w_{jl}\}.
\end{equation}
We have now exhausted every exchange that can impose a nontrivial condition
on $w$.

\emph{Translation to conditions on $d$.}
Because $w_{ij}=w_{ji}$, the matrix $d$ is symmetric, and its definition gives
$d_{ii}=0$.  Inequality~\eqref{eq:qtree-double} is exactly
\[
 d_{ij}=w_{ii}+w_{jj}-2w_{ij}\ge0.
\]
For three distinct indices $i,j,l$, direct substitution gives
\[
 d_{ij}+d_{il}-d_{jl}
 =2\bigl(w_{ii}+w_{jl}-w_{ij}-w_{il}\bigr).
\]
Thus \eqref{eq:qtree-triangle} is equivalent to the triangle inequality
\[
 d_{jl}\le d_{ij}+d_{il}.
\]

For four distinct indices $i,j,l,m$, put
\[
 C=w_{ii}+w_{jj}+w_{ll}+w_{mm}.
\]
Then
\begin{align*}
 d_{ij}+d_{lm}&=C-2(w_{ij}+w_{lm}),\\
 d_{il}+d_{jm}&=C-2(w_{il}+w_{jm}),\\
 d_{im}+d_{jl}&=C-2(w_{im}+w_{jl}).
\end{align*}
Because each expression is obtained by subtracting twice one of the three
sums from the same quantity $C$, their order is reversed. Therefore
\eqref{eq:qtree-fourpoint} is equivalent to
\[
 d_{ij}+d_{lm}
 \le\max\{d_{il}+d_{jm},\,d_{im}+d_{jl}\},
\]
which is the four-point inequality.

The preceding calculations used distinct indices, but the repeated-index
cases introduce no additional conditions.  A triangle inequality with a
repeated index is either an equality or has the form $0\le2d_{ij}$.  In the
four-point inequality, if $i=j$, it reduces to
$d_{lm}\le d_{il}+d_{im}$, and similarly if $l=m$.  If an index in the first
pair equals an index in the second pair, one of the two expressions inside
the maximum is exactly the left-hand side.  Thus nonnegativity and the
triangle inequalities above imply every repeated-index instance of the
four-point condition.

\emph{Conclusion.}
If $w$ is M-convex, the three exchange inequalities above show that $d$ is a
metric satisfying the four-point condition.  Hence $d$ is a tree metric by
Definition~\ref{def:tree}.  Conversely, suppose that $d$ is a tree metric.
Its nonnegativity, triangle inequalities, and four-point condition imply
\eqref{eq:qtree-double}, \eqref{eq:qtree-triangle}, and
\eqref{eq:qtree-fourpoint}, respectively, by reversing the preceding
calculations. Exchanges at distance $0$ or $2$ are automatic, and the three
distance-$4$ cases exhaust all remaining pairs.  Therefore the exchange axiom
holds for every prescribed surplus coordinate, and $w$ is M-convex.
\end{proof}

\begin{proposition}[Local tree metrics]\label{prop:treemetric}
With the local discrete Hessians $d^{\vb}(\nu)$ of
Definition~\ref{def:local-hessian-field}, a
function $\nu:H(n,k)\to\mathbb R$ is M-convex if and only if
$d^{\vb}(\nu)$ is a tree metric for every $\vb\in H(n-2,k)$.  Consequently,
the facet normals of $\BR_{\mathring L}(n,k)$ are precisely the convexly
indecomposable classes $[\nu]$ whose local Hessians are all tree metrics.
\end{proposition}

\begin{proof}
For the quadratic slice $w=\nu_{\vb}$ of Lemma~\ref{lem:localslices}, one has
\[
 w_{ij}=\nu(\vb+\ve_i+\ve_j)
 \quad\text{and}\quad
 w_{ii}+w_{jj}-2w_{ij}=d^{\vb}_{ij}(\nu).
\]
Lemma~\ref{lem:quadratic-tree} therefore identifies M-convexity of this slice
with the tree-metric condition on $d^{\vb}(\nu)$.  Applying
Lemma~\ref{lem:localslices} to all $\vb$ proves the first assertion. The
second is Corollary~\ref{cor:facets}.
\end{proof}

\begin{remark}\label{rem:recover-HHSW}
For $n=2$, the index set $H(n-2,k)=H(0,k)$ is a single point. Hence
Proposition~\ref{prop:treemetric} reduces to Lemma~\ref{lem:quadratic-tree}.
By Lemma~\ref{lem:affine-hessian}, the map
\[
 \mathbb R^{H(2,k)}/\operatorname{Aff}(H(2,k))
 \longrightarrow \mathbb R^{\binom{k}{2}},
 \qquad
 [\nu]\longmapsto
 \bigl(d^{\bm0}_{ij}(\nu)\bigr)_{1\le i<j\le k}
 =\bigl(\nu(2\ve_i)+\nu(2\ve_j)
       -2\nu(\ve_i+\ve_j)\bigr)_{i<j},
\]
is well defined on the affine quotient and identifies $\mathcal M_{2,k}$ with
the cone generated by the tree metrics on $k$ points. This is the \defn{cut
cone} $\mathrm{Cut}_k$, since every split metric $\delta_S$ is a tree metric
and every tree metric is a nonnegative sum of split metrics.

To describe the corresponding dual coordinates, let
$a_{ij}=\gamma_{\ve_i+\ve_j}$ for $i<j$. Since $\vg$ is balanced,
\[
 \gamma_{2\ve_i}=-\frac12\sum_{j\ne i}a_{ij}.
\]
It follows that
\[
 \langle\vg,\nu\rangle
 =-\frac12\sum_{i<j}a_{ij}\,d^{\bm0}_{ij}(\nu).
\]
Thus, in the standard reduced coordinates $a=(a_{ij})$, Theorem
\ref{thm:BRL-duality} gives
\[
 \BR_{\mathring L}(2,k)=-\mathrm{Cut}_k^\vee.
\]
The minus sign is solely a matter of convention. In
\cite[Theorem~B]{HHSW}, the supporting inequalities are written as
\[
 \sum_{i<j}a_{ij}d_{ij}\le0
 \qquad\bigl(d\in\mathrm{Cut}_k\bigr).
\]
Throughout this paper, however, we use the positive-dual convention
\[
 C^\vee=\{y:\langle y,x\rangle\ge0\text{ for every }x\in C\}.
\]
Thus the inequalities used in \cite{HHSW} define
$-\mathrm{Cut}_k^\vee$ in our convention.

Equivalently, after transporting the dual coordinates by
$b_{ij}=-a_{ij}/2$, this cone becomes $\mathrm{Cut}_k^\vee$. This is the
cut-cone description of \cite[Theorem~B]{HHSW}. The quadratic case of
Proposition~\ref{prop:treemetric} therefore recovers that result.
Proposition~\ref{prop:treemetric} extends it to arbitrary degree by passing
from a single tree metric to a field of tree metrics indexed by $H(n-2,k)$.
\end{remark}

We next introduce the split construction and analyze it using the local
tree-metric description.

\begin{definition}[{\rm \defn{Split metric}}]\label{def:split-metric}
Let $\varnothing \neq S \subsetneq \{1,\dots,k\}$. The \defn{split metric} of
the partition $\{S, S^{c}\}$ is
\[
  (\delta_S)_{ij} \;=\;
  \begin{cases}
    1, & \text{exactly one of } i,j \text{ lies in } S,\\
    0, & \text{otherwise.}
  \end{cases}
\]
This is a tree metric realized by a single edge of length one separating $S$
from $S^{c}$.
\end{definition}

\begin{definition}[{\rm \defn{Split function}}]\label{def:split}
Let $\varnothing \neq S \subsetneq \{1,\dots,k\}$ and $1 \le r \le n-1$.
For $\va=(\alpha_1,\ldots,\alpha_k)\in H(n,k)$, write
$\va(S) = \sum_{i \in S} \alpha_i$. The \defn{split function} is
\[
  \sigma_{S,r}\colon H(n,k)\longrightarrow\mathbb R,
  \qquad
  \sigma_{S,r}(\va) \;=\; \max\{0, \va(S) - r\} .
\]
\end{definition}

\begin{definition}[{\rm \defn{Split field}}]\label{def:split-field}
Let $\varnothing \neq S \subsetneq \{1,\dots,k\}$.
A field $(d^{\vb})_{\vb \in H(n-2,k)}$ of local Hessians is a \defn{split field}
for $S$ if each $d^{\vb}$ is a multiple of $\delta_S$. Equivalently, for every
$\vb$, the matrix $d^{\vb}$ vanishes on pairs on the same side of the partition
and has one common value on all cross-pairs. 
\end{definition}

\begin{lemma}\label{lem:splitmetric}
Let $\sigma_{S,r}$ be the split function.  Its
local discrete Hessians, in the sense of
Definition~\ref{def:local-hessian-field}, are
\[
  d^{\vb}(\sigma_{S,r}) \;=\;
  \begin{cases}
    \delta_S, & \vb(S) = r-1,\\
    0, & \text{otherwise,}
  \end{cases}
  \qquad \vb \in H(n-2,k).
\]
In particular $\sigma_{S,r}$
is M-convex.
\end{lemma}

\begin{proof}
Put
\[
 g(u)=\max\{0,u-r\},
\]
so that $\sigma_{S,r}(\va)=g(\va(S))$.  Fix
$\vb\in H(n-2,k)$ and write $t=\vb(S)$.  Note that $t$ is an integer.  For
each $i\in\{1,\ldots,k\}$, let
\[
 \epsilon_i=\mathbf 1_S(i)=
 \begin{cases}
  1,&i\in S,\\
  0,&i\notin S.
 \end{cases}
\]
Then
\[
 (\vb+2\ve_i)(S)=t+2\epsilon_i,
 \qquad
 (\vb+\ve_i+\ve_j)(S)=t+\epsilon_i+\epsilon_j.
\]
Substituting these identities into the definition of the local discrete
Hessian gives
\begin{align*}
 d^{\vb}_{ij}(\sigma_{S,r})
 &=\sigma_{S,r}(\vb+2\ve_i)
   +\sigma_{S,r}(\vb+2\ve_j)
   -2\sigma_{S,r}(\vb+\ve_i+\ve_j)\\
 &=g(t+2\epsilon_i)+g(t+2\epsilon_j)
   -2g(t+\epsilon_i+\epsilon_j).
\end{align*}

We evaluate this expression according to the positions of $i$ and $j$
relative to the partition $\{S,S^c\}$.  If $i$ and $j$ lie on the same side,
then $\epsilon_i=\epsilon_j$.  All three arguments of $g$ in the last
expression are therefore equal: they are all $t+2$ if $i,j\in S$, and all
$t$ if $i,j\in S^c$.  Hence
\[
 d^{\vb}_{ij}(\sigma_{S,r})=0.
\]

Suppose instead that exactly one of $i,j$ lies in $S$.  Then
$\{\epsilon_i,\epsilon_j\}=\{0,1\}$, and 
\[
 d^{\vb}_{ij}(\sigma_{S,r})
 =g(t+2)-2g(t+1)+g(t).
\]
There are three cases.  If $t\le r-2$, then $t,t+1,t+2\le r$, so
\[
 g(t)=g(t+1)=g(t+2)=0,
\]
and the second difference is zero.  If $t=r-1$, then
\[
 g(t)=g(r-1)=0,\qquad
 g(t+1)=g(r)=0,\qquad
 g(t+2)=g(r+1)=1,
\]
so the second difference is one.  Finally, if $t\ge r$, then
$g(u)=u-r$ at each of $u=t,t+1,t+2$, and therefore
\[
 g(t+2)-2g(t+1)+g(t)
 =(t+2-r)-2(t+1-r)+(t-r)=0.
\]
Consequently,
\[
 d^{\vb}_{ij}(\sigma_{S,r})=
 \begin{cases}
  1,&\vb(S)=r-1\text{ and exactly one of }i,j\text{ lies in }S,\\
  0,&\text{otherwise.}
 \end{cases}
\]
By the definition of the split metric, this entrywise calculation says
precisely that
\[
 d^{\vb}(\sigma_{S,r})=
 \begin{cases}
  \delta_S,&\vb(S)=r-1,\\
  0,&\vb(S)\ne r-1.
 \end{cases}
\]

The split metric $\delta_S$ is a tree metric and the
zero matrix is also a tree metric.  Thus every local discrete Hessian of
$\sigma_{S,r}$ is a tree metric.  Proposition~\ref{prop:treemetric} now
implies that $\sigma_{S,r}$ is M-convex.
\end{proof}

\begin{lemma}[Integrability of a split field]\label{lem:split}
Let $\varnothing\ne S\subsetneq\{1,\ldots,k\}$ and
$f:H(n,k)\to\mathbb R$.  Suppose that for any
$\vb\in H(n-2,k)$ there is a real number $h_{\vb}$ such that
\[
 d^{\vb}(f)=h_{\vb}\,\delta_S.
\]
Then
\begin{enumerate}
\item The number $h_{\vb}$ depends only on $\vb(S)$.
\item If $h_t$ denotes its common value on the layer $\vb(S)=t$, then
\[
 f=\ell+\sum_{t=0}^{n-2}h_t\,\sigma_{S,t+1}
\]
for some affine function $\ell$.
\end{enumerate}
\end{lemma}

\begin{proof}
Put $N=n-2$ and choose $p\in S$ and $q\in S^c$. The proof uses the following
polynomial derived from the Bernstein polynomial of $f$:
\[
 Q=(\partial_p-\partial_q)^2\mathcal B_f.
\]
We prove (i) in two steps and then prove (ii).

\emph{Proof of (i), Step 1: $Q$ depends only on the two block sums.}
Since $(\delta_S)_{pq}=1$, Lemma~\ref{lem:bernstein} and the hypothesis
$d^{\vb}(f)=h_{\vb}\delta_S$ give
\begin{equation}\label{eq:split-Q-coefficients}
 Q=n(n-1)\sum_{\vb\in H(N,k)}
 \binom{N}{\vb}h_{\vb}\,\vx^{\vb}.
\end{equation}
Applying the same argument to an arbitrary pair of indices yields the identity
\begin{equation}\label{eq:split-Q-pairs}
 (\partial_a-\partial_b)^2\mathcal B_f
 =(\delta_S)_{ab}Q.
\end{equation}
Indeed, Lemma~\ref{lem:bernstein} says that the coefficient of
$\vx^{\vb}$ on the left is
\[
 n(n-1)\binom{N}{\vb}h_{\vb}(\delta_S)_{ab}.
\]
Thus the second derivative is zero for a pair on the same side and equals $Q$
for a cross-pair.

We next show that redistributing the variables within either block does not
change $Q$.  Fix $i\in S\setminus\{p\}$.  The pair $(i,p)$ is on the same
side, while $(i,q)$ and $(p,q)$ are cross-pairs.  Hence
\eqref{eq:split-Q-pairs} gives
\[
 (\partial_i-\partial_p)^2\mathcal B_f=0,
 \qquad
 (\partial_i-\partial_q)^2\mathcal B_f
 =(\partial_p-\partial_q)^2\mathcal B_f=Q.
\]
Using
\[
 \partial_i-\partial_q
 =(\partial_i-\partial_p)+(\partial_p-\partial_q)
\]
and expanding the square, we obtain
\[
 (\partial_i-\partial_p)(\partial_p-\partial_q)\mathcal B_f=0.
\]
All these differential operators have constant coefficients and therefore
commute.  Applying $\partial_p-\partial_q$ once more gives
\begin{equation}\label{eq:split-Q-S-directions}
 (\partial_i-\partial_p)Q=0
 \qquad(i\in S\setminus\{p\}).
\end{equation}

An analogous calculation applies to the other block, with the sign in the
decomposition reversed. If $j\in S^c\setminus\{q\}$, then
\[
 (\partial_j-\partial_q)^2\mathcal B_f=0,
 \qquad
 (\partial_j-\partial_p)^2\mathcal B_f
 =(\partial_p-\partial_q)^2\mathcal B_f=Q,
\]
and
\[
 \partial_j-\partial_p
 =(\partial_j-\partial_q)-(\partial_p-\partial_q).
\]
Expanding the square gives
\[
 (\partial_j-\partial_q)(\partial_p-\partial_q)\mathcal B_f=0.
\]
Applying $\partial_p-\partial_q$ once more gives
\begin{equation}\label{eq:split-Q-Sc-directions}
 (\partial_j-\partial_q)Q=0
 \qquad(j\in S^c\setminus\{q\}).
\end{equation}

Write
\[
 X=\sum_{i\in S}x_i,
 \qquad
 Y=\sum_{j\in S^c}x_j.
\]
For an arbitrary $\vx\in\mathbb R^k$, put
\[
 \vw=\vx-X(\vx)\ve_p-Y(\vx)\ve_q.
\]
This vector is a linear combination of the directions appearing in
\eqref{eq:split-Q-S-directions} and
\eqref{eq:split-Q-Sc-directions}, because
\[
 \vw=
 \sum_{i\in S\setminus\{p\}}x_i(\ve_i-\ve_p)
 +\sum_{j\in S^c\setminus\{q\}}x_j(\ve_j-\ve_q).
\]
If one block is a singleton, the corresponding sum is empty.  By linearity
of directional differentiation, \eqref{eq:split-Q-S-directions} and
\eqref{eq:split-Q-Sc-directions} give the polynomial identity
\[
 D_{\vw}Q
 =
 \sum_{i\in S\setminus\{p\}}x_i(\partial_i-\partial_p)Q
 +
 \sum_{j\in S^c\setminus\{q\}}x_j(\partial_j-\partial_q)Q
 =0.
\]
For fixed $\vx$, the quantities $X(\vx)$, $Y(\vx)$, and $\vw$ are fixed as
$\lambda$ varies.  Hence the chain rule gives
\[
 \frac{d}{d\lambda}
 Q\bigl(X(\vx)\ve_p+Y(\vx)\ve_q+\lambda\vw\bigr)
 =D_{\vw}Q\bigl(X(\vx)\ve_p+Y(\vx)\ve_q+\lambda\vw\bigr)
 =0
\]
for every $\lambda\in\mathbb R$.  Therefore the function
\[
 \lambda\longmapsto
 Q\bigl(X(\vx)\ve_p+Y(\vx)\ve_q+\lambda\vw\bigr)
\]
is constant.  At $\lambda=0$, its value is
\[
 Q\bigl(X(\vx)\ve_p+Y(\vx)\ve_q\bigr).
\]
At $\lambda=1$, the definition of $\vw$ gives
\begin{align*}
 X(\vx)\ve_p+Y(\vx)\ve_q+\vw
 &=X(\vx)\ve_p+Y(\vx)\ve_q
   +\vx-X(\vx)\ve_p-Y(\vx)\ve_q\\
 &=\vx,
\end{align*}
so the value is $Q(\vx)$.  Since the function is constant, these two values
are equal.  Therefore
\[
 Q(\vx)=Q\bigl(X(\vx)\ve_p+Y(\vx)\ve_q\bigr).
\]
Thus $Q$ depends only on the two block sums $X$ and $Y$.  Since $Q$ is
homogeneous of degree $N$, there are real numbers $c_0,\ldots,c_N$ such that
\begin{equation}\label{eq:split-Q-blocks}
 Q=\sum_{t=0}^{N}c_tX^tY^{N-t}.
\end{equation}

\emph{Proof of (i), Step 2: $h_{\vb}$ is constant on each layer.}
Fix $\vb=(\beta_1,\ldots,\beta_k)\in H(N,k)$ and put $t=\vb(S)$.  From
\eqref{eq:split-Q-coefficients},
\[
 [\vx^{\vb}]Q
 =n(n-1)\frac{N!}{\vb!}\,h_{\vb},
 \qquad \vb!=\prod_{i=1}^k\beta_i!.
\]
On the other hand, every monomial in $X^sY^{N-s}$ has total degree $s$ in
the variables indexed by $S$.  Hence $\vx^{\vb}$ occurs in
\eqref{eq:split-Q-blocks} only in the term $c_tX^tY^{N-t}$.  The multinomial
theorem gives
\[
 [\vx^{\vb}]Q
 =c_t
 \frac{t!}{\prod_{i\in S}\beta_i!}
 \frac{(N-t)!}{\prod_{j\in S^c}\beta_j!}
 =c_t\frac{t!(N-t)!}{\vb!}.
\]
Equating the two coefficient formulas and cancelling $\vb!$ yields
\[
 h_{\vb}
 =\frac{c_t\,t!(N-t)!}{n(n-1)N!}
 =\frac{c_t}{n(n-1)\binom Nt}.
\]
The right-hand side depends on $\vb$ only through $t=\vb(S)$, proving (i).
Every layer is nonempty: for $0\le t\le N$, the vector
\[
 t\ve_p+(N-t)\ve_q
\]
belongs to $H(N,k)$ and has $S$-sum $t$. Thus the common value $h_t$ on each
layer is well defined. When $n=2$, this simply gives the unique layer
$t=N=0$.

\emph{Proof of (ii): reconstruction from the split functions.}
Define
\[
 \widetilde f=\sum_{t=0}^{N}h_t\sigma_{S,t+1}.
\]
The indices $t+1$ range from $1$ to $n-1$, as required in
Definition~\ref{def:split}.  Fix $\vb\in H(N,k)$ and put $t=\vb(S)$.  By
Lemma~\ref{lem:splitmetric}, only the term indexed by this value of $t$ has a
nonzero local Hessian at $\vb$.  Hence
\[
 d^{\vb}(\widetilde f)
 =h_t\delta_S
 =h_{\vb}\delta_S
 =d^{\vb}(f).
\]
It follows that every local Hessian of $f-\widetilde f$ is zero.
Lemma~\ref{lem:hessiankernel} therefore implies that $f-\widetilde f$ is
affine.  Writing $f-\widetilde f=\ell$ and recalling that $N=n-2$, we obtain
\[
 f=\ell+\sum_{t=0}^{n-2}h_t\sigma_{S,t+1}.
\]
This proves (ii).
\end{proof}

\begin{lemma}\label{lem:splitrigidity}
Let $d_1,\ldots,d_m$ be tree metrics such that
\[
 \delta_S=d_1+\cdots+d_m.
\]
Then, for every $a$, there exists $c_a\ge0$ such that $d_a=c_a\delta_S$.
\end{lemma}

\begin{proof}
All entries of a tree metric are nonnegative.  Since $\delta_S$ vanishes on
pairs contained in $S$ and on pairs contained in $S^c$, every $d_a$ vanishes
on those pairs as well.  Fix $a$.  If $p,p'\in S$ and $q\in S^c$, the triangle
inequality gives
\[
 d_a(p,q)\le d_a(p,p')+d_a(p',q)=d_a(p',q),
\]
and interchanging $p,p'$ gives equality. Repeating the argument for the
endpoint in $S^c$ shows that every cross-distance of $d_a$ has the same value
$c_a\ge0$.  Every same-side distance is zero.  Hence
$d_a=c_a\delta_S$.
\end{proof}

\begin{theorem}[Split facets]\label{thm:splits}
For every nonempty $S\subsetneq\{1,\ldots,k\}$ and every $1\le r\le n-1$, the
class of
\[
 \sigma_{S,r}(\va)=\max\{0,\va(S)-r\}
\]
spans an extreme ray of $\mathcal M_{n,k}$.  Equivalently,
\[
 \sum_{\va\in H(n,k)}
 \gamma_{\va}\max\{0,\va(S)-r\}\ge0
\]
defines a facet of $\BR_{\mathring L}(n,k)$.
\end{theorem}

\begin{proof}
Write
\[
 \sigma=\sigma_{S,r},
 \qquad
 \delta=\delta_S.
\]
We divide the proof into three steps.

\emph{Step 1: $[\sigma]$ is a nonzero element of $\mathcal M_{n,k}$.}
Choose $p\in S$ and $q\in S^c$, and set
\[
 \vb_0=(r-1)\ve_p+(n-r-1)\ve_q.
\]
Because $1\le r\le n-1$, both coefficients are nonnegative and their sum is
$n-2$.  Thus $\vb_0\in H(n-2,k)$.  Moreover, $p\in S$ and $q\notin S$, so
\[
 \vb_0(S)=r-1.
\]
Lemma~\ref{lem:splitmetric} now gives
\[
 d^{\vb_0}(\sigma)=\delta.
\]
This matrix is nonzero because $\delta_{pq}=1$.  Since affine functions have
zero local Hessians, $\sigma$ is not affine and hence $[\sigma]\ne0$.
Lemma~\ref{lem:splitmetric} also shows that $\sigma$ is M-convex.  Therefore
\[
 [\sigma]\in\mathcal M_{n,k}\setminus\{0\}.
\]

\emph{Step 2: every decomposition of $[\sigma]$ stays on its ray.}
Suppose that
\[
 [\sigma]=x+y,
 \qquad
 x,y\in\mathcal M_{n,k}.
\]
By Lemma~\ref{lem:Mpolyhedral}, after reindexing there are M-convex functions
$f_1,\ldots,f_m$ and an integer $0\le u\le m$ such that
\[
 x=\sum_{a=1}^{u}[f_a],
 \qquad
 y=\sum_{a=u+1}^{m}[f_a].
\]
Consequently $[\sigma]=\sum_{a=1}^{m}[f_a]$.  Equality in the affine quotient
means that $\sigma-\sum_{a=1}^{m}f_a$ is affine.  Local Hessians are linear
and vanish on affine functions, so for every $\vb\in H(n-2,k)$,
\[
 \sum_{a=1}^{m}d^{\vb}(f_a)
 =d^{\vb}(\sigma)
 =
 \begin{cases}
  \delta,&\vb(S)=r-1,\\
  0,&\vb(S)\ne r-1,
 \end{cases}
\]
where the last equality is Lemma~\ref{lem:splitmetric}.  Since every $f_a$ is
M-convex, Proposition~\ref{prop:treemetric} shows that each
$d^{\vb}(f_a)$ is a tree metric and therefore has nonnegative entries.

If $\vb(S)\ne r-1$, the preceding display expresses the zero matrix as a sum
of entrywise nonnegative matrices.  Hence
\[
 d^{\vb}(f_a)=0
 \qquad(a=1,\ldots,m).
\]
If $\vb(S)=r-1$, the same display becomes
\[
 \delta=\sum_{a=1}^{m}d^{\vb}(f_a).
\]
For every $a$, Lemma~\ref{lem:splitrigidity} gives a number
$\lambda_{a,\vb}\ge0$ such that
$d^{\vb}(f_a)=\lambda_{a,\vb}\delta$.  For uniformity, set
$\lambda_{a,\vb}=0$ whenever $\vb(S)\ne r-1$.  We then have
\[
 d^{\vb}(f_a)=\lambda_{a,\vb}\delta
 \qquad
 \text{for every $a$ and every $\vb\in H(n-2,k)$.}
\]

Fix $a$.  Lemma~\ref{lem:split}(i) shows that $\lambda_{a,\vb}$ depends only on
$\vb(S)$.  Therefore the values $\lambda_{a,\vb}$ are equal for all $\vb$ in
the active layer $\vb(S)=r-1$, even though that layer may contain many
multi-indices.  Denote this common value by $\lambda_a\ge0$.  Since
$\lambda_{a,\vb}=0$ on every other layer, the reconstruction formula in
Lemma~\ref{lem:split}(ii) reduces to
\[
 f_a=\ell_a+\lambda_a\sigma
\]
for some affine function $\ell_a$.  Thus
\[
 [f_a]=\lambda_a[\sigma],
 \qquad \lambda_a\ge0.
\]
Substituting these identities into the decompositions of $x$ and $y$ gives
\[
 x=\left(\sum_{a=1}^{u}\lambda_a\right)[\sigma],
 \qquad
 y=\left(\sum_{a=u+1}^{m}\lambda_a\right)[\sigma].
\]
Both coefficients are nonnegative.  Since $x+y=[\sigma]$ and
$[\sigma]\ne0$, their sum is $1$.  Hence every decomposition of $[\sigma]$
in $\mathcal M_{n,k}$ has both summands on
$\mathbb R_{\ge0}[\sigma]$, proving that this ray is extreme.

\emph{Step 3: the corresponding facet.}
Since $[\sigma]$ is nonzero and spans an extreme ray,
Corollary~\ref{cor:facets}(ii) shows that
\[
 \langle\vg,\sigma\rangle=0
\]
cuts out a facet of $\BR_{\mathring L}(n,k)$.  Its supporting inequality is
\[
 \langle\vg,\sigma\rangle
 =\sum_{\va\in H(n,k)}
 \gamma_{\va}\max\{0,\va(S)-r\}\ge0,
\]
as claimed.
\end{proof}

\begin{proposition}
\label{prop:splitclasses}
The rays spanned by two split functions modulo affine functions satisfy
\[
 \mathbb R_{\ge0}[\sigma_{S,r}]
 =\mathbb R_{\ge0}[\sigma_{T,s}]
 \quad\Longleftrightarrow\quad
 (T,s)=(S,r)\ \text{or}\ (T,s)=(S^c,n-r).
\]
Consequently $\BR_{\mathring L}(n,k)$ has exactly
\[
 (2^{k-1}-1)(n-1)
\]
split facets.
\end{proposition}

\begin{proof}
Since $\va(S^c)=n-\va(S)$ and
$\max\{0,u\}-\max\{0,-u\}=u$, one has
\[
 \sigma_{S,r}(\va)-\sigma_{S^c,n-r}(\va)=\va(S)-r,
\]
which is affine. Thus either case displayed on the right yields the same ray.

Conversely, suppose
$[\sigma_{S,r}]=\lambda[\sigma_{T,s}]$ for some $\lambda>0$.  Local Hessians
annihilate affine functions, so the two local metric fields are proportional.
Choose $\vb\in H(n-2,k)$ with $\vb(S)=r-1$. At this index, the first field is
$\delta_S\ne0$, so the second field must also be active:
\[
 \vb(T)=s-1,\qquad \delta_S=\lambda\delta_T.
\]
All nonzero entries of a split metric equal $1$, so $\lambda=1$.  Equality of
the two split metrics means that their unordered bipartitions agree.  Thus
$T=S$ or $T=S^c$.  In the first case
$s-1=\vb(S)=r-1$, so $s=r$.  In the second,
\[
 s-1=\vb(S^c)=(n-2)-(r-1)=n-r-1,
\]
so $s=n-r$.

There are $(2^k-2)(n-1)$ pairs $(S,r)$.  The complement involution
$(S,r)\mapsto(S^c,n-r)$ has no fixed point, so the number of equivalence
classes, and hence of split facets by Theorem~\ref{thm:splits}, is
$(2^k-2)(n-1)/2=(2^{k-1}-1)(n-1)$.
\end{proof}

We conclude the section by specializing the local description to three
variables. In this case every metric is a tree metric, so the resulting
inequalities can be dualized explicitly.

\begin{definition}[{\rm \defn{Triangular ratios}}]\label{def:triangular}
Let $p=(p_{uv})$ be a symmetric matrix with positive entries.  For pairwise
distinct $i,j,l$, the \defn{triangular ratio} of $p$ at $(i\mid jl)$ is
\[
 T_{(i\mid jl)}(p)
 =\frac{p_{ii}\,p_{jl}}{p_{ij}\,p_{il}}.
\]
If $p$ is Lorentzian, then $T_{(i\mid jl)}(p)\le2$
\cite[Example~1.4]{HHSW}.  In geometric settings this is a special case of
the reverse Khovanskii--Teissier inequality
\cite[Theorem~5.7]{LX17}.

Now write
\[
 F(\vx)=\sum_{\va\in H(n,k)}P_{\va}\vx^{\va},
 \qquad \widehat P_{\va}=\va!P_{\va},
\]
and, for $\vb\in H(n-2,k)$, put
\[
 \mathcal H^{(\vb)}=\operatorname{Hess}(\partial^{\vb}F).
\]
Then
\[
 \mathcal H^{(\vb)}_{uv}=\widehat P_{\vb+\ve_u+\ve_v},
\]
so the triangular ratio of this Hessian is
\[
 T_{(\vb;i\mid jl)}
 :=T_{(i\mid jl)}(\mathcal H^{(\vb)})
 =\frac{
 \widehat P_{\vb+2\ve_i}\,
 \widehat P_{\vb+\ve_j+\ve_l}}
 {\widehat P_{\vb+\ve_i+\ve_j}\,
 \widehat P_{\vb+\ve_i+\ve_l}}.
\]
Rewriting this expression in the original coefficients gives
\[
 T_{(\vb;i\mid jl)}
 =\kappa_{\vb,i}
 \frac{P_{\vb+2\ve_i}\,P_{\vb+\ve_j+\ve_l}}
 {P_{\vb+\ve_i+\ve_j}\,P_{\vb+\ve_i+\ve_l}},
 \qquad
 \kappa_{\vb,i}
 =\frac{\beta_i+2}{\beta_i+1}.
\]
The positive factor $\kappa_{\vb,i}$ depends only on $\vb$ and $i$, not on
$F$. Although it changes the numerical value and the optimal bounding
constant, it leaves the exponent vector unchanged. We define the
\defn{triangular ratio} $\vrho(\vb;i\mid jl)$ to be the exponent
vector of the coefficient ratio
\[
 R_{\vrho(\vb;i\mid jl)}(\vP)
 =
  \frac{P_{\vb+2\ve_i}\; P_{\vb+\ve_j+\ve_l}}
       {P_{\vb+\ve_i+\ve_j}\; P_{\vb+\ve_i+\ve_l}} .
\]
Thus
\[
 T_{(\vb;i\mid jl)}
 =\kappa_{\vb,i}R_{\vrho(\vb;i\mid jl)}(\vP).
\]
Interchanging $j$ and $l$ does not change this ratio, so the pair
$\{j,l\}$ is always regarded as unordered.
\end{definition}

\begin{lemma}[Three-variable rhombus inequalities]
\label{lem:three-rhombus}
A function $\nu:H(n,3)\to\mathbb R$ is M-convex if and only if
 \[
 \langle\vrho(\vb;i\mid jk),\nu\rangle\ge0
 \quad\text{for all }\vb\in H(n-2,3),\quad i\in\{1,2,3\},\quad
 \{j,k\}=\{1,2,3\}\setminus\{i\}.
 \]
\end{lemma}

\begin{proof}
\emph{Step 1: reduction to the local metric fields.}
For each $\vb\in H(n-2,3)$, let
\[
 d^{\vb}=d^{\vb}(\nu).
\]
By Proposition~\ref{prop:treemetric}, the function $\nu$ is M-convex if and
only if every matrix $d^{\vb}$ is a tree metric on the three labels
$\{1,2,3\}$.  Thus it remains to translate the tree-metric condition on each
$d^{\vb}$ into inequalities involving the values of $\nu$.

\smallskip
\noindent
\emph{Step 2: tree metrics on three labels.}
A symmetric matrix $d=(d_{uv})_{u,v=1}^3$ with zero diagonal is a tree metric
if and only if it satisfies the three triangle inequalities
\begin{equation}\label{eq:three-triangle}
 d_{ij}+d_{ik}-d_{jk}\ge0,
 \qquad
 i\in\{1,2,3\},\quad
 \{j,k\}=\{1,2,3\}\setminus\{i\}.
\end{equation}
Indeed, if these inequalities hold, attach the three labels to a central
vertex by edges of lengths
\[
 \frac{d_{ij}+d_{ik}-d_{jk}}2,\qquad
 \frac{d_{ij}+d_{jk}-d_{ik}}2,\qquad
 \frac{d_{ik}+d_{jk}-d_{ij}}2.
\]
These lengths are nonnegative by \eqref{eq:three-triangle}, and the resulting
tree has pairwise distances $d_{ij},d_{ik},d_{jk}$.  The converse follows
from the triangle inequality for distances in a tree.

No separate nonnegativity assumption on the entries of $d$ is needed.  For
example, adding
\[
 d_{12}\le d_{13}+d_{23}
 \qquad\text{and}\qquad
 d_{13}\le d_{12}+d_{23}
\]
gives $d_{23}\ge0$.  The other two entries are handled in the same way.

\smallskip
\noindent
\emph{Step 3: each triangle inequality is a rhombus
inequality.}
Fix $\vb\in H(n-2,3)$ and distinct $i,j,k$.  By the definition of the local
discrete Hessian,
\[
\begin{aligned}
 d^{\vb}_{ij}(\nu)
 &=\nu(\vb+2\ve_i)+\nu(\vb+2\ve_j)
   -2\nu(\vb+\ve_i+\ve_j),\\
 d^{\vb}_{ik}(\nu)
 &=\nu(\vb+2\ve_i)+\nu(\vb+2\ve_k)
   -2\nu(\vb+\ve_i+\ve_k),\\
 d^{\vb}_{jk}(\nu)
 &=\nu(\vb+2\ve_j)+\nu(\vb+2\ve_k)
   -2\nu(\vb+\ve_j+\ve_k).
\end{aligned}
\]
Subtracting the third line from the sum of the first two cancels the terms
$\nu(\vb+2\ve_j)$ and $\nu(\vb+2\ve_k)$ and gives
\[
\begin{aligned}
 &d^{\vb}_{ij}(\nu)+d^{\vb}_{ik}(\nu)-d^{\vb}_{jk}(\nu)\\
 &\quad=2\bigl(
 \nu(\vb+2\ve_i)+\nu(\vb+\ve_j+\ve_k)
 -\nu(\vb+\ve_i+\ve_j)-\nu(\vb+\ve_i+\ve_k)
 \bigr)\\
 &\quad=2\langle\vrho(\vb;i\mid jk),\nu\rangle.
\end{aligned}
\]
Since the factor $2$ is positive, the triangle inequality
\[
 d^{\vb}_{ij}(\nu)+d^{\vb}_{ik}(\nu)-d^{\vb}_{jk}(\nu)\ge0
\]
holds if and only if
\[
 \langle\vrho(\vb;i\mid jk),\nu\rangle\ge0.
\]
Allowing $i$ to range over $\{1,2,3\}$ gives exactly the three triangle
inequalities for $d^{\vb}(\nu)$.  Combining Steps 1--3 proves the lemma.
\end{proof}

\begin{lemma}[The three-variable M-convex cone]
\label{lem:three-rhombus-cone}
The M-convex functions on $H(n,3)$ form the rational
polyhedral cone
\[
 \begin{aligned}
 \mathcal C_{n,3}
 =\bigl\{\nu\in\mathbb R^{H(n,3)}:\;&
 \langle\vrho(\vb;i\mid jk),\nu\rangle\ge0
 \text{ for all }\vb\in H(n-2,3),\\
 &i\in\{1,2,3\},\quad
 \{j,k\}=\{1,2,3\}\setminus\{i\}\bigr\}.
 \end{aligned}
\]
This cone is cut out by the $3\binom n2$ hive inequalities, also called
rhombus inequalities.
Moreover, if
\[
 \pi\colon\mathbb R^{H(n,3)}
 \longrightarrow
 \mathbb R^{H(n,3)}/\operatorname{Aff}(H(n,3))
\]
is the quotient map, then
\[
 \mathcal M_{n,3}=\pi(\mathcal C_{n,3}).
\]
\end{lemma}

\begin{proof}
Lemma~\ref{lem:three-rhombus} shows that the set of M-convex functions is exactly
the intersection of the displayed linear half-spaces.  It is therefore a
closed convex rational polyhedral cone.  There are
\[
 3|H(n-2,3)|=3\binom n2
\]
displayed inequalities: one for each $\vb\in H(n-2,3)$ and each choice of the
distinguished index $i\in\{1,2,3\}$.

This set is already a closed convex cone, and its image under the quotient map
$\pi$ is a polyhedral cone and hence closed. Therefore, taking
the conical hull and then the closure in the definition of $\mathcal M_{n,3}$
adds nothing, and
\[
 \mathcal M_{n,3}=\pi(\mathcal C_{n,3}).
\]
\end{proof}

\begin{lemma}[Irredundancy of the rhombus inequalities]
\label{lem:three-irredundant}
Each inequality in Lemma~\ref{lem:three-rhombus} defines a distinct facet of
the M-convex cone modulo affine functions.  Equivalently, the vectors
$\vrho(\vb;i\mid jk)$ span distinct extreme rays of its dual.
\end{lemma}

\begin{proof}
First, note that every rhombus normal annihilates affine functions.  Indeed,
the sum of its four coefficients is zero, and the corresponding weighted sum
of the indices is
\[
 (\bm\eta+2\ve_\ell)+(\bm\eta+\ve_p+\ve_q)
 -(\bm\eta+\ve_\ell+\ve_p)
 -(\bm\eta+\ve_\ell+\ve_q)=\bm0.
\]
Thus every rhombus inequality is well defined on
$\mathbb R^{H(n,3)}/\operatorname{Aff}(H(n,3))$.

Fix one of the rhombus inequalities, say
\begin{equation}\label{eq:chosen-rhombus}
 \langle\vrho(\vb;i\mid jk),\nu\rangle\ge0.
\end{equation}
We construct a function that satisfies \eqref{eq:chosen-rhombus} with equality
and every other rhombus inequality strictly.  The construction proceeds in
four steps.

\smallskip
\noindent
\emph{(i) Construct a strictly feasible function.}
For $\ell\in\{1,2,3\}$ and $0\le t\le n-2$, define
\[
 s_{\ell,t}=
 \begin{cases}
  1,&(\ell,t)=(i,\beta_i),\\
  2,&\text{otherwise}.
 \end{cases}
\]
Thus all the numbers $s_{\ell,t}$ are positive and exactly one of them is
equal to $1$.

For each $\ell\in\{1,2,3\}$, define
$f_\ell:\{0,\ldots,n\}\to\mathbb R$ by
\[
 f_\ell(0)=f_\ell(1)=0
\]
and the recurrence
\begin{equation}\label{eq:second-difference-witness}
 f_\ell(t+2)-2f_\ell(t+1)+f_\ell(t)=s_{\ell,t},
 \qquad 0\le t\le n-2.
\end{equation}
The initial values and the recurrence determine $f_\ell$ uniquely.  Now put
\[
 \nu_0(\va)=\sum_{m=1}^3 f_m(\alpha_m),
 \qquad \va=(\alpha_1,\alpha_2,\alpha_3)\in H(n,3).
\]

Consider an arbitrary rhombus normal
$\vrho(\bm\eta;\ell\mid pq)$, where
$\{\ell,p,q\}=\{1,2,3\}$.  Its pairing with a function $w$ is
\begin{equation}\label{eq:rhombus-pairing-expanded}
\begin{aligned}
 \langle\vrho(\bm\eta;\ell\mid pq),w\rangle
 ={}&w(\bm\eta+2\ve_\ell)
     +w(\bm\eta+\ve_p+\ve_q)\\
 &-w(\bm\eta+\ve_\ell+\ve_p)
  -w(\bm\eta+\ve_\ell+\ve_q).
\end{aligned}
\end{equation}
Upon substituting $w=\nu_0$, the terms involving $f_p$ cancel in pairs, as do
those involving $f_q$. The remaining terms are
\[
\begin{aligned}
 \langle\vrho(\bm\eta;\ell\mid pq),\nu_0\rangle
 &=f_\ell(\eta_\ell+2)
   -2f_\ell(\eta_\ell+1)+f_\ell(\eta_\ell)\\
 &=s_{\ell,\eta_\ell}>0
\end{aligned}
\]
by \eqref{eq:second-difference-witness}.  Hence every rhombus inequality is
strict at $\nu_0$.  Since there are only finitely many rhombus inequalities,
all their values remain positive under every sufficiently small perturbation
of $\nu_0$.  Therefore $[\nu_0]$ is an interior point of the cone in
$\mathbb R^{H(n,3)}/\operatorname{Aff}(H(n,3))$.  In particular, this cone is
full-dimensional in the quotient.

\smallskip
\noindent
\emph{(ii) Move the chosen inequality to equality.}
Let
\[
 x=\vb+2\ve_i\in H(n,3),
 \qquad
 \nu=\nu_0-\mathbf1_x.
\]
Here $\mathbf1_x$ is the indicator function of the lattice point $x$.
For an arbitrary rhombus normal, linearity and Step~(i) give
\begin{equation}\label{eq:witness-slack}
 \langle\vrho(\bm\eta;\ell\mid pq),\nu\rangle
 =s_{\ell,\eta_\ell}
  -\langle\vrho(\bm\eta;\ell\mid pq),\mathbf1_x\rangle.
\end{equation}
By \eqref{eq:rhombus-pairing-expanded}, the second term is simply the
coefficient of $w(x)$ in that formula.  Since the four points there are
distinct, this coefficient is $1$, $-1$, or $0$.  It is $1$ precisely when
\[
 x=\bm\eta+2\ve_\ell
 \qquad\text{or}\qquad
 x=\bm\eta+\ve_p+\ve_q.
\]
For the chosen normal, $x=\vb+2\ve_i$ is the first of these two positive
points, and $s_{i,\beta_i}=1$.  Hence
\[
 \langle\vrho(\vb;i\mid jk),\nu\rangle=1-1=0.
\]

\smallskip
\noindent
\emph{(iii) Every other inequality remains strict.}
We have $s_{\ell,\eta_\ell}\ge1$ and
$\langle\vrho(\bm\eta;\ell\mid pq),\mathbf1_x\rangle\le1$, so
\eqref{eq:witness-slack} is always nonnegative.  Equality can occur only if
\begin{equation}\label{eq:possible-second-zero}
 s_{\ell,\eta_\ell}=1
 \qquad\text{and}\qquad
 \langle\vrho(\bm\eta;\ell\mid pq),\mathbf1_x\rangle=1.
\end{equation}
The first equality in \eqref{eq:possible-second-zero} forces
\[
 \ell=i,
 \qquad
 \eta_i=\beta_i.
\]
The second equality says that $x$ is one of the two points having positive
coefficient in $\vrho(\bm\eta;i\mid pq)$.  Hence either
\[
 x=\bm\eta+2\ve_i
 \qquad\text{or}\qquad
 x=\bm\eta+\ve_p+\ve_q.
\]
The second possibility is impossible: its $i$-coordinate is
$\eta_i=\beta_i$, whereas the $i$-coordinate of
$x=\vb+2\ve_i$ is $\beta_i+2$.  Therefore
\[
 x=\bm\eta+2\ve_i.
\]
Since also $x=\vb+2\ve_i$, we obtain $\bm\eta=\vb$.  Thus the only
inequality that is active at $\nu$ is the chosen inequality
\eqref{eq:chosen-rhombus}.  All the other inequalities are strict.

\par\smallskip
\noindent
\emph{(iv) Deduce the facet and distinctness statements.}
Every rhombus normal annihilates affine functions, so every rhombus inequality
defines a half-space in
$\mathbb R^{H(n,3)}/\operatorname{Aff}(H(n,3))$. By Step~(i), their intersection
is full-dimensional in this quotient.  At the point $\nu$, the chosen
inequality is an equality and all the others are strict.  Consequently, a
sufficiently small neighborhood of $[\nu]$ inside the hyperplane
\[
 \langle\vrho(\vb;i\mid jk),\,\cdot\,\rangle=0
\]
remains in the cone.  Thus this face contains a relatively open neighborhood
in the displayed hyperplane.  Since $\vrho(\vb;i\mid jk)\ne0$, the
hyperplane has codimension one in the quotient.  The face therefore also has
codimension one and is a facet.

The same argument applies to every choice of $\vb$ and $i$.  Finally, two
distinct displayed normals cannot be positive scalar multiples of one
another: if they were, their inequalities would be equalities at exactly the
same points, contradicting the witness above at which only one of them is
active.  Hence the resulting facets, and equivalently the dual extreme rays,
are distinct.
\end{proof}

\begin{corollary}
\label{cor:threevariable}
\[
 \BR_{\mathring L}(n,3)
 =
 \operatorname{cone}\bigl\{
 \vrho(\vb;i\mid jk):
 \vb\in H(n-2,3),\ i\in\{1,2,3\},\
 \{j,k\}=\{1,2,3\}\setminus\{i\}
 \bigr\}, ~~~\qquad n\geq2.
\]
All the displayed generators span distinct extreme rays.  Hence
$\BR_{\mathring L}(n,3)$ has exactly
\[
 3|H(n-2,3)|=3\binom n2
\]
triangular-ratio extreme rays.  Their facet normals are the extreme
rays of the hive cone modulo affine functions.
\end{corollary}

\begin{proof}
Lemma~\ref{lem:three-rhombus-cone} presents the M-convex functions themselves as a
convex cone, so in three variables this cone is already
$\mathcal M_{n,3}$, rather than merely a generating subset of it.  Theorem
\ref{thm:BRL-duality} and elementary polyhedral duality give
\[
 \BR_{\mathring L}(n,3)=\mathcal M_{n,3}^{\vee}
 =\operatorname{cone}\{\vrho(\vb;i\mid jk)\}.
\]
Lemma~\ref{lem:three-irredundant} says exactly that these generators span
distinct extreme rays.  Finally,
$|H(n-2,3)|=\binom n2$, which gives the count.  The statement about facet
normals follows from Corollary~\ref{cor:facets}.
\end{proof}

The role of the hypothesis $k=3$ extends beyond the fact that every metric on
three labels is a tree metric. On four labels, the split metrics
$\delta_{\{1,2\}}$ and
$\delta_{\{1,3\}}$ are tree metrics, but their sum is not. Its three
four-point sums are $2,2,4$.

This obstruction persists in every degree $n\ge2$ and for every $k\ge4$.
Indeed, put
\[
 S=\{1,2\},\qquad T=\{1,3\},\qquad
 \vb=(n-2)\ve_4.
\]
By Lemma~\ref{lem:splitmetric}, the split functions
$\sigma_{S,1}$ and $\sigma_{T,1}$ are M-convex, while
\[
 d^{\vb}(\sigma_{S,1}+\sigma_{T,1})
 =\delta_S+\delta_T.
\]
The restriction of this local discrete Hessian to the first four labels
violates the four-point condition. Proposition~\ref{prop:treemetric} therefore
shows that $\sigma_{S,1}+\sigma_{T,1}$ is not M-convex. Hence, for every
$n\ge2$ and $k\ge4$, the M-convex functions are not closed under addition,
and the rhombus inequalities alone do not permit the preceding dualization.
The three-variable conclusion also fails numerically for $(n,k)=(2,5)$:
$\BR_{\mathring L}(2,5)$ has forty extreme rays---the thirty nondegenerate
triangular rays and ten pentagonal rays \cite[Theorem~A]{HHSW}.

\section{Optimal upper bounds for ternary Lorentzian cubics}
\label{sec:optimal-ternary}

We now determine the optimal bounding constant for every bounded ratio on
strictly Lorentzian ternary cubics. We first solve the problem for triangular
ratios coming from a single Hessian based on the proof of
\cite[Theorem~C]{HHSW}. Then we study the compatibility conditions linking all
nine triangular ratios from the three Hessians of a Lorentzian ternary cubic.
We apply the Fenchel--Rockafellar sum theorem to untie the dependence among all
nine ratios in the minimization process.

\subsection{One-Hessian optimization}

\begin{definition}[{\rm \defn{The function $\mathfrak m$}}]\label{def:m-function}
For $a,b,c \ge 0$, set $m = a+b+c$ and
$\Delta(a,b,c) = a^2+b^2+c^2-2ab-2ac-2bc$. Define $\mathfrak{m}(a,b,c) = 1$ if
$\Delta(a,b,c) \le 0$. If $\Delta(a,b,c) > 0$, then the largest of the three
entries is unique and exceeds the sum of the other two. After permuting the
entries, we may assume that this largest entry is $a$, so that $a > b+c$. Set
\[
  \mathfrak{m}(a,b,c) \;=\;
  2^{m}\,
  \frac{a^a\, b^b\, c^c\, (a-b-c)^{a-b-c}}
       {(a-b+c)^{a-b+c}\,(a+b-c)^{a+b-c}},
  \qquad 0^0 = 1 .
\]
The right-hand side is symmetric in $b$ and $c$ and is therefore independent
of the ordering of the two smaller entries. Hence this construction defines
$\mathfrak{m}$ as a symmetric function on all of $\Rge^{3}$.
\end{definition}

\medskip
\noindent
For $s,i\in\{1,2,3\}$, let $j,k$ denote, in either order, the two indices
different from $i$. Specializing Definition~\ref{def:triangular}, the
corresponding triangular ratio is
\[
 R_{\vrho(\ve_s;i\mid jk)}
 =\frac{P_{\ve_s + 2\ve_i}\; P_{\ve_s + \ve_j + \ve_k}}
       {P_{\ve_s + \ve_i + \ve_j}\; P_{\ve_s + \ve_i + \ve_k}}.
\]
As $s$ and $i$ vary, Corollary~\ref{cor:threevariable} shows that these
exponent vectors are the nine generators of $\BR_{\mathring L}(3,3)$.
\medskip

\begin{definition}[{\rm \defn{One-Hessian logarithmic ratios}}]
\label{def:one-hessian-data}
For a symmetric $3\times3$ matrix $\mathcal H$ with positive entries, define
\[
 T_i(\mathcal H)=\frac{\mathcal H_{ii}\,\mathcal H_{jk}}
                      {\mathcal H_{ij}\,\mathcal H_{ik}},
 \qquad
 i\in\{1,2,3\},\quad
 \{j,k\}=\{1,2,3\}\setminus\{i\}.
\]
The quantity
\[
 t_i(\mathcal H)=\log T_i(\mathcal H)
\]
is the \defn{logarithmic triangular ratio} of $\mathcal H$ at $i$. Put
\[
 t(\mathcal H)=\bigl(t_1(\mathcal H),t_2(\mathcal H),t_3(\mathcal H)\bigr)
\]
and let $\mathcal C$ be the set of all $t(\mathcal H)$ as $\mathcal H$
ranges over symmetric $3\times3$ matrices with positive entries and
Lorentzian signature.

For a set $A\subseteq\mathbb R^d$, its \defn{support function} is
\[
 \sigma_A(w)=\sup_{x\in A}\langle w,x\rangle.
\]
Thus $\sigma_A(w)$ is the supremum of the corresponding weighted sums of
the coordinates of points of $A$.
\end{definition}

For cubics with positive coefficients, strict Lorentzianity is equivalent to
requiring every first-derivative Hessian to have Lorentzian signature. The
next two lemmas describe the logarithmic triangular ratios of a single such
Hessian.

\begin{lemma}[Strict-triangle parametrization]
\label{lem:strict-triangle-model}
Positive diagonal congruence preserves the triangular ratios and the
signature. After normalizing the diagonal of a symmetric $3\times3$ matrix
with positive entries to $1$, the matrix has Lorentzian signature if and
only if its off-diagonal entries are
\[
 \cosh\ell_{12},\qquad \cosh\ell_{13},\qquad \cosh\ell_{23},
\]
where $\ell_{12},\ell_{13},\ell_{23}$ are the side lengths of a strict
triangle.
\end{lemma}

\begin{proof}
To verify the invariance of the triangular ratios, let
$D=\operatorname{diag}(d_1,d_2,d_3)$ with every $d_i>0$. Then
\[
 \frac{(D\mathcal HD)_{ii}(D\mathcal HD)_{jk}}
      {(D\mathcal HD)_{ij}(D\mathcal HD)_{ik}}
 =
 \frac{d_i^2h_{ii}\,d_jd_kh_{jk}}
      {d_id_jh_{ij}\,d_id_kh_{ik}}
 =T_i(\mathcal H).
\]
Signature invariance follows from Sylvester's law of inertia. Taking
$d_i=h_{ii}^{-1/2}$ and retaining the notation $\mathcal H=(h_{uv})$, we may
therefore assume that $h_{11}=h_{22}=h_{33}=1$.

Fix a principal $2\times2$ submatrix of the normalized matrix.
By Cauchy interlacing, its smaller eigenvalue is at most the second-largest
eigenvalue of $\mathcal H$, which is negative. Its trace is $2$, so its
other eigenvalue is positive. Its determinant is therefore negative. Thus
\[
 1-h_{uv}^2<0.
\]
Since the entries are positive, we have $h_{uv}>1$. Consequently, there is a
unique number $\ell_{uv}>0$ such that
\[
 h_{uv}=\cosh\ell_{uv}.
\]
The Lorentzian signature of $\mathcal H$ also gives $\det\mathcal H>0$.

Expanding and factoring the determinant gives
\[
 \det\mathcal H
 =-\bigl(\cosh\ell_{23}-\cosh(\ell_{12}+\ell_{13})\bigr)
   \bigl(\cosh\ell_{23}-\cosh(\ell_{12}-\ell_{13})\bigr).
\]
It follows that $\det\mathcal H>0$ precisely when
\[
 \cosh|\ell_{12}-\ell_{13}|<\cosh\ell_{23}
 <\cosh(\ell_{12}+\ell_{13}).
\]
Since $\cosh$ is strictly increasing on $[0,\infty)$, this is equivalent to
\[
 |\ell_{12}-\ell_{13}|<\ell_{23}<\ell_{12}+\ell_{13}.
\]
These are precisely the three strict triangle inequalities for
$\ell_{12},\ell_{13},\ell_{23}$.

Conversely, suppose that $\ell_{12},\ell_{13},\ell_{23}>0$ satisfy the strict
triangle inequalities. Form the normalized symmetric matrix with
off-diagonal entries $\cosh\ell_{12}$, $\cosh\ell_{13}$, and
$\cosh\ell_{23}$. The factorization above gives $\det\mathcal H>0$. The
matrix is not positive definite because each of its principal $2\times2$
minors is
\[
 1-\cosh^2\ell_{uv}<0.
\]
Since the matrix has positive determinant, it is nonsingular and has an even
number of negative eigenvalues. That number cannot be zero because the
matrix is not positive definite. It therefore has exactly one positive and
two negative eigenvalues.
We conclude that the normalized matrices with positive entries and
Lorentzian signature are exactly those obtained from strict triangles in
this way.
\end{proof}

\begin{lemma}[The one-Hessian region]
\label{lem:one-hessian-region}
The set $\mathcal C$ is nonempty, open, and convex. Its support function is
\[
 \sigma_{\mathcal C}(a,b,c)=
 \begin{cases}
  \log\mathfrak m(a,b,c),&a,b,c\ge0,\\
  +\infty,&\text{otherwise.}
 \end{cases}
\]
\end{lemma}

\begin{proof}
\emph{(i) Logarithmic triangle coordinates.}
Normalize the diagonal of $\mathcal H=(h_{uv})$ to $1$ as in
Lemma~\ref{lem:strict-triangle-model}, and put
$q_{uv}=\log h_{uv}>0$. Thus
$q_{uv}=\log\cosh\ell_{uv}$. We first rewrite the triangle inequality
$\ell_{23}<\ell_{12}+\ell_{13}$ in these logarithmic coordinates. The
strict monotonicity of $\ell\mapsto\log\cosh\ell$ on $[0,\infty)$ gives
\[
 \ell_{23}<\ell_{12}+\ell_{13}
 \quad\Longleftrightarrow\quad
 q_{23}<\log\cosh(\ell_{12}+\ell_{13}).
\]
It remains to express the right-hand side in terms of $q_{12}$ and $q_{13}$.
We have $\cosh\ell_{12}=e^{q_{12}}$ and
$\cosh\ell_{13}=e^{q_{13}}$. Since $\ell_{12},\ell_{13}>0$, it follows that
\[
 \sinh\ell_{12}=e^{q_{12}}\sqrt{1-e^{-2q_{12}}},
 \qquad
 \sinh\ell_{13}=e^{q_{13}}\sqrt{1-e^{-2q_{13}}}.
\]
The hyperbolic addition formula therefore yields
\begin{align*}
 \cosh(\ell_{12}+\ell_{13})
 &=\cosh\ell_{12}\cosh\ell_{13}
   +\sinh\ell_{12}\sinh\ell_{13}\\
 &=e^{q_{12}+q_{13}}
   \Bigl(1+\sqrt{(1-e^{-2q_{12}})(1-e^{-2q_{13}})}\Bigr),
\end{align*}
and hence
\[
 \log\cosh(\ell_{12}+\ell_{13})
 =q_{12}+q_{13}
  +\log\Bigl(1+\sqrt{(1-e^{-2q_{12}})(1-e^{-2q_{13}})}\Bigr).
\]
We denote this expression by
\[
 \psi(q_{12},q_{13})
 :=q_{12}+q_{13}
   +\log\Bigl(1+\sqrt{(1-e^{-2q_{12}})(1-e^{-2q_{13}})}\Bigr).
\]
Thus $q_{23}<\psi(q_{12},q_{13})$. In other words,
$\psi(q_{12},q_{13})$ is precisely the boundary value of $q_{23}$
corresponding to the degenerate triangle
$\ell_{23}=\ell_{12}+\ell_{13}$.

The key analytic point is the concavity of $\psi$. To verify it directly,
set
\[
 \xi=\sqrt{1-e^{-2q_{12}}},\qquad
 \eta=\sqrt{1-e^{-2q_{13}}},\qquad
 d=1+\xi\eta.
\]
A direct differentiation gives
\[
 \frac{\partial^2\psi}{\partial q_{12}^2}
 =-\frac{\eta(1-\xi^2)(1+\xi^2+2\xi\eta)}
         {\xi^3d^2}<0
\]
and
\[
 \det\operatorname{Hess}\psi
 =\frac{2(1-\xi^2)(1-\eta^2)(\xi+\eta)^2}
        {\xi^2\eta^2d^3}>0.
\]
Thus the Hessian of $\psi$ is negative definite, so $\psi$ is strictly
concave on $\mathbb R_{>0}^2$. Applying the same calculation cyclically,
the set
\[
 \mathcal Q=\Bigl\{q\in\mathbb R_{>0}^3:
 q_{23}<\psi(q_{12},q_{13}),\quad
 q_{13}<\psi(q_{12},q_{23}),\quad
 q_{12}<\psi(q_{13},q_{23})\Bigr\}
\]
is open. It is convex because the region below the graph of a concave
function is convex, and $\mathcal Q$ is the intersection of three such
regions.

\emph{(ii) The region $\mathcal C$.}
With the diagonal normalized, put
$t(\mathcal H)=(t_1,t_2,t_3)$. Then
\[
 \begin{pmatrix}t_1\\t_2\\t_3\end{pmatrix}
 =
 \begin{pmatrix}-1&-1&1\\-1&1&-1\\1&-1&-1\end{pmatrix}
 \begin{pmatrix}q_{12}\\q_{13}\\q_{23}\end{pmatrix}.
\]
This map is invertible. Explicitly,
\[
 q_{12}=-\frac{t_1+t_2}{2},\qquad
 q_{13}=-\frac{t_1+t_3}{2},\qquad
 q_{23}=-\frac{t_2+t_3}{2}.
\]
Lemma~\ref{lem:strict-triangle-model} now shows that $\mathcal C$ is the
invertible linear image of $\mathcal Q$. In particular, $\mathcal C$ is a
nonempty open convex subset of $\mathbb R^3$.

\emph{(iii) The support function.}
Definition~\ref{def:one-hessian-data} gives
\[
 \sigma_{\mathcal C}(a,b,c)
 =
 \sup_{\mathcal H}
 \log\bigl(T_1(\mathcal H)^aT_2(\mathcal H)^b
                    T_3(\mathcal H)^c\bigr),
\]
where $\mathcal H$ ranges over symmetric matrices with positive entries and
Lorentzian signature.
For $a,b,c\ge0$ with $a+b+c=1$, \cite[Theorem~C]{HHSW} identifies this
supremum with $\log\mathfrak m(a,b,c)$. The matrix class in the cited theorem
is the closure of our strict-signature locus. By
Lemma~\ref{lem:strict-triangle-model}, this closure is obtained by replacing the strict
triangle inequalities by weak ones. Every weak triangle is a limit of
strict triangles. Hence the supremum is unchanged.

Now let $w=(a,b,c)\ge0$ be nonzero and put $M=a+b+c$. The positive
homogeneity of the support function and of $\log\mathfrak m$ gives
\[
 \sigma_{\mathcal C}(w)
 =M\,\sigma_{\mathcal C}(w/M)
 =\log\mathfrak m(w).
\]
Here
\[
 \log\mathfrak m(\tau w)=\tau\log\mathfrak m(w)
 \qquad(\tau\ge0).
\]
The latter homogeneity follows directly from the formula in
Definition~\ref{def:m-function}.
For $\tau>0$, scaling by $\tau$ preserves the sign of $\Delta$, and all
powers of $\tau$ cancel between the numerator and denominator. At $w=0$,
both sides are zero.
Finally, \cite{HHSW} shows that the supremum is finite exactly when all three
coordinates of $w$ are nonnegative. This proves the displayed formula in
the statement.

In particular, for every $t=(t_1,t_2,t_3)\in\mathcal C$ and every
$(a,b,c)\ge0$,
\begin{equation}\label{eq:one-hessian-bound}
 at_1+bt_2+ct_3
 \le \log\mathfrak m(a,b,c).
\end{equation}
We will apply this one-row inequality to each of the three Hessians of a
ternary cubic.
\end{proof}

\begin{lemma}\label{lem:m-at-least-one}
For every $a,b,c\ge0$,
\[
 \mathfrak m(a,b,c)\ge1.
\]
\end{lemma}

\begin{proof}
For $\varepsilon>0$, let $\mathcal H_\varepsilon$ be the symmetric matrix
whose diagonal entries are $1$ and whose off-diagonal entries all equal
$\cosh\varepsilon$. The triple
$(\varepsilon,\varepsilon,\varepsilon)$ satisfies the strict triangle
inequalities, so Lemma~\ref{lem:strict-triangle-model} shows that
$\mathcal H_\varepsilon$ has Lorentzian signature. Its three triangular
ratios are all $1/\cosh\varepsilon$. Applying
\eqref{eq:one-hessian-bound} to $\mathcal H_\varepsilon$ gives
\[
 -(a+b+c)\log\cosh\varepsilon
 \le \log\mathfrak m(a,b,c).
\]
Letting $\varepsilon\to0^+$ yields
$\log\mathfrak m(a,b,c)\ge0$, and hence
$\mathfrak m(a,b,c)\ge1$.
\end{proof}

\subsection{Compatible Hessian triples and convex duality}

\begin{definition}[{\rm \defn{Compatible Hessian arrays}}]
\label{def:compatible-hessian-arrays}
For a ternary cubic $F$ with positive coefficients and coefficient vector
$\vP=(P_{\va})_{\va\in H(3,3)}$, put
\[
 \widehat P_{\va}=\va!\,P_{\va},
 \qquad
 \mathcal H^{(\ve_s)}=\operatorname{Hess}(\partial_sF).
\]
For $s,i\in\{1,2,3\}$, define
\[
 T_{(\ve_s;i\mid jk)}=T_i(\mathcal H^{(\ve_s)}),
\]
where $j,k$ are the two indices different from $i$.

Define
\[
 L\colon\mathbb R^{H(3,3)}\longrightarrow\mathbb R^{3\times3},
 \qquad
 L(z)_{s,i}=\langle\vrho(\ve_s;i\mid jk),z\rangle,
\]
and
\[
 L^*\colon\mathbb R^{3\times3}\longrightarrow\mathbb R^{H(3,3)},
 \qquad
 L^*(\lambda)=\sum_{s,i}\lambda_{s,i}\cdot\vrho(\ve_s;i\mid jk).
\]
We identify $\mathbb R^{3\times3}$ with $\mathbb R^9$ and denote its
standard basis by $E_{s,i}$. Let
\[
 S=
 \left\{x\in\mathbb R^{3\times3}:
 \begin{aligned}
  x_{2,1}+x_{2,3}&=x_{1,3}+x_{3,1},\\
  x_{1,2}+x_{1,3}&=x_{2,3}+x_{3,2}
 \end{aligned}
 \right\}.
\]
Finally, let
\[
 \mathcal Z=
 \Bigl\{(\log\widehat P_{\va})_{\va\in H(3,3)}:
 \vP\in\mathring L(3,3)\Bigr\}.
\]
We regard $\mathcal C^3$ as the set of $3\times3$ arrays whose three rows
belong to $\mathcal C$.
\end{definition}

Because the three Hessians share the same normalized coefficients, their
rows of logarithmic triangular ratios cannot be chosen independently. The
next lemma characterizes exactly which triples of rows can occur together.

\begin{lemma}[Compatible Hessian triples]
\label{lem:compatible-hessian-triples}
We have
\[
 \operatorname{im}L=S,
 \qquad
 S^\perp=\ker L^*,
 \qquad
 L(\mathcal Z)=S\cap\mathcal C^3.
\]
\end{lemma}

\begin{proof}
\emph{(i) The compatibility equations.}
For every $\lambda\in\mathbb R^{3\times3}$ and
$z\in\mathbb R^{H(3,3)}$, we have
\[
\begin{aligned}
 \langle\lambda,L(z)\rangle
 &=\sum_{s,i}\lambda_{s,i}L(z)_{s,i}\\
 &=\sum_{s,i}\lambda_{s,i}
   \langle\vrho(\ve_s;i\mid jk),z\rangle\\
 &=\left\langle
   \sum_{s,i}\lambda_{s,i}\cdot\vrho(\ve_s;i\mid jk),z
   \right\rangle\\
 &=\langle L^*(\lambda),z\rangle.
\end{aligned}
\]
The first equality expands the Euclidean inner product on
$\mathbb R^{3\times3}$. The remaining equalities follow from the definitions
of $L$ and $L^*$ and the linearity of the inner product. Thus $L^*$ is the
adjoint of $L$, and
\[
 \operatorname{im}L=(\ker L^*)^\perp.
\]
Here $L(z)$ is a $3\times3$ array indexed by $(s,i)$, whereas
$L^*(\lambda)$ is a vector indexed by $\va\in H(3,3)$. In particular,
$\ker L^*$ records the linear relations among the nine exponent vectors.

The nine exponent vectors satisfy the following two relations:
\[
 \begin{aligned}
 \vrho(\ve_2;1\mid23)+\vrho(\ve_2;3\mid12)
 &=\vrho(\ve_1;3\mid12)+\vrho(\ve_3;1\mid23),\\
 \vrho(\ve_1;2\mid13)+\vrho(\ve_1;3\mid12)
 &=\vrho(\ve_2;3\mid12)+\vrho(\ve_3;2\mid13).
 \end{aligned}
\]
These relations are independent. By Lemma~\ref{lem:Mpointed} and
Theorem~\ref{thm:BRL-duality},
$\BR_{\mathring L}(3,3)$ is full-dimensional in $V_{3,3}$. This space has
dimension $|H(3,3)|-3=10-3=7$. Corollary~\ref{cor:threevariable} shows that
the nine exponent vectors generate this cone, so they span a
$7$-dimensional space. Hence there are exactly two independent relations,
and the displayed ones form a basis of the entire relation space:
\begin{equation}\label{eq:kernel-Lstar}
 \ker L^*
 =
 \operatorname{span}\left\{
 \begin{aligned}
  &E_{1,3}-E_{2,1}-E_{2,3}+E_{3,1},\\
  &-E_{1,2}-E_{1,3}+E_{2,3}+E_{3,2}
 \end{aligned}\right\}.
\end{equation}
The orthogonal complement of this space is precisely the subspace $S$
defined by the two compatibility equations. Therefore
$\operatorname{im}L=S$ and $S^\perp=\ker L^*$.

For later reference, the $(s,i)$-entry of $L(z)$ is
\[
 L(z)_{s,i}
 =z_{\ve_s+2\ve_i}+z_{\ve_s+\ve_j+\ve_k}
  -z_{\ve_s+\ve_i+\ve_j}-z_{\ve_s+\ve_i+\ve_k}.
\]
For a cubic $F$, direct differentiation gives
\[
 \mathcal H^{(\ve_s)}_{uv}
 =\widehat P_{\ve_s+\ve_u+\ve_v}.
\]
Consequently, when $z=\log\widehat P$, the array $L(z)$ consists of the
logarithmic triangular ratios. Equivalently, the ratios themselves satisfy
\[
\begin{aligned}
 T_{(\ve_2;1\mid23)}T_{(\ve_2;3\mid12)}
 &=T_{(\ve_1;3\mid12)}T_{(\ve_3;1\mid23)},\\
 T_{(\ve_1;2\mid13)}T_{(\ve_1;3\mid12)}
 &=T_{(\ve_2;3\mid12)}T_{(\ve_3;2\mid13)}.
\end{aligned}
\]

\emph{(ii) The exact image.}
For the forward inclusion, take
$z=\log\widehat P\in\mathcal Z$. For every $s$ and $i$, we have
\[
 L(z)_{s,i}
 =\langle\vrho(\ve_s;i\mid jk),\log\widehat P\rangle
 =\log\widehat P^{\,\vrho(\ve_s;i\mid jk)}
 =\log T_{(\ve_s;i\mid jk)}.
\]
Thus the $s$-th row of $L(z)$ is the triple of logarithmic triangular ratios
of $\mathcal H^{(\ve_s)}$. Since the cubic is strictly Lorentzian, each
$\mathcal H^{(\ve_s)}$ has Lorentzian signature, and hence each row belongs to
$\mathcal C$. We also have $L(z)\in S$ because
$\operatorname{im}L=S$. Therefore
$L(z)\in S\cap\mathcal C^3$.

For the reverse inclusion, take $x\in S\cap\mathcal C^3$. Since
$S=\operatorname{im}L$, there is some $z\in\mathbb R^{H(3,3)}$ such that
\[
 L(z)=x.
\]
The vector $z$ is, at this stage, an arbitrary preimage of $x$. We show that
it is the vector of logarithms of the normalized coefficients of a strictly
Lorentzian cubic.
Define
\[
 \widehat P_{\va}=e^{z_{\va}},
 \qquad
 P_{\va}=\frac{\widehat P_{\va}}{\va!}.
\]
These coefficients are positive and define the cubic
$F=\sum_{\va\in H(3,3)}P_{\va}\vx^{\va}$. For each $s$, let
$\mathcal H^{(\ve_s)}=\operatorname{Hess}(\partial_sF)$. Then
\[
 \mathcal H^{(\ve_s)}_{uv}
 =\widehat P_{\ve_s+\ve_u+\ve_v}.
\]
The matrix $\mathcal H^{(\ve_s)}$ is symmetric because the expression on the
right is unchanged when $u$ and $v$ are interchanged, and all of its entries
are positive.

Since $\widehat P_{\va}=e^{z_{\va}}$ and $L(z)=x$, the three logarithmic
triangular ratios of $\mathcal H^{(\ve_s)}$ are
\[
\begin{aligned}
 \log T_{(\ve_s;1\mid23)}
 &=\log\frac{
   \widehat P_{\ve_s+2\ve_1}
   \widehat P_{\ve_s+\ve_2+\ve_3}}
  {\widehat P_{\ve_s+\ve_1+\ve_2}
   \widehat P_{\ve_s+\ve_1+\ve_3}}
 =L(z)_{s,1}=x_{s,1},\\
 \log T_{(\ve_s;2\mid13)}
 &=\log\frac{
   \widehat P_{\ve_s+2\ve_2}
   \widehat P_{\ve_s+\ve_1+\ve_3}}
  {\widehat P_{\ve_s+\ve_1+\ve_2}
   \widehat P_{\ve_s+\ve_2+\ve_3}}
 =L(z)_{s,2}=x_{s,2},\\
 \log T_{(\ve_s;3\mid12)}
 &=\log\frac{
   \widehat P_{\ve_s+2\ve_3}
   \widehat P_{\ve_s+\ve_1+\ve_2}}
  {\widehat P_{\ve_s+\ve_1+\ve_3}
   \widehat P_{\ve_s+\ve_2+\ve_3}}
 =L(z)_{s,3}=x_{s,3}.
\end{aligned}
\]
Thus the $s$-th row of $x$ is explicitly
\[
 (x_{s,1},x_{s,2},x_{s,3})
 =\bigl(\log T_{(\ve_s;1\mid23)},
         \log T_{(\ve_s;2\mid13)},
         \log T_{(\ve_s;3\mid12)}\bigr).
\]

It remains to show that the constructed cubic is strictly Lorentzian. Fix
$s$ and put
\[
 T_i=e^{x_{s,i}},\qquad i=1,2,3.
\]
The $s$-th row of $x$ belongs to $\mathcal C$. By the definition of
$\mathcal C$, there is a positive symmetric matrix $\mathcal K$ of Lorentzian
signature whose three triangular ratios are $T_1,T_2,T_3$.

Normalize the diagonals of both $\mathcal H^{(\ve_s)}$ and $\mathcal K$ to $1$
by positive diagonal congruence. Lemma~\ref{lem:strict-triangle-model}
shows that this operation preserves both the triangular ratios and the
signature. For any symmetric matrix with positive entries and diagonal
entries $1$,
write
\[
 a=h_{12},\qquad b=h_{13},\qquad c=h_{23}.
\]
Its triangular ratios satisfy
\[
 T_1=\frac{c}{ab},\qquad
 T_2=\frac{b}{ac},\qquad
 T_3=\frac{a}{bc},
\]
and therefore
\[
 a=(T_1T_2)^{-1/2},\qquad
 b=(T_1T_3)^{-1/2},\qquad
 c=(T_2T_3)^{-1/2}.
\]
Thus the three ratios uniquely determine a symmetric matrix with positive
entries and diagonal entries $1$. The normalized matrices obtained from
$\mathcal H^{(\ve_s)}$ and $\mathcal K$ have the same three ratios, so they
coincide. The normalized form of $\mathcal H^{(\ve_s)}$ therefore has Lorentzian
signature. Positive diagonal congruence preserves the numbers of positive
and negative eigenvalues, so the original matrix $\mathcal H^{(\ve_s)}$ has the
same signature.

This argument applies to each $s\in\{1,2,3\}$. Hence the cubic with
coefficients $P_{\va}$ is strictly Lorentzian, which means that
$z\in\mathcal Z$. Since
$L(z)=x$, we conclude that $x\in L(\mathcal Z)$. This proves
\begin{equation}\label{eq:three-hessian-image}
 L(\mathcal Z)=S\cap\mathcal C^3.
\end{equation}
\end{proof}

\begin{definition}[{\rm \defn{Optimization data}}]
\label{def:ternary-optimization-data}
Put
\[
 \mathcal D=(\overline{\mathcal C})^3.
\]
For $\vg\in\BR_{\mathring L}(3,3)$, define
\[
 \Lambda(\vg)
 =\{\lambda\in\Rge^{\,3\times3}:L^*(\lambda)=\vg\}
\]
and, for $\lambda\in\Rge^{\,3\times3}$,
\[
 \Phi(\lambda)
 =\sum_{s=1}^3
 \log\mathfrak m(\lambda_{s,1},\lambda_{s,2},\lambda_{s,3}).
\]
Thus $\Lambda(\vg)$ is the set of all nonnegative representations
\[
 \vg=\sum_{s=1}^{3}\bigl(
    \lambda_{s,1}\cdot\vrho(\ve_s;1\mid23)
   +\lambda_{s,2}\cdot\vrho(\ve_s;2\mid13)
   +\lambda_{s,3}\cdot\vrho(\ve_s;3\mid12)
 \bigr).
\]

For a closed convex set $A\subseteq\mathbb R^{3\times3}$, its
\defn{indicator function} is
\[
 \delta_A(y)=
 \begin{cases}
  0,&y\in A,\\
  +\infty,&y\notin A.
 \end{cases}
\]
An extended-real function is \defn{proper} if it never takes the value
$-\infty$ and is finite at some point. For an extended-real convex function
$f$, its \defn{Fenchel conjugate} is
\[
 f^*(w)=\sup_y\bigl(\langle w,y\rangle-f(y)\bigr).
\]
\end{definition}

\begin{lemma}[Compatible-row optimization]
\label{lem:compatible-row-optimization}
The set $S\cap\mathcal C^3$ is dense in $S\cap\mathcal D$. Moreover, for
$\vg\in\BR_{\mathring L}(3,3)$ and
$\lambda^0\in\Lambda(\vg)$,
\[
 \sigma_{S\cap\mathcal C^3}(\lambda^0)
 =\inf_{\lambda\in\Lambda(\vg)}\Phi(\lambda).
\]
\end{lemma}

\begin{proof}
\emph{(i) Passing to the closure.}
We first establish the density assertion and, in the process, the
interior-point condition needed for convex duality.
Because $\mathcal C$ is nonempty, open, and convex,
$\operatorname{int}\overline{\mathcal C}=\mathcal C$, and hence
$\operatorname{int}\mathcal D=\mathcal C^3$. We begin by constructing a point in
$S\cap\mathcal C^3$. Consider the positive symmetric matrix
\[
 \mathcal H^\circ=
 \begin{pmatrix}
  1&2&2\\
  2&1&2\\
  2&2&1
 \end{pmatrix}.
\]
The vector $(1,1,1)$ is an eigenvector with eigenvalue $5$. Every vector
whose coordinates sum to $0$ is an eigenvector with eigenvalue $-1$.
Thus the eigenvalues are $5,-1,-1$, so $\mathcal H^\circ$ has Lorentzian
signature. For each $i\in\{1,2,3\}$, its triangular ratio is
\[
 T_i(\mathcal H^\circ)=\frac{1\cdot2}{2\cdot2}=\frac12.
\]
Hence its three logarithmic triangular ratios form the vector
\[
 t(\mathcal H^\circ)=(-\log2,-\log2,-\log2)\in\mathcal C.
\]
Now define
\[
 x^\circ=-\log2
 \begin{pmatrix}
  1&1&1\\
  1&1&1\\
  1&1&1
 \end{pmatrix}.
\]
Every row of $x^\circ$ equals $t(\mathcal H^\circ)$, so
$x^\circ\in\mathcal C^3$. Because all its entries are equal to $-\log2$,
the two compatibility equations defining $S$ become
\[
\begin{aligned}
 x^\circ_{2,1}+x^\circ_{2,3}
 &=x^\circ_{1,3}+x^\circ_{3,1}=-2\log2,\\
 x^\circ_{1,2}+x^\circ_{1,3}
 &=x^\circ_{2,3}+x^\circ_{3,2}=-2\log2.
\end{aligned}
\]
Thus
\[
 x^\circ\in S\cap\mathcal C^3
 =S\cap\operatorname{int}\mathcal D.
\]

We claim that
\begin{equation}\label{eq:section-closure}
 \overline{S\cap\mathcal C^3}=S\cap\mathcal D.
\end{equation}
The forward inclusion is immediate because $S\cap\mathcal D$ is closed and
contains $S\cap\mathcal C^3$. For the reverse inclusion, fix
$x\in S\cap\mathcal D$ and $0<\varepsilon<1$, and set
\[
 x_\varepsilon=(1-\varepsilon)x+\varepsilon x^\circ.
\]
Because $S$ is a linear subspace and both $x$ and $x^\circ$ belong to $S$, we
have $x_\varepsilon\in S$. Since
$x^\circ\in\operatorname{int}\mathcal D$, we may choose $r>0$ such that
$B(x^\circ,r)\subseteq\mathcal D$. Moreover,
\[
\begin{aligned}
 (1-\varepsilon)x+\varepsilon B(x^\circ,r)
 &=\{x_\varepsilon+\varepsilon u:\|u\|<r\}\\
 &=B(x_\varepsilon,\varepsilon r).
\end{aligned}
\]
Every point in the first set is a convex combination of $x\in\mathcal D$ and
a point of $B(x^\circ,r)\subseteq\mathcal D$. Hence the convexity of
$\mathcal D$ gives
\[
 B(x_\varepsilon,\varepsilon r)\subseteq\mathcal D.
\]
Thus a full open ball around $x_\varepsilon$ lies in $\mathcal D$, so
$x_\varepsilon\in\operatorname{int}\mathcal D=\mathcal C^3$. Together with
$x_\varepsilon\in S$, this yields
\[
 x_\varepsilon\in S\cap\mathcal C^3.
\]
Finally,
\[
 \|x_\varepsilon-x\|=\varepsilon\|x^\circ-x\|\longrightarrow0
 \qquad\text{as }\varepsilon\longrightarrow0^+.
\]
This proves \eqref{eq:section-closure}. In particular, every linear
functional has the same supremum on $S\cap\mathcal C^3$ and
$S\cap\mathcal D$.

\emph{(ii) Weak duality.}
Since $L^*(\lambda^0)=\vg$ and
$S^\perp=\ker L^*$, all real representations of $\vg$ are
\[
 \lambda^0+S^\perp
 =\{\lambda\in\mathbb R^{3\times3}:L^*(\lambda)=\vg\}.
\]
If $\eta\in S^\perp$ and $x\in S$, then
$\langle\eta,x\rangle=0$. Hence
\[
 \langle\lambda^0,x\rangle
 =\langle\lambda^0+\eta,x\rangle
 \le\sigma_{\mathcal D}(\lambda^0+\eta)
 \qquad(x\in S\cap\mathcal D).
\]
Taking the supremum over $x\in S\cap\mathcal D$ and then the infimum over
$\eta\in S^\perp$ gives
\[
 \sigma_{S\cap\mathcal D}(\lambda^0)
 \leq\inf_{\eta\in S^\perp}\sigma_{\mathcal D}(\lambda^0+\eta).
\]

\emph{(iii) Equality by convex duality.}
It follows directly from Definition~\ref{def:ternary-optimization-data} that
\[
 \delta_A^*(w)=\sup_{y\in A}\langle w,y\rangle=\sigma_A(w).
\]
We apply the Fenchel--Rockafellar sum theorem. In the form needed here, it
states that if $f$ and $g$ are proper lower semicontinuous convex functions
and $g$ is continuous at a point where $f$ is finite, then
\[
 (f+g)^*(w)
 =\inf_u\bigl(f^*(u)+g^*(w-u)\bigr).
\]
The hypotheses hold for $f=\delta_S$ and $g=\delta_{\mathcal D}$. Indeed,
$S$ and $\mathcal D$ are nonempty closed convex sets, so their indicator
functions are proper lower semicontinuous and convex. Part~(i) gives
\[
 x^\circ\in S\cap\operatorname{int}\mathcal D.
\]
Hence some open ball around $x^\circ$ is contained in $\mathcal D$. The
function $\delta_{\mathcal D}$ is identically zero on this ball and is
therefore continuous at $x^\circ$. At the same point,
$\delta_S(x^\circ)=0$.

Since
\[
 \delta_S+\delta_{\mathcal D}=\delta_{S\cap\mathcal D},
\]
because both sides are zero on $S\cap\mathcal D$ and $+\infty$ elsewhere,
the sum theorem gives
\[
\begin{aligned}
 \sigma_{S\cap\mathcal D}(\lambda^0)
 &=\delta_{S\cap\mathcal D}^*(\lambda^0)\\
 &=(\delta_S+\delta_{\mathcal D})^*(\lambda^0)\\
 &=\inf_u\bigl(
      \delta_S^*(u)+\sigma_{\mathcal D}(\lambda^0-u)
    \bigr).
\end{aligned}
\]
It remains to compute $\delta_S^*$. By the definition of the conjugate,
\[
 \delta_S^*(u)=\sup_{y\in S}\langle u,y\rangle.
\]
Recall that $S^\perp$ consists of all $u\in\mathbb R^{3\times3}$ such that
$\langle u,y\rangle=0$ for every $y\in S$.
If $u\in S^\perp$, then every inner product in the supremum is zero, so
$\delta_S^*(u)=0$. Now suppose that $u\notin S^\perp$. There is then a
$y_0\in S$ such that $\langle u,y_0\rangle\ne0$. Since $S$ is a linear
subspace, $-y_0\in S$. After replacing $y_0$ by $-y_0$ if necessary, we
may assume that $\langle u,y_0\rangle>0$. Moreover, $t y_0\in S$ for every
$t>0$, and
\[
 \langle u,t y_0\rangle
 =t\langle u,y_0\rangle\longrightarrow+\infty
 \qquad\text{as }t\longrightarrow+\infty.
\]
Consequently,
\[
 \delta_S^*(u)
 =
 \begin{cases}
  0,&u\in S^\perp,\\
  +\infty,&u\notin S^\perp.
 \end{cases}
\]
Thus only $u\in S^\perp$ contributes to the infimum, and
\[
 \sigma_{S\cap\mathcal D}(\lambda^0)
 =\inf_{u\in S^\perp}\sigma_{\mathcal D}(\lambda^0-u).
\]
Finally, $S^\perp$ is a linear subspace, so $-u$ ranges over $S^\perp$ as
$u$ does. Replacing $-u$ by $\eta$ gives
\begin{equation}\label{eq:support-section-duality}
 \sigma_{S\cap\mathcal D}(\lambda^0)
 =\inf_{\eta\in S^\perp}\sigma_{\mathcal D}(\lambda^0+\eta).
\end{equation}
The interior point constructed in part~(i) is precisely what rules out a
duality gap in this equality.

\emph{(iv) Independent row maximizations.}
Because $\mathcal D=(\overline{\mathcal C})^3$, its support function
separates into three independent row maximizations.\footnote{This
independence would not hold on
$S\cap\mathcal D$, because membership in $S$ imposes compatibility equations
between the rows. It becomes available only after convex duality replaces the
constrained support function by an infimum of support functions of
$\mathcal D$.} Taking the closure does not change the row suprema, so
\[
 \sigma_{\mathcal D}(\lambda)
 =\sum_{s=1}^3
   \sigma_{\mathcal C}(\lambda_{s,1},\lambda_{s,2},\lambda_{s,3}).
\]
Lemma~\ref{lem:one-hessian-region} shows that a row supremum is finite
exactly when its three weights are nonnegative. In that case, it is the
logarithm of the corresponding $\mathfrak m$-term. Thus shifts
$\lambda^0+\eta$ with a negative entry contribute $+\infty$ and may be
discarded. The remaining shifts are exactly the nonnegative representations
$\Lambda(\vg)$. Using \eqref{eq:section-closure} and
\eqref{eq:support-section-duality}, we obtain
\begin{equation}\label{eq:dual-log-constant}
 \sup_{x\in S\cap\mathcal C^3}\langle\lambda^0,x\rangle
 =
 \inf_{\lambda\in\Lambda(\vg)}
 \sum_{s=1}^3
 \log\mathfrak m(\lambda_{s,1},\lambda_{s,2},\lambda_{s,3}).
\end{equation}
Let $M$ denote the finite common value in \eqref{eq:dual-log-constant}. For
every $\delta>0$, there is an $x\in S\cap\mathcal C^3$ such that
\[
 \langle\lambda^0,x\rangle>M-\delta.
\]
Thus $M$ is not merely an upper bound: it is the smallest possible upper
bound, even when the supremum is not attained.
\end{proof}

\begin{lemma}[The representation polytope]
\label{lem:representation-polytope}
For every $\vg\in\BR_{\mathring L}(3,3)$, the set $\Lambda(\vg)$ is a
nonempty compact polytope, and $\Phi$ attains its minimum on
$\Lambda(\vg)$.
\end{lemma}

\begin{proof}
The set $\Lambda(\vg)$ is nonempty by
Corollary~\ref{cor:threevariable} and is closed by definition. It remains to
show that it is bounded.

The proof of Lemma~\ref{lem:three-irredundant} shows that
$\mathcal M_{3,3}$ is full-dimensional modulo affine functions. Its dual
$\BR_{\mathring L}(3,3)$ is therefore pointed.

Suppose, to the contrary, that there are
$\lambda^{(r)}\in\Lambda(\vg)$ with
$\|\lambda^{(r)}\|\to\infty$. After passing to a subsequence,
\[
 \mu^{(r)}=\frac{\lambda^{(r)}}{\|\lambda^{(r)}\|}
 \longrightarrow\mu
\]
for some $\mu\ge0$ with $\|\mu\|=1$. Since
$L^*(\lambda^{(r)})=\vg$, division by $\|\lambda^{(r)}\|$ and passage to the
limit give
\[
 L^*(\mu)
 =\sum_{s,i}\mu_{s,i}\cdot\vrho(\ve_s;i\mid jk)=0.
\]
Choose $(s_0,i_0)$ with $\mu_{s_0,i_0}>0$, and let $j_0,k_0$ be the two
indices different from $i_0$. Then
\[
 y=\mu_{s_0,i_0}\cdot\vrho(\ve_{s_0};i_0\mid j_0k_0)
\]
is a nonzero element of $\BR_{\mathring L}(3,3)$, while the preceding
relation gives
\[
 -y=\sum_{(s,i)\ne(s_0,i_0)}
       \mu_{s,i}\cdot\vrho(\ve_s;i\mid jk)
 \in\BR_{\mathring L}(3,3).
\]
This is impossible because the cone is pointed. It cannot contain both a
nonzero vector and its negative. Thus $\Lambda(\vg)$ is bounded. Since it
is defined by finitely many linear equalities and inequalities, it is a
compact polytope.

It remains to prove that the minimum is attained. By
Lemma~\ref{lem:one-hessian-region}, each
summand of $\Phi$ is the restriction to $\Rge^3$ of the support function
$\sigma_{\mathcal C}$. A support function is a supremum of continuous
linear functions, so it is convex and lower semicontinuous. Hence $\Phi$
is lower semicontinuous on the compact set $\Lambda(\vg)$ and attains its
minimum there. Exponentiating shows that the product of the three
$\mathfrak m$-terms also attains its minimum.
\end{proof}

\subsection{Optimal constants and sharp inequalities}

\begin{theorem}[Optimal bounding constants for ternary Lorentzian cubics]
\label{thm:constants33}
For $\vg\in\BR_{\mathring L}(3,3)$, the optimal bounding constant of $\vg$
on the strictly Lorentzian cubics is
\[
 f_{\mathring L}(\vg)
 =
 \Bigl(\prod_{\va\in H(3,3)}(\va!)^{-\gamma_{\va}}\Bigr)
 \min_{\lambda\in\Lambda(\vg)}
 \prod_{s=1}^{3}
 \mathfrak m(\lambda_{s,1},\lambda_{s,2},\lambda_{s,3}).
\]
The supremum defining $f_{\mathring L}(\vg)$ is not attained for
$\vg\ne0$.
\end{theorem}

\begin{proof}
Let $F$ be a strictly Lorentzian ternary cubic with coefficient vector
$\vP$. For each fixed $s$, the three ratios $T_{(\ve_s;i\mid jk)}$ are formed
from the entries of the same Hessian
$\mathcal H^{(\ve_s)}=\operatorname{Hess}(\partial_sF)$. The nine ratios thus
form three rows, one for each value of $s$. These rows cannot be chosen
independently because their entries use the same normalized coefficients.
Concretely,
\[
 \mathcal H^{(\ve_s)}_{uv}
 =\mathcal H^{(\ve_u)}_{sv}
 =\widehat P_{\ve_s+\ve_u+\ve_v}.
\]

We first relate the Hessian ratios to the coefficient ratios. The preceding
identity gives
\[
\begin{aligned}
 T_{(\ve_s;i\mid jk)}
 &=
 \frac{\mathcal H^{(\ve_s)}_{ii}\,\mathcal H^{(\ve_s)}_{jk}}
      {\mathcal H^{(\ve_s)}_{ij}\,\mathcal H^{(\ve_s)}_{ik}}\\
 &=
 \frac{\widehat P_{\ve_s+2\ve_i}\,
       \widehat P_{\ve_s+\ve_j+\ve_k}}
      {\widehat P_{\ve_s+\ve_i+\ve_j}\,
       \widehat P_{\ve_s+\ve_i+\ve_k}}
 =\widehat P^{\,\vrho(\ve_s;i\mid jk)}.
\end{aligned}
\]
This is the triangular ratio of $\mathcal H^{(\ve_s)}$ at $(i\mid jk)$ in the
sense of Definition~\ref{def:triangular}. Specializing the conversion factor
from that definition to $\vb=\ve_s$ gives
\[
 T_{(\ve_s;i\mid jk)}
 =\kappa_{\ve_s,i}R_{\vrho(\ve_s;i\mid jk)},
 \qquad
 \kappa_{\ve_s,i}=
 \begin{cases}
  \tfrac32,&s=i,\\[2pt]
  2,&s\ne i.
 \end{cases}
\]
Indeed, the $i$-th coordinate of $\ve_s$ is $1$ when $s=i$ and is $0$
otherwise. Substituting these two values into the formula
$\kappa_{\vb,i}=(\beta_i+2)/(\beta_i+1)$ gives $3/2$ and $2$, respectively.

Now let $\lambda\in\Lambda(\vg)$. By definition,
\[
 \vg=\sum_{s=1}^3\sum_{i=1}^3
 \lambda_{s,i}\cdot\vrho(\ve_s;i\mid jk),
\]
where in each summand $j,k$ are the two indices different from $i$.
Therefore
\[
\begin{aligned}
 \log\widehat P^{\,\vg}
 &=\sum_{s=1}^3\sum_{i=1}^3
   \lambda_{s,i}\log T_{(\ve_s;i\mid jk)},\\
 \widehat P^{\,\vg}
 &=\prod_{s=1}^3\prod_{i=1}^3
   T_{(\ve_s;i\mid jk)}^{\lambda_{s,i}}.
\end{aligned}
\]
Since $\widehat P_{\va}=\va!P_{\va}$, this gives
\begin{equation}\label{eq:master-hessian-ratio}
 P^{\vg}
 =
 \left(\prod_{\va\in H(3,3)}(\va!)^{-\gamma_{\va}}\right)
 \prod_{s=1}^3\prod_{i=1}^3
 T_{(\ve_s;i\mid jk)}^{\lambda_{s,i}}.
\end{equation}
The factorial correction depends only on $\vg$, not on its representation
$\lambda$.

We next derive the row-by-row upper bound. Put
\[
 z=\log\widehat P,
 \qquad
 x=L(z).
\]
The $s$-th row of $x$ is the triple of logarithmic triangular ratios of
$\mathcal H^{(\ve_s)}$, so it belongs to $\mathcal C$. We have
\[
\begin{aligned}
 \log\widehat P^{\,\vg}
 &=\sum_{s=1}^3\sum_{i=1}^3\lambda_{s,i}x_{s,i}\\
 &=\sum_{s=1}^3
   \bigl(\lambda_{s,1}x_{s,1}
        +\lambda_{s,2}x_{s,2}
        +\lambda_{s,3}x_{s,3}\bigr).
\end{aligned}
\]
Applying \eqref{eq:one-hessian-bound} separately to the three rows gives
\begin{equation}\label{eq:rowwise-upper-bound}
 \log\widehat P^{\,\vg}
 \le
 \sum_{s=1}^3
 \log\mathfrak m(\lambda_{s,1},\lambda_{s,2},\lambda_{s,3}).
\end{equation}
Thus every nonnegative representation of $\vg$ gives a valid upper bound.

We now identify the exact supremum. Fix
$\lambda^0\in\Lambda(\vg)$. For $z=\log\widehat P$ and $x=L(z)$, the adjoint
identity from Lemma~\ref{lem:compatible-hessian-triples} gives
\begin{align*}
 \langle\lambda^0,x\rangle
 &=\langle\lambda^0,L(z)\rangle
 =\langle L^*(\lambda^0),z\rangle\\
 &=\langle\vg,z\rangle\\
 &=\sum_{\va}\gamma_{\va}\log\widehat P_{\va}
 =\log\widehat P^{\,\vg}.
\end{align*}
Thus the quantity being maximized does not depend on the choice of
representation $\lambda^0$. Equation~\eqref{eq:three-hessian-image} shows
that, as $\vP$ ranges over the strictly Lorentzian cubics, $x$ ranges over
all of $S\cap\mathcal C^3$. Hence Lemma~\ref{lem:compatible-row-optimization}
gives
\[
 \sup_{\vP\in\mathring L(3,3)}\log\widehat P^{\,\vg}
 =
 \inf_{\lambda\in\Lambda(\vg)}
 \sum_{s=1}^3
 \log\mathfrak m(\lambda_{s,1},\lambda_{s,2},\lambda_{s,3}).
\]
Lemma~\ref{lem:representation-polytope} shows that this infimum is a
minimum. The identity $\widehat P_{\va}=\va!P_{\va}$ gives
\[
 P^{\vg}
 =\widehat P^{\,\vg}
  \prod_{\va\in H(3,3)}(\va!)^{-\gamma_{\va}}.
\]
Exponentiating proves the formula in the theorem.

Finally, we prove the nonattainment assertion. The set $\mathcal Z$ is an
open subset of $\mathbb R^{H(3,3)}$. Indeed, positivity of the coefficients
is preserved by small perturbations, and a nonsingular symmetric matrix with
one positive and two negative eigenvalues retains these signs under
sufficiently small perturbations. Thus strict Lorentzianity is an open
condition. Multiplying the coordinates by the fixed factorials and then
taking logarithms preserves openness.

Suppose that $\vg\ne0$ and take any $z\in\mathcal Z$. Since $\mathcal Z$ is
open, for every sufficiently small $\varepsilon>0$ we also have
$z+\varepsilon\vg\in\mathcal Z$. But
\[
 \langle\vg,z+\varepsilon\vg\rangle
 =\langle\vg,z\rangle+\varepsilon\|\vg\|^2
 >\langle\vg,z\rangle.
\]
Thus no point of $\mathcal Z$ can maximize the linear function
$z\mapsto\langle\vg,z\rangle$, which is precisely
$\log\widehat P^{\,\vg}$. Passing from $\widehat P^{\,\vg}$ to
$P^{\vg}$ only multiplies the ratio by the fixed factorial correction.
Therefore the supremum defining $f_{\mathring L}(\vg)$ is not attained.
\end{proof}

\begin{corollary}[Nine sharp multiplicative coefficient inequalities]
\label{cor:paired-ratios33}
Let $F$ be a Lorentzian ternary cubic with coefficient vector
$\vP=(P_{\va})_{\va\in H(3,3)}$. Then
\[
\begin{gathered}
 P_{120}P_{201}\le P_{210}P_{111},
 \qquad
 P_{300}P_{111}\le\frac43 P_{210}P_{201},\\[3pt]
 P_{120}P_{102}\le\frac14 P_{111}^{2},
 \qquad
 P_{300}P_{120}\le\frac13 P_{210}^{2},\\[3pt]
 P_{012}P_{120}P_{201}\le\frac4{27}P_{111}^{3},
 \qquad
 P_{300}P_{030}P_{003}\le\frac1{216}P_{111}^{3},\\[8pt]
 (P_{300}P_{030}P_{003})^2
 \le\frac1{729}
 P_{210}P_{201}P_{120}P_{021}P_{102}P_{012},\\[8pt]
 P_{300}P_{030}^{2}P_{003}^{4}
 \le\frac1{2187}P_{120}P_{102}^{2}P_{012}^{4},\\[8pt]
 P_{030}P_{003}P_{210}P_{201}P_{021}
 \le\frac1{18}P_{120}^{2}P_{102}P_{012}P_{111}.\\[5pt]
\end{gathered}
\]
The constants $1$, $\frac43$, $\frac14$, $\frac13$, $\frac4{27}$,
$\frac1{216}$, $\frac1{729}$, $\frac1{2187}$, and $\frac1{18}$ are optimal.
If $F$ is strictly Lorentzian, then all nine inequalities are strict.
In particular, the first and third inequalities are the reverse
Khovanskii--Teissier inequality and the Alexandrov--Fenchel inequality,
respectively.
\end{corollary}

\begin{proof}
We first assume that $F$ is strictly Lorentzian, so all relevant
coefficients are positive. We use the coefficient-to-Hessian conversion
from Definition~\ref{def:triangular}. If
$\lambda\in\Lambda(\vg)$, then
\[
 \prod_{\va\in H(3,3)}(\va!)^{-\gamma_{\va}}
 =
 \prod_{s=1}^3\prod_{i=1}^3
 \kappa_{\ve_s,i}^{-\lambda_{s,i}},
 \qquad
 \kappa_{\ve_s,i}=
 \begin{cases}
  \tfrac32,&s=i,\\
  2,&s\ne i.
 \end{cases}
\]
Indeed,
\[
 \prod_{s,i}T_{(\ve_s;i\mid jk)}^{\lambda_{s,i}}
 =
 \left(\prod_{s,i}\kappa_{\ve_s,i}^{\lambda_{s,i}}\right)
 \prod_{s,i}R_{\vrho(\ve_s;i\mid jk)}(\vP)^{\lambda_{s,i}}
 =
 \left(\prod_{s,i}\kappa_{\ve_s,i}^{\lambda_{s,i}}\right)P^{\vg}.
\]
Comparing this identity with \eqref{eq:master-hessian-ratio} gives the stated
factorial correction.

\emph{(i) The first and second inequalities.}
The two quotients are
\[
 \frac{P_{120}P_{201}}{P_{210}P_{111}}
 =R_{\vrho(\ve_1;2\mid13)}(\vP),
 \qquad
 \frac{P_{300}P_{111}}{P_{210}P_{201}}
 =R_{\vrho(\ve_1;1\mid23)}(\vP).
\]
By Corollary~\ref{cor:threevariable}, these exponent vectors span distinct
extreme rays of $\BR_{\mathring L}(3,3)$. An extreme generator cannot be
expressed as a nonnegative sum of generators on other rays. Consequently,
the only nonnegative representations are $E_{1,2}$ and $E_{1,1}$,
respectively.
Definition~\ref{def:m-function} gives the corresponding minimizing
$\mathfrak m$-products from Theorem~\ref{thm:constants33}:
\[
 \mathfrak m(0,1,0)=2,
 \qquad
 \mathfrak m(1,0,0)=2.
\]
Here the two zero rows contribute
$\mathfrak m(0,0,0)=1$.
For the first quotient, $s=1\ne2=i$, so the factorial correction is
$2^{-1}$. Its optimal constant is therefore
\[
 2\cdot\frac12=1.
\]
For the second quotient, $s=i=1$, so the factorial correction is
$(\frac32)^{-1}=\frac23$. Its optimal constant is therefore
\[
 2\cdot\frac23=\frac43.
\]
The same two computations can be read directly from the Hessian
normalization:
\[
 T_{(\ve_1;2\mid13)}
 =2\frac{P_{120}P_{201}}{P_{210}P_{111}},
 \qquad
 T_{(\ve_1;1\mid23)}
 =\frac32\frac{P_{300}P_{111}}{P_{210}P_{201}}.
\]
Thus the optimal Hessian bound $T<2$ becomes, respectively,
$R<2/2=1$ and $R<2/(\frac32)=\frac43$.

\emph{(ii) The third inequality.}
The triangular ratios give
\[
 \frac{P_{120}P_{102}}{P_{111}^{2}}
 =R_{\vrho(\ve_1;2\mid13)}(\vP)\,
  R_{\vrho(\ve_1;3\mid12)}(\vP).
\]
Thus its exponent vector has a representation whose first row is
$(0,1,1)$ and whose other two rows are zero. Every row has $\Delta\le0$, so
the corresponding $\mathfrak m$-product is $1$.
Lemma~\ref{lem:m-at-least-one} shows that this is the minimum. Both factors
have $s\ne i$, so the factorial correction is
\[
 \frac1{2\cdot2}=\frac14.
\]
Theorem~\ref{thm:constants33} therefore gives the optimal constant
$\frac14$.

\emph{(iii) The fourth inequality.}
Here
\[
 \frac{P_{300}P_{120}}{P_{210}^{2}}
 =R_{\vrho(\ve_1;1\mid23)}(\vP)\,
  R_{\vrho(\ve_1;2\mid13)}(\vP).
\]
The corresponding representation has first row $(1,1,0)$, while its other
two rows are zero. As above, its $\mathfrak m$-product is the minimum value
$1$. The first factor has $s=i$, while the second has $s\ne i$. Hence the
factorial correction is
\[
 \frac1{(\frac32)\cdot2}=\frac13.
\]
Theorem~\ref{thm:constants33} therefore gives the optimal constant
$\frac13$.

\emph{(iv) The fifth inequality.}
The ratio identity is
\[
 \frac{P_{012}P_{120}P_{201}}{P_{111}^{3}}
 =R_{\vrho(\ve_2;3\mid12)}(\vP)\,
  R_{\vrho(\ve_3;1\mid23)}(\vP)\,
  R_{\vrho(\ve_3;2\mid13)}(\vP).
\]
Put
\[
 \vg=\vrho(\ve_2;3\mid12)
     +\vrho(\ve_3;1\mid23)
     +\vrho(\ve_3;2\mid13).
\]
Starting from the representation
\[
 \begin{pmatrix}
  0&0&0\\
  0&0&1\\
  1&1&0
 \end{pmatrix},
\]
every real representation is obtained by adding a linear combination of the
two matrices in the basis of $\ker L^*$ from
\eqref{eq:kernel-Lstar}. Setting the two coefficients equal to $-a$ and
$-b$ gives
\[
 \Lambda(\vg)=
 \left\{
 \begin{pmatrix}
  0&b&b-a\\
  a&0&1+a-b\\
  1-a&1-b&0
 \end{pmatrix}
 :0\le a\le b\le1
 \right\}.
\]
Indeed, the inequalities $0\le a\le b\le1$ are exactly the conditions that
all nine entries in the displayed matrix are nonnegative.
A cyclic permutation of the variables simultaneously permutes the rows and
columns of these matrices. Both $\vg$ and the logarithm of the objective are
invariant under this operation. Hence the two cyclic images of any
$\lambda\in\Lambda(\vg)$ also belong to $\Lambda(\vg)$. By convexity,
averaging $\lambda$ with its two cyclic images cannot increase the logarithm
of the objective. The unique cyclically invariant point
in the displayed fiber is
\[
 \lambda^*=\frac13
 \begin{pmatrix}
  0&2&1\\
  1&0&2\\
  2&1&0
 \end{pmatrix}.
\]
Thus $\lambda^*$ is a minimizer. Its rows are permutations of
$(0,\frac23,\frac13)$. By the homogeneity of $\log\mathfrak m$,
\[
\begin{aligned}
 \prod_{s=1}^3
 \mathfrak m(\lambda^*_{s,1},\lambda^*_{s,2},\lambda^*_{s,3})
 &=\mathfrak m\left(0,\frac23,\frac13\right)^3\\
 &=\mathfrak m(0,2,1)
 =\frac{2^3\cdot2^2}{3^3}
 =\frac{32}{27}.
\end{aligned}
\]
All three factors in the ratio identity have $s\ne i$. Thus the factorial
correction is $2^{-3}=\frac18$, and Theorem~\ref{thm:constants33} gives the
optimal constant
\[
 \frac18\cdot\frac{32}{27}=\frac4{27}.
\]

\emph{(v) The sixth and seventh inequalities.}
Consider the two weight matrices
\[
 \lambda^{(1)}=
 \begin{pmatrix}
  1&1&1\\
  1&1&1\\
  1&1&1
 \end{pmatrix},
 \qquad
 \lambda^{(2)}=
 \begin{pmatrix}
  2&1&1\\
  1&2&1\\
  1&1&2
 \end{pmatrix},
\]
and let
\[
 \vg_r=
 \sum_{s=1}^3\bigl(
  \lambda^{(r)}_{s,1}\cdot\vrho(\ve_s;1\mid23)
 +\lambda^{(r)}_{s,2}\cdot\vrho(\ve_s;2\mid13)
 +\lambda^{(r)}_{s,3}\cdot\vrho(\ve_s;3\mid12)
 \bigr),
 \qquad r=1,2.
\]
Every entry in each matrix is positive, so both representations use all nine
extreme rays. For $\lambda^{(1)}$, multiplying first within each row
and then across the three rows gives
\[
\begin{aligned}
 \prod_{s=1}^3\prod_{i=1}^3
 R_{\vrho(\ve_s;i\mid jk)}(\vP)
 &=
 \frac{P_{300}P_{120}P_{102}}{P_{210}P_{201}P_{111}}\,
 \frac{P_{030}P_{210}P_{012}}{P_{120}P_{021}P_{111}}\,
 \frac{P_{003}P_{201}P_{021}}{P_{102}P_{012}P_{111}}\\
 &=\frac{P_{300}P_{030}P_{003}}{P_{111}^3}.
\end{aligned}
\]
The matrix $\lambda^{(2)}$ adds one more copy of each of the three diagonal
rays. Their product is
\[
\begin{aligned}
 \prod_{s=1}^3 R_{\vrho(\ve_s;s\mid jk)}(\vP)
 &=
 \frac{P_{300}P_{111}}{P_{210}P_{201}}\,
 \frac{P_{030}P_{111}}{P_{120}P_{021}}\,
 \frac{P_{003}P_{111}}{P_{102}P_{012}}\\
 &=
 \frac{P_{300}P_{030}P_{003}P_{111}^3}
 {P_{210}P_{201}P_{120}P_{021}P_{102}P_{012}}.
\end{aligned}
\]
Multiplying the last two displayed expressions proves
\[
 R_{\vg_1}
 =\frac{P_{300}P_{030}P_{003}}{P_{111}^3},
 \qquad
 R_{\vg_2}
 =\frac{(P_{300}P_{030}P_{003})^2}
 {P_{210}P_{201}P_{120}P_{021}P_{102}P_{012}}.
\]
The rows of $\lambda^{(1)}$ are $(1,1,1)$, while the rows of
$\lambda^{(2)}$ are permutations of $(2,1,1)$. Since
\[
 \Delta(1,1,1)=-3,
 \qquad
 \Delta(2,1,1)=-4,
\]
Definition~\ref{def:m-function} gives
$\mathfrak m(1,1,1)=\mathfrak m(2,1,1)=1$.
Lemma~\ref{lem:m-at-least-one} then shows that both displayed representations
minimize the product in Theorem~\ref{thm:constants33}. For
$\lambda^{(1)}$, the product of the three conversion factors in each row is
$(\frac32)\cdot2\cdot2=6$. For $\lambda^{(2)}$, it is
$(\frac32)^2\cdot2\cdot2=9$. The corresponding factorial corrections are
\[
 \frac1{6^3}=\frac1{216},
 \qquad
 \frac1{9^3}=\frac1{729}.
\]
Theorem~\ref{thm:constants33} therefore gives the sixth and seventh constants
and proves their optimality.

\emph{(vi) The eighth inequality.}
Consider the weight matrix
\[
 \lambda=
 \begin{pmatrix}
  1&1&1\\
  1&2&1\\
  1&2&4
 \end{pmatrix}.
\]
Every entry of $\lambda$ is positive, so all nine triangular-ratio
generators occur. Multiplying within each row gives
\[
\begin{aligned}
 \prod_{i=1}^3
 R_{\vrho(\ve_1;i\mid jk)}^{\lambda_{1,i}}
 &=\frac{P_{300}P_{120}P_{102}}
         {P_{210}P_{201}P_{111}},\\
 \prod_{i=1}^3
 R_{\vrho(\ve_2;i\mid jk)}^{\lambda_{2,i}}
 &=\frac{P_{210}P_{030}^{2}P_{012}}
         {P_{120}^{2}P_{021}^{2}},\\
 \prod_{i=1}^3
 R_{\vrho(\ve_3;i\mid jk)}^{\lambda_{3,i}}
 &=\frac{P_{201}P_{021}^{2}P_{003}^{4}P_{111}}
         {P_{102}^{3}P_{012}^{5}}.
\end{aligned}
\]
Multiplying these three expressions and cancelling common factors yields
\[
 \prod_{s=1}^3\prod_{i=1}^3
 R_{\vrho(\ve_s;i\mid jk)}^{\lambda_{s,i}}
 =\frac{P_{300}P_{030}^{2}P_{003}^{4}}
        {P_{120}P_{102}^{2}P_{012}^{4}}.
\]
The rows of $\lambda$ are $(1,1,1)$, $(1,2,1)$, and $(1,2,4)$, and
\[
 \Delta(1,1,1)=-3,
 \qquad
 \Delta(1,2,1)=-4,
 \qquad
 \Delta(1,2,4)=-7.
\]
Hence all three $\mathfrak m$-factors equal $1$. By
Lemma~\ref{lem:m-at-least-one}, this is the minimum in
Theorem~\ref{thm:constants33}. The total diagonal weight is $7$, and the total
off-diagonal weight is also $7$. Thus the factorial correction is
\[
 \left(\frac32\right)^{-7}2^{-7}
 =\frac1{3^7}
 =\frac1{2187}.
\]
Theorem~\ref{thm:constants33} proves the eighth inequality and the optimality
of its constant.

\emph{(vii) The ninth inequality.}
Consider the weight matrix
\[
 \lambda=
 \begin{pmatrix}
  0&0&0\\
  1&1&0\\
  1&1&1
 \end{pmatrix}.
\]
The two nonzero rows give
\[
\begin{aligned}
 R_{\vrho(\ve_2;1\mid23)}
 R_{\vrho(\ve_2;2\mid13)}
 &=\frac{P_{030}P_{210}}{P_{120}^{2}},\\
 R_{\vrho(\ve_3;1\mid23)}
 R_{\vrho(\ve_3;2\mid13)}
 R_{\vrho(\ve_3;3\mid12)}
 &=\frac{P_{003}P_{201}P_{021}}
         {P_{102}P_{012}P_{111}}.
\end{aligned}
\]
Their product is therefore
\[
 \frac{P_{030}P_{003}P_{210}P_{201}P_{021}}
      {P_{120}^{2}P_{102}P_{012}P_{111}}.
\]
The rows of $\lambda$ are $(0,0,0)$, $(1,1,0)$, and $(1,1,1)$, and
\[
 \Delta(0,0,0)=0,
 \qquad
 \Delta(1,1,0)=0,
 \qquad
 \Delta(1,1,1)=-3.
\]
Hence all three $\mathfrak m$-factors equal $1$. By
Lemma~\ref{lem:m-at-least-one}, this is the minimum in
Theorem~\ref{thm:constants33}. There are two diagonal factors and three
off-diagonal factors. Thus the factorial correction is
\[
 \left(\frac32\right)^{-2}2^{-3}=\frac1{18}.
\]
Theorem~\ref{thm:constants33} proves the ninth inequality and the optimality
of its constant.

All nine exponent vectors considered above are nonzero. The nonattainment
assertion in Theorem~\ref{thm:constants33} therefore shows that equality is
impossible in each of the nine inequalities when $F$ is strictly
Lorentzian.

\emph{(viii) Passing to all Lorentzian cubics.}
Now let $F$ be an arbitrary Lorentzian ternary cubic. By
Definition~\ref{def:strictly-lorentzian}, there is a sequence of strictly
Lorentzian cubics $F_r$ converging coefficientwise to $F$. Apply the nine
strict inequalities established above to $F_r$ and let $r\to\infty$.
Each side is a monomial in the coefficients, so it is continuous under
coefficientwise convergence. The resulting non-strict inequalities are
exactly the nine inequalities stated above. This argument remains valid
when some coefficients of $F$ vanish.
Since the constants are already optimal on the strictly Lorentzian subclass,
they remain optimal on the full Lorentzian class.
\end{proof}

\begin{remark}\label{rem:entropy}
On its nontrivial branch, $\mathfrak m$ is an exponentiated mutual information.
In the $2\times2$ contingency-table model of \cite{Har14}, the quantity
$2\log\mathfrak m(a,b,c)$ is the likelihood-ratio statistic $G^2$ for testing
independence.
\end{remark}

\section{Ternary Lorentzian forms of arbitrary degree}
\label{sec:further-optimal-constants}
\label{subsec:all-ternary-constants}

We extend Theorem~\ref{thm:constants33} to ternary Lorentzian forms of
arbitrary degree, using the same one-Hessian function $\mathfrak m$ as in
Definition~\ref{def:m-function}.

Fix $n\ge3$. For $\vb\in H(n-2,3)$ and $i\in\{1,2,3\}$, the indices
$j,k$ will always be the two elements of
$\{1,2,3\}\setminus\{i\}$. We retain the
triangular ratios of Definition~\ref{def:triangular}, namely,
\[
 R_{\vrho(\vb;i\mid jk)}(\vP)
 =\frac{P_{\vb+2\ve_i}P_{\vb+\ve_j+\ve_k}}
 {P_{\vb+\ve_i+\ve_j}P_{\vb+\ve_i+\ve_k}}.
\]
There are $|H(n-2,3)|=\binom n2$ choices of $\vb$ and three choices of
$i$ for each $\vb$. By Corollary~\ref{cor:threevariable}, these
$3\binom n2$ vectors generate $\BR_{\mathring L}(n,3)$.

For $\vg\in\BR_{\mathring L}(n,3)$, define
\begin{equation}\label{eq:all-ternary-representations}
 \Lambda_n(\vg)=
 \left\{\lambda=(\lambda_{\vb,i})\in\Rge^{\,3\binom n2}:\ 
 \vg=\sum_{\vb\in H(n-2,3)}\sum_{i=1}^3
 \lambda_{\vb,i}\cdot\vrho(\vb;i\mid jk)\right\}.
\end{equation}
Thus a point of $\Lambda_n(\vg)$ is an array with one row of three
nonnegative entries for every $\vb\in H(n-2,3)$. Different arrays may
represent the same exponent vector.

\begin{theoremB}[Optimal bounding constants for ternary Lorentzian forms]
\label{thm:all-ternary-constants}
For every $n\ge3$ and $\vg\in\BR_{\mathring L}(n,3)$,
\begin{equation}\label{eq:all-ternary-constant}
 f_{\mathring L}(\vg)
 =\left(\prod_{\va\in H(n,3)}(\va!)^{-\gamma_{\va}}\right)
 \min_{\lambda\in\Lambda_n(\vg)}
 \prod_{\vb\in H(n-2,3)}
 \mathfrak m(\lambda_{\vb,1},\lambda_{\vb,2},\lambda_{\vb,3}).
\end{equation}
If $\vg\ne0$, the supremum
defining $f_{\mathring L}(\vg)$ is not attained on
$\mathring L(n,3)$.
\end{theoremB}

The proof follows the argument for ternary cubics in
Theorem~\ref{thm:constants33}. We verify the compatibility statement and
construct a common interior point in arbitrary degree. The remaining
optimization steps then follow from the corresponding parts of the cubic
proof. We describe the compatibility space as the image of a linear map,
so that its equations need not be listed separately for each $n$.

For $z\in\mathbb R^{H(n,3)}$, define
\begin{align}
 L_n(z)_{\vb,i}
 &=\langle\vrho(\vb;i\mid jk),z\rangle\notag\\
 &=z_{\vb+2\ve_i}+z_{\vb+\ve_j+\ve_k}
 -z_{\vb+\ve_i+\ve_j}-z_{\vb+\ve_i+\ve_k}.
 \label{eq:all-ternary-linear-map}
\end{align}
Put
\[
 S_n=\operatorname{im}L_n,
 \qquad
 \mathcal Z_n=
 \left\{(\log(\va!P_{\va}))_{\va\in H(n,3)}:
 \vP\in\mathring L(n,3)\right\}.
\]
All pairings on arrays mean the sum of the products of corresponding
entries. The adjoint computation in part~(i) of the proof of
Lemma~\ref{lem:compatible-hessian-triples} gives
\begin{equation}\label{eq:all-ternary-adjoint}
 L_n^*(\lambda)=
 \sum_{\vb\in H(n-2,3)}\sum_{i=1}^3
 \lambda_{\vb,i}\cdot\vrho(\vb;i\mid jk).
\end{equation}
In particular,
\[
 S_n^\perp=\ker L_n^*.
\]
We use the one-Hessian region $\mathcal C$ from
Definition~\ref{def:one-hessian-data}. The notation
$\mathcal C^{H(n-2,3)}$ denotes the set of arrays whose $\vb$-th row
belongs to $\mathcal C$ for every $\vb$.

\begin{lemma}[Compatibility in arbitrary ternary degree]
\label{lem:all-ternary-image}
The exact image of the strictly Lorentzian locus is
\begin{equation}\label{eq:all-ternary-exact-image}
 L_n(\mathcal Z_n)=S_n\cap\mathcal C^{H(n-2,3)}.
\end{equation}
\end{lemma}

\begin{proof}
Let $F=\sum_{\va\in H(n,3)}P_{\va}\vx^{\va}$, with all
$P_{\va}>0$, and set $\widehat P_{\va}=\va!P_{\va}$.
For $\vb\in H(n-2,3)$, the polynomial $\partial^{\vb}F$ is quadratic.
The differentiation identity in Definition~\ref{def:triangular} gives
\begin{equation}\label{eq:all-ternary-hessian-entries}
 \mathcal H^{(\vb)}_{uv}
 :=\operatorname{Hess}(\partial^{\vb}F)_{uv}
 =\widehat P_{\vb+\ve_u+\ve_v}.
\end{equation}
If $z_{\va}=\log\widehat P_{\va}$, then
\[
 L_n(z)_{\vb,i}
 =\log\frac{\mathcal H^{(\vb)}_{ii}\mathcal H^{(\vb)}_{jk}}
 {\mathcal H^{(\vb)}_{ij}\mathcal H^{(\vb)}_{ik}}.
\]
If $F$ is strictly Lorentzian, every $\mathcal H^{(\vb)}$ has
Lorentzian signature. Hence every row of $L_n(z)$ belongs to
$\mathcal C$, proving the forward inclusion.

Conversely, take $x\in S_n\cap\mathcal C^{H(n-2,3)}$.
Since $S_n=\operatorname{im}L_n$, choose a single vector
$z\in\mathbb R^{H(n,3)}$ such that $L_n(z)=x$. Define
\[
 P_{\va}=\frac{e^{z_{\va}}}{\va!},
 \qquad
 F=\sum_{\va\in H(n,3)}P_{\va}\vx^{\va}.
\]
All coefficients are positive. Formula~\eqref{eq:all-ternary-hessian-entries}
constructs every derivative Hessian from this same coefficient vector.
For each $\vb$, its three triangular ratios are
$T_i=e^{x_{\vb,i}}$.

As shown in part~(ii) of the proof of
Lemma~\ref{lem:compatible-hessian-triples}, these three ratios determine
a positive symmetric $3\times3$ matrix up to positive diagonal congruence.
Since the $\vb$-th row of $x$ belongs to $\mathcal C$, it follows that
$\mathcal H^{(\vb)}$ has Lorentzian signature. This holds for every
$\vb\in H(n-2,3)$, so the single form $F$ constructed above is strictly
Lorentzian. Thus $z\in\mathcal Z_n$, proving the reverse inclusion.
\end{proof}

\begin{lemma}[A common interior point and closure]
\label{lem:all-ternary-interior}
Put
\[
 \mathcal D_n=(\overline{\mathcal C})^{H(n-2,3)}.
\]
Then $S_n\cap\operatorname{int}\mathcal D_n$ is nonempty, and
\begin{equation}\label{eq:all-ternary-section-closure}
 \overline{S_n\cap\mathcal C^{H(n-2,3)}}=S_n\cap\mathcal D_n.
\end{equation}
\end{lemma}

\begin{proof}
Define $Q\in\mathbb R^{H(n,3)}$ by
\[
 Q_{\va}=\alpha_1^2+\alpha_2^2+\alpha_3^2.
\]
For every $\vb,i$, direct expansion gives
\begin{align*}
 \langle\vrho(\vb;i\mid jk),Q\rangle
 &=(4\beta_i+4)+(2\beta_j+2\beta_k+2)\\
 &\quad -(2\beta_i+2\beta_j+2)
 -(2\beta_i+2\beta_k+2)=2.
\end{align*}
The four copies of $\sum_r\beta_r^2$ cancel in this expansion.
Choose $t>0$ and put $z^\circ=-tQ/2$. Then
$x^\circ=L_n(z^\circ)$ has every entry equal to $-t$.
The matrix with diagonal entries $1$ and off-diagonal entries $e^t$
has eigenvalues $1+2e^t,1-e^t,1-e^t$. It has positive entries and
Lorentzian signature, and its triangular ratios are all $e^{-t}$.
Consequently $(-t,-t,-t)\in\mathcal C$, so
\[
 x^\circ\in S_n\cap\mathcal C^{H(n-2,3)}.
\]
By Lemma~\ref{lem:one-hessian-region}, $\mathcal C$ is nonempty,
open, and convex. For a nonempty open convex set, the interior of its
closure is the original set. Taking finite products therefore gives
$\operatorname{int}\mathcal D_n=\mathcal C^{H(n-2,3)}$.
This proves the required interior-point assertion.

The closure assertion now follows from the interpolation argument in
part~(i) of the proof of Lemma~\ref{lem:compatible-row-optimization}.
Explicitly, for $x\in S_n\cap\mathcal D_n$,
\[
 x_\varepsilon=(1-\varepsilon)x+\varepsilon x^\circ,
 \qquad 0<\varepsilon<1,
\]
belongs to $S_n\cap\operatorname{int}\mathcal D_n$ and converges to $x$.
The opposite inclusion follows because $S_n\cap\mathcal D_n$ is closed.
\end{proof}

\begin{lemma}[The representation set]
\label{lem:all-ternary-compact}
For $\vg\in\BR_{\mathring L}(n,3)$, the set $\Lambda_n(\vg)$
is a nonempty compact polytope. The function
\[
 \Phi_n(\lambda)=
 \sum_{\vb\in H(n-2,3)}
 \log\mathfrak m(\lambda_{\vb,1},\lambda_{\vb,2},\lambda_{\vb,3})
\]
attains its minimum on $\Lambda_n(\vg)$.
\end{lemma}

\begin{proof}
Nonemptiness follows from Corollary~\ref{cor:threevariable}.
The set is defined by finitely many linear equalities and nonnegativity
conditions, so it is closed. Pair the defining representation with the
vector $Q$ from Lemma~\ref{lem:all-ternary-interior}. Since every
generator pairs with $Q$ to give $2$, we obtain
\begin{equation}\label{eq:all-ternary-weight-sum}
 \sum_{\vb,i}\lambda_{\vb,i}=\frac12\langle\vg,Q\rangle.
\end{equation}
Every coordinate is nonnegative and is at most the fixed sum on the
right. Thus $\Lambda_n(\vg)$ is bounded and hence compact.
The lower-semicontinuity argument in the final paragraph of the proof of
Lemma~\ref{lem:representation-polytope} applies to any finite number of
rows. Thus $\Phi_n$ attains its minimum on $\Lambda_n(\vg)$.
\end{proof}

\begin{proof}[Proof of Theorem~\ref{thm:all-ternary-constants}]
Fix $\lambda^0\in\Lambda_n(\vg)$. For a strictly Lorentzian form,
write $z=\log\widehat P$ and $x=L_n(z)$. As in the proof of
Theorem~\ref{thm:constants33}, the adjoint identity gives
\[
 \log\widehat P^{\vg}
 =\langle\vg,z\rangle=\langle\lambda^0,x\rangle.
\]
Lemmas~\ref{lem:all-ternary-image} and
\ref{lem:all-ternary-interior} therefore yield
\begin{equation}\label{eq:all-ternary-primal-support}
 \sup_{\vP\in\mathring L(n,3)}\log\widehat P^{\vg}
 =\sigma_{S_n\cap\mathcal D_n}(\lambda^0).
\end{equation}

Apply the convex-duality argument in parts~(ii)--(iii) of the proof of
Lemma~\ref{lem:compatible-row-optimization}, with $S_n$, $\mathcal D_n$,
and $L_n$ in place of $S$, $\mathcal D$, and $L$. Both sets are closed
and convex, and Lemma~\ref{lem:all-ternary-interior} supplies the common
interior point required there. Thus the same argument that proves
\eqref{eq:support-section-duality} gives
\[
 \sigma_{S_n\cap\mathcal D_n}(\lambda^0)
 =\inf_{\eta\in S_n^\perp}
 \sigma_{\mathcal D_n}(\lambda^0+\eta).
\]
Since $S_n^\perp=\ker L_n^*$, the arrays $\lambda^0+\eta$ in this
infimum are exactly the real arrays satisfying $L_n^*(\lambda)=\vg$.

Part~(iv) of the same proof evaluates the support function of a Cartesian
product by optimizing its rows independently. Applied to the rows indexed
by $H(n-2,3)$, it gives
\[
 \sigma_{\mathcal D_n}(\lambda)
 =\sum_{\vb\in H(n-2,3)}
 \sigma_{\mathcal C}(\lambda_{\vb,1},\lambda_{\vb,2},\lambda_{\vb,3}).
\]
By Lemma~\ref{lem:one-hessian-region}, this value is $\Phi_n(\lambda)$
when every entry is nonnegative, and is $+\infty$ otherwise. The finite
part of the dual infimum is consequently $\Lambda_n(\vg)$, so
\begin{equation}\label{eq:all-ternary-duality}
 \sup_{\vP\in\mathring L(n,3)}\log\widehat P^{\vg}
 =\inf_{\lambda\in\Lambda_n(\vg)}\Phi_n(\lambda)
 =\min_{\lambda\in\Lambda_n(\vg)}\Phi_n(\lambda).
\end{equation}
The last equality follows from Lemma~\ref{lem:all-ternary-compact}.
Because \eqref{eq:all-ternary-primal-support} uses the exact image of the
strictly Lorentzian locus, this equality also proves sharpness.

Finally, the coefficient conversion in the proof of
Theorem~\ref{thm:constants33} applies unchanged:
\[
 P^{\vg}=\widehat P^{\vg}
 \prod_{\va\in H(n,3)}(\va!)^{-\gamma_{\va}}.
\]
Exponentiating \eqref{eq:all-ternary-duality} proves
\eqref{eq:all-ternary-constant}. The nonattainment argument in the final
paragraphs of that proof also applies. Indeed, $\mathcal Z_n$ is open
because positivity of the coefficients and the signatures of the finitely
many derivative Hessians persist under sufficiently small perturbations.
For $\vg\ne0$, replacing $z$ by $z+\varepsilon\vg$ for sufficiently
small $\varepsilon>0$ increases $\langle\vg,z\rangle$. Hence the
supremum is not attained. For $\vg=0$, the ratio is identically $1$.
\end{proof}

\begin{remark}\label{rem:all-ternary-size}
For $n=3$, the rows are indexed by $\ve_1,\ve_2,\ve_3$, and
Theorem~\ref{thm:all-ternary-constants} recovers
Theorem~\ref{thm:constants33}. For $n=4$, there are six rows, indexed by
$2\ve_1,2\ve_2,2\ve_3,\ve_1+\ve_2,\ve_1+\ve_3,\ve_2+\ve_3$.
In general the optimization has $3\binom n2$ nonnegative weights.
The logarithm of the objective is convex, and the constraints in
\eqref{eq:all-ternary-representations} are linear. The one-Hessian
function does not change with the degree.
\end{remark}

\section{Parametrized inequality for quaternary cubics}
\label{weighted-star-subsection}

We compute the optimal constant for a weighted product of three triangular
ratios of a quaternary Lorentzian cubic.

Fix pairwise distinct indices $i,j,k,\ell\in\{1,2,3,4\}$. For
$\vP\in\mathring L(3,4)$, consider the three ratios
\begin{align}
 X(\vP)
 &:=R_{\vrho(\ve_i;i\mid k\ell)}(\vP)
 =\frac{P_{\ve_i+\ve_k+\ve_\ell}P_{3\ve_i}}
 {P_{2\ve_i+\ve_k}P_{2\ve_i+\ve_\ell}},\notag\\
 Y(\vP)
 &:=R_{\vrho(\ve_i;i\mid j\ell)}(\vP)
 =\frac{P_{\ve_i+\ve_j+\ve_\ell}P_{3\ve_i}}
 {P_{2\ve_i+\ve_j}P_{2\ve_i+\ve_\ell}},\label{weighted-star-ratios}\\
 Z(\vP)
 &:=R_{\vrho(\ve_i;i\mid jk)}(\vP)
 =\frac{P_{\ve_i+\ve_j+\ve_k}P_{3\ve_i}}
 {P_{2\ve_i+\ve_j}P_{2\ve_i+\ve_k}}.\notag
\end{align}
All three ratios have the same derivative index $\ve_i$ and the same
central index $i$. Their joint optimal bound is smaller than the product
of their individual optimal bounds in general. We first solve the
three-dimensional matrix optimization that gives the joint bound. We
then construct strictly Lorentzian cubics that approach it while
satisfying the condition that all four Hessians are strictly Lorentzian matrices.

\begin{lemma}\label{weighted-star-correlation}
Let $a,b,c>0$, and put $m=a+b+c$. Among real symmetric positive
semidefinite matrices $G=(g_{rs})_{r,s=1}^3$ with $g_{11}=g_{22}=g_{33}=1$,
the maximum of
\[
 (1-g_{23})^a(1-g_{13})^b(1-g_{12})^c
\]
is
\begin{equation}\label{weighted-star-correlation-value}
 2^m\frac{m^m a^a b^b c^c}
 {(a+b)^{a+b}(a+c)^{a+c}(b+c)^{b+c}}.
\end{equation}
It is attained at the matrix $G^*$ with diagonal entries $1$ and
\begin{equation}\label{weighted-star-candidate}
 1-g^*_{23}=\frac{2am}{(a+b)(a+c)},\qquad
 1-g^*_{13}=\frac{2bm}{(a+b)(b+c)},\qquad
 1-g^*_{12}=\frac{2cm}{(a+c)(b+c)}.
\end{equation}
The matrix $G^*$ has rank two.
\end{lemma}

\begin{proof}
Every principal $2\times2$ submatrix of a feasible $G$ is positive
semidefinite, so $-1\le g_{rs}\le1$. In particular all three factors in
the objective are nonnegative.

Write $x^*_{rs}=1-g^*_{rs}$ for the quantities in
\eqref{weighted-star-candidate}. For example,
\[
 (a+b)(a+c)-am=bc>0,
\]
which shows that $0<x^*_{23}<2$. The same conclusion holds for the other
two quantities. Put
\[
 v=\begin{pmatrix}b+c\\a+c\\a+b\end{pmatrix}.
\]
We claim that $G^*v=0$. Its first coordinate is
\begin{align*}
 (b+c)+g^*_{12}(a+c)+g^*_{13}(a+b)
 &=2m-x^*_{12}(a+c)-x^*_{13}(a+b)\\
 &=2m-\frac{2cm}{b+c}-\frac{2bm}{b+c}=0.
\end{align*}
The other two coordinates vanish by the same calculation with the
indices permuted. Thus $\det G^*=0$. Its principal $2\times2$ minors are
strictly positive. For instance,
\[
 1-(g^*_{23})^2
 =x^*_{23}(2-x^*_{23})
 =\frac{4abcm}{(a+b)^2(a+c)^2}>0.
\]
The other two formulas are obtained by permuting $a,b,c$. Every
principal minor of $G^*$ is therefore nonnegative, so $G^*$ is positive
semidefinite. The positive $2\times2$ minors and the zero determinant
show that its rank is exactly two.

We prove that this feasible matrix is a global maximizer. If one of the
three factors for another feasible matrix $G$ is zero, its objective is
zero and there is nothing to prove. Otherwise put $x_{rs}=1-g_{rs}>0$.
The inequality $\log u\le\log u_0+(u-u_0)/u_0$ for $u,u_0>0$ gives
\begin{align*}
 &a\log\frac{x_{23}}{x^*_{23}}
 +b\log\frac{x_{13}}{x^*_{13}}
 +c\log\frac{x_{12}}{x^*_{12}}\\
 &\qquad\le
 \frac{a}{x^*_{23}}(x_{23}-x^*_{23})
 +\frac{b}{x^*_{13}}(x_{13}-x^*_{13})
 +\frac{c}{x^*_{12}}(x_{12}-x^*_{12}).
\end{align*}
The three coefficients on the right are
\[
 \frac{a}{x^*_{23}}=\frac{v_2v_3}{2m},\qquad
 \frac{b}{x^*_{13}}=\frac{v_1v_3}{2m},\qquad
 \frac{c}{x^*_{12}}=\frac{v_1v_2}{2m}.
\]
Since $G$ and $G^*$ have the same diagonal, the right-hand side is
\[
 -\frac1{2m}\sum_{1\le r<s\le3}v_rv_s(g_{rs}-g^*_{rs})
 =-\frac{v^T(G-G^*)v}{4m}
 =-\frac{v^TGv}{4m}\le0.
\]
The second equality uses the fact that each off-diagonal entry occurs
twice in the quadratic form. The third uses $G^*v=0$. The inequality
follows because $G$ is positive semidefinite. Exponentiation proves global
optimality. Finally,
substituting \eqref{weighted-star-candidate} into the objective gives
\eqref{weighted-star-correlation-value}, because the factor $a+b$ occurs
with total exponent $a+b$, and likewise for the other two pair sums.
\end{proof}

\begin{lemma}\label{weighted-star-compatible-cubics}
Let $a,b,c>0$ and $m=a+b+c$. There is a family of strictly Lorentzian
cubics $F_t$ in four variables, defined for all sufficiently small $t>0$,
whose coefficient vectors $\vP(t)$ satisfy
\begin{align}
 \lim_{t\to0^{+}}X(\vP(t))
 &=\frac{4am}{3(a+b)(a+c)},\notag\\
 \lim_{t\to0^{+}}Y(\vP(t))
 &=\frac{4bm}{3(a+b)(b+c)},\label{weighted-star-limits}\\
 \lim_{t\to0^{+}}Z(\vP(t))
 &=\frac{4cm}{3(a+c)(b+c)}.\notag
\end{align}
\end{lemma}

\begin{proof}
The order of the three indices different from $i$ will be $(j,k,\ell)$.
Set
\[
 B_j=a(b+c),\qquad B_k=b(a+c),\qquad B_\ell=c(a+b),
 \qquad B=\begin{pmatrix}B_j\\B_k\\B_\ell\end{pmatrix}.
\]
For distinct $r,s\in\{1,2,3,4\}$, define positive numbers
\begin{equation}\label{weighted-star-mixed-coefficients}
 a_{rs}(t)=
 \begin{cases}
  tB_s,&r=i,\\
  t,&r\ne i.
 \end{cases}
\end{equation}
The subscripts distinguish this array from the scalar weight $a$.
For each $s\in\{1,2,3,4\}$, order the other three indices consistently
and let
\[
 b_s(t)=(a_{sr}(t))_{r\ne s},\qquad
 C_s(t)_{rr}=a_{rs}(t)\quad(r\ne s),\qquad
 C_s(t)_{ru}=1\quad(r,u\ne s,\ r\ne u).
\]
Thus the diagonal of $C_s(t)$ uses the entries with second index $s$,
whereas $b_s(t)$ uses the entries with first index $s$. Let $J$ denote
the $3\times3$ matrix whose entries are all $1$. For every $s$,
\[
 C_s(t)\longrightarrow C_0:=J-I
 \qquad(t\to0^{+}).
\]
The eigenvalues of $C_0$ are $2,-1,-1$. Continuity of eigenvalues shows
that all four matrices $C_s(t)$ are invertible and have signature
$(1,2)$ for sufficiently small $t>0$. Their inverses converge to
\[
 C_0^{-1}=-I+\frac12J.
\]
Put
\[
 q_s(t)=b_s(t)^TC_s(t)^{-1}b_s(t).
\]
For $s\ne i$ the vector $b_s(t)$ is $t\mathbf1$, where
$\mathbf1=(1,1,1)^T$. Consequently,
\[
 \lim_{t\to0^{+}}\frac{q_s(t)}{t^2}
 =\mathbf1^T\left(-I+\frac12J\right)\mathbf1
 =\frac32>0.
\]
For $s=i$ we have $b_i(t)=tB$, and hence
\begin{equation}\label{weighted-star-q-limit}
 \lim_{t\to0^{+}}\frac{q_i(t)}{t^2}
 =B^T\left(-I+\frac12J\right)B
 =2abcm>0.
\end{equation}
For completeness, let $A=ab$, $D=ac$, and $E=bc$. The entries of $B$
are then $A+D,A+E,D+E$. It follows that
\begin{align*}
 \frac12(B_j+B_k+B_\ell)^2-(B_j^2+B_k^2+B_\ell^2)
 &=2(AD+AE+DE)\\
 &=2abc(a+b+c),
\end{align*}
which proves the last equality in \eqref{weighted-star-q-limit}.
In particular every $q_s(t)$ is positive when $t>0$ is sufficiently
small.

Restrict to such $t$ with $t<1$, set $u_s(t)=(1-t)q_s(t)$, and define
\begin{equation}\label{weighted-star-cubic-construction}
 F_t(\vx)=\frac16\sum_{s=1}^4u_s(t)x_s^3
 +\frac12\sum_{r\ne s}a_{rs}(t)x_r^2x_s
 +\sum_{r<s<u}x_rx_sx_u.
\end{equation}
Every coefficient is positive. We now check the four derivative
Hessians, which is the necessary compatibility check for this
construction. With index $s$ first, direct differentiation gives
\[
 \operatorname{Hess}(\partial_sF_t)
 =\begin{pmatrix}
  u_s(t)&b_s(t)^T\\
  b_s(t)&C_s(t)
 \end{pmatrix}.
\]
Since $C_s(t)$ is invertible, block elimination by an invertible
congruence changes this matrix into
\[
 \begin{pmatrix}
  u_s(t)-b_s(t)^TC_s(t)^{-1}b_s(t)&0\\
  0&C_s(t)
 \end{pmatrix}
 =\begin{pmatrix}-tq_s(t)&0\\0&C_s(t)\end{pmatrix}.
\]
Its first entry is negative and $C_s(t)$ has one positive and two
negative eigenvalues. Sylvester's law of inertia therefore shows that
the derivative Hessian has one positive and three negative eigenvalues.
This holds for every $s$, so $F_t$ is strictly Lorentzian.

The coefficients relevant to $X$ are
\[
 P_{3\ve_i}=\frac{u_i(t)}6,\qquad
 P_{\ve_i+\ve_k+\ve_\ell}=1,\qquad
 P_{2\ve_i+\ve_k}=\frac{tB_k}2,\qquad
 P_{2\ve_i+\ve_\ell}=\frac{tB_\ell}2.
\]
It follows from \eqref{weighted-star-q-limit} that
\[
 X(\vP(t))=\frac{2(1-t)q_i(t)}{3t^2B_kB_\ell}
 \longrightarrow
 \frac{4abcm}{3bc(a+c)(a+b)}
 =\frac{4am}{3(a+b)(a+c)}.
\]
The two other limits follow by using $B_jB_\ell$ and $B_jB_k$ in the
denominator. This proves \eqref{weighted-star-limits}.
\end{proof}

\begin{theorem}\label{weighted-star-optimal-constant}
For every $a,b,c\ge0$, put $m=a+b+c$ and use the convention $0^0=1$.
Then the ratios in \eqref{weighted-star-ratios} satisfy
\begin{equation}\label{weighted-star-main-formula}
 \sup_{\vP\in\mathring L(3,4)}X(\vP)^aY(\vP)^bZ(\vP)^c
 =\left(\frac43\right)^m
 \frac{m^m a^a b^b c^c}
 {(a+b)^{a+b}(a+c)^{a+c}(b+c)^{b+c}}.
\end{equation}
If $m>0$, the supremum is not attained in $\mathring L(3,4)$.
\end{theorem}

\begin{proof}
First suppose $a,b,c>0$. Let $F$ be a strictly Lorentzian cubic with
coefficient vector $\vP$, and put
$\mathcal H=\operatorname{Hess}(\partial_iF)$. This matrix has positive
entries and signature $(1,3)$. As in Definition~\ref{def:triangular},
write $\widehat P_{\va}=\va!P_{\va}$. Then
\[
 \mathcal H_{rs}=\widehat P_{\ve_i+\ve_r+\ve_s},
\]
and the factorials in the three ratios give
\begin{equation}\label{weighted-star-factorials}
 \frac{\mathcal H_{ii}\mathcal H_{k\ell}}
 {\mathcal H_{ik}\mathcal H_{i\ell}}=\frac32X,\qquad
 \frac{\mathcal H_{ii}\mathcal H_{j\ell}}
 {\mathcal H_{ij}\mathcal H_{i\ell}}=\frac32Y,\qquad
 \frac{\mathcal H_{ii}\mathcal H_{jk}}
 {\mathcal H_{ij}\mathcal H_{ik}}=\frac32Z.
\end{equation}
For example, the pure-cube entry has factor $6$, each repeated-index
denominator entry has factor $2$, and the squarefree entry has factor
$1$, giving $6/(2\cdot2)=3/2$.

Define a positive diagonal matrix $D$ by
\[
 D_{ii}=\frac1{\sqrt{\mathcal H_{ii}}},\qquad
 D_{rr}=\frac{\sqrt{\mathcal H_{ii}}}{\mathcal H_{ir}}\quad(r\ne i).
\]
The matrix $D\mathcal HD$ has every entry in row and column $i$ equal to
$1$. With the remaining indices ordered as $(j,k,\ell)$, write it as
\[
 D\mathcal HD=\begin{pmatrix}1&\mathbf1^T\\\mathbf1&C\end{pmatrix}.
\]
Diagonal congruence preserves both signature and triangular ratios.
Taking the Schur complement of the upper-left entry $1$ shows that
$C-J$ is negative definite. Thus $G_0=J-C$ is positive definite.
Each diagonal entry $C_{rr}$ is positive. The matrix
\[
 G=G_0+\operatorname{diag}(C_{11},C_{22},C_{33})
\]
is therefore positive definite and has diagonal entries $1$. Its
off-diagonal entries satisfy $g_{rs}=1-C_{rs}$. By
\eqref{weighted-star-factorials},
\begin{equation}\label{weighted-star-correlation-ratios}
 1-g_{23}=C_{23}=\frac32X,\qquad
 1-g_{13}=C_{13}=\frac32Y,\qquad
 1-g_{12}=C_{12}=\frac32Z.
\end{equation}
Applying Lemma~\ref{weighted-star-correlation} and multiplying by
$(2/3)^m$ proves the upper bound in
\eqref{weighted-star-main-formula}. On the other hand, the simultaneous
limits in Lemma~\ref{weighted-star-compatible-cubics} have weighted
product equal to the right-hand side of
\eqref{weighted-star-main-formula}. They prove that the upper bound is
sharp for positive weights.

We now treat weights that may be zero. Denote the explicit right-hand
side of \eqref{weighted-star-main-formula} by $K(a,b,c)$ and the
supremum on its left by $S(a,b,c)$. The convention $0^0=1$ makes $K$
continuous on $\Rge^3$. Indeed, its logarithm is
\begin{align*}
 \log K(a,b,c)
 ={}&m\log(4/3)+m\log m+a\log a+b\log b+c\log c\\
 &-(a+b)\log(a+b)-(a+c)\log(a+c)-(b+c)\log(b+c),
\end{align*}
where $u\log u$ extends continuously to $0$ with value $0$.

Fix $\vP\in\mathring L(3,4)$. Its three ratios are positive, so the
already proved upper bound with weights
$(a+\varepsilon,b+\varepsilon,c+\varepsilon)$ tends, as
$\varepsilon\to0^{+}$, to
\[
 X(\vP)^aY(\vP)^bZ(\vP)^c\le K(a,b,c).
\]
Taking the supremum proves $S(a,b,c)\le K(a,b,c)$. For the reverse
inequality, the positive definite matrix $G$ constructed above has
$|g_{rs}|<1$. Equation~\eqref{weighted-star-correlation-ratios} therefore
gives
\[
 0<X(\vP),Y(\vP),Z(\vP)<\frac43.
\]
Consequently, for every $\varepsilon>0$,
\[
 X^{a+\varepsilon}Y^{b+\varepsilon}Z^{c+\varepsilon}
 \le\left(\frac43\right)^{3\varepsilon}X^aY^bZ^c.
\]
Taking suprema and using the formula for positive weights yields
\[
 K(a+\varepsilon,b+\varepsilon,c+\varepsilon)
 \le\left(\frac43\right)^{3\varepsilon}S(a,b,c).
\]
Letting $\varepsilon\to0^{+}$ proves $K(a,b,c)\le S(a,b,c)$.
This argument does not require continuity of the supremum as a function
of the weights. In particular, if $a=b=c=0$, the objective and the
formula both equal $1$.

Finally assume $m>0$. The locus of strictly Lorentzian cubics is open
in the space of coefficient vectors: positivity of the coefficients
and the strict signs of the finitely many Hessian eigenvalues persist
under sufficiently small perturbations. Given any
$\vP\in\mathring L(3,4)$, we may therefore increase $P_{3\ve_i}$ by a
small positive factor $1+\delta$ while keeping every other coefficient
fixed and remaining in $\mathring L(3,4)$. Each of $X,Y,Z$ is then
multiplied by $1+\delta$, so their weighted product increases by
$(1+\delta)^m>1$. Hence no coefficient vector in the strictly
Lorentzian locus attains the supremum.
\end{proof}

\begin{corollary}\label{weighted-star-product}
For every $a,b,c\ge0$ and $\vP\in\mathring L(3,4)$,
\begin{equation}\label{weighted-star-product-ratio}
 \frac{P_{\ve_i+\ve_k+\ve_\ell}^{\,a}
 P_{\ve_i+\ve_j+\ve_\ell}^{\,b}
 P_{\ve_i+\ve_j+\ve_k}^{\,c}
 P_{3\ve_i}^{\,a+b+c}}
 {P_{2\ve_i+\ve_j}^{\,b+c}
 P_{2\ve_i+\ve_k}^{\,a+c}
 P_{2\ve_i+\ve_\ell}^{\,a+b}}
 \le
 \left(\frac43\right)^{a+b+c}\cdot
 \frac{(a+b+c)^{a+b+c}a^ab^bc^c}
 {(a+b)^{a+b}(a+c)^{a+c}(b+c)^{b+c}}.
\end{equation}
Here $0^0=1$. For each fixed $a,b,c$, the constant on the right is
optimal. The inequality is strict on $\mathring L(3,4)$ whenever
$a+b+c>0$.
\end{corollary}

\begin{proof}
Expanding $X^aY^bZ^c$ using \eqref{weighted-star-ratios} gives the ratio
on the left. Apply Theorem~\ref{weighted-star-optimal-constant} with
$m=a+b+c$.
\end{proof}

\begin{corollary}\label{weighted-star-product-lorentzian}
Let $k\ge4$, and let
$F(\vx)=\sum_{\va\in H(3,k)}P_{\va}\vx^{\va}$ be a Lorentzian cubic.
For every $a,b,c\ge0$ and every four pairwise
distinct indices $i,j,\ell,r\in\{1,\ldots,k\}$,
\[
\begin{gathered}
 P_{\ve_i+\ve_\ell+\ve_r}^{\,a}
 P_{\ve_i+\ve_j+\ve_r}^{\,b}
 P_{\ve_i+\ve_j+\ve_\ell}^{\,c}
 P_{3\ve_i}^{\,a+b+c}
 \le
 C\,P_{2\ve_i+\ve_j}^{\,b+c}
 P_{2\ve_i+\ve_\ell}^{\,a+c}
 P_{2\ve_i+\ve_r}^{\,a+b},\\[4pt]
 \text{where}\quad C=\left(\frac43\right)^{a+b+c}\cdot
 \frac{(a+b+c)^{a+b+c}a^ab^bc^c}
 {(a+b)^{a+b}(a+c)^{a+c}(b+c)^{b+c}}.
\end{gathered}
\]
Here $0^0=1$. For each fixed $a,b,c$, the constant $C$ is optimal
for every $k\ge4$.
\end{corollary}

\begin{proof}
Fix $a,b,c\ge0$. For $k=4$, approximate $F$ coefficientwise by strictly
Lorentzian cubics and apply \eqref{weighted-star-product-ratio} after
clearing denominators. For every fixed exponent $d\ge0$, the function
$u\mapsto u^d$ is continuous on $\Rge$, with $u^0=1$. Both sides are
therefore continuous functions of the nonnegative coefficients, so the
inequality passes to the limit.

For $k>4$, set all variables except the four selected ones equal to zero.
The resulting polynomial is Lorentzian by \cite[Theorem~2.10]{BH20} and
has the same coefficients appearing in the inequality. Relabeling these
four variables and applying the case $k=4$ proves the result.

Optimality follows from Corollary~\ref{weighted-star-product}, since a
four-variable Lorentzian cubic can be viewed as a Lorentzian cubic in
$k$ variables independent of the remaining variables, again by
\cite[Theorem~2.10]{BH20}.
\end{proof}

\section{\texorpdfstring{Products of linear forms and the cone
$\BR_{\mathring{Q}}(n,k)$}{Products of linear forms and the cone BRQ(n,k)}}

\subsection{\texorpdfstring{The forms $Q(n,k)$}{The forms Q(n,k)}}

\begin{definition}[{\rm \defn{Product of linear forms}}]\label{def:Q}
Let $Q(n,k)$ denote a product of $n$ linear forms in $k$ variables with
nonnegative coefficients:
\[
  Q(n,k) \;=\; \prod_{r=1}^{n} \ell_r,
  \qquad
  \ell_r \;=\; \sum_{i=1}^{k} a_{ri}\,x_i,
  \qquad
  a_{ri} \ge 0 \quad (1 \le r \le n,\ 1 \le i \le k).
\]
We collect the coefficients in the matrix
$A = (a_{ri}) \in \Rge^{\,n \times k}$.  The $r$th row records the coefficients
of the linear form $\ell_r$, and the $i$th column records the coefficients of
$x_i$ in $\ell_1,\dots,\ell_n$.  Thus $Q(n,k)$ is a homogeneous form of degree
$n$ in $x_1, \dots, x_k$, depending on the $nk$ nonnegative parameters
$a_{ri}$.
\end{definition}

\begin{example}[$n = 2$, $k = 3$]\label{ex:Q23}
Writing $a_i = a_{1i}$ and $b_i = a_{2i}$,
\[
  Q(2,3) \;=\; (a_1x_1 + a_2x_2 + a_3x_3)(b_1x_1 + b_2x_2 + b_3x_3),
  \qquad \text{all $a$'s and $b$'s} \ \ge 0 .
\]
\end{example}

\begin{theorem}\label{thm:Q-lorentzian}
Every product of linear forms with nonnegative coefficients is a Lorentzian
polynomial.  In particular, $Q(n,k)$ is Lorentzian for every
$A \in \Rge^{\,n \times k}$.
\end{theorem}

\begin{proof}
Suppose first that no factor is zero. Each factor is stable: every nonzero
linear form with nonnegative coefficients has positive imaginary part whenever
all the variables have positive imaginary parts. Stability is preserved under
products, and every
homogeneous stable polynomial with nonnegative coefficients is Lorentzian
\cite[Proposition~2.2]{BH20}.  If a factor is zero, the product is the zero
polynomial, which is Lorentzian by coefficientwise closure.
\end{proof}

\begin{definition}[{\rm \defn{Coefficients of $Q(n,k)$}}]\label{def:P}
An \defn{assignment of type $\va$} is a map
$\phi \colon \{1,\dots,n\} \to \{1,\dots,k\}$ with $|\phi^{-1}(i)| = \alpha_i$
for every $i$.  There are $\binom{n}{\va} = n!/\va!$ such assignments.  For
$\va \in H(n,k)$, let $P_{\va} = P_{\va}(A)$ denote the coefficient of
$\vx^{\va}$ in $Q(n,k)$, so that
\[
  Q(n,k) \;=\; \sum_{\va \in H(n,k)} P_{\va}\, \vx^{\va}.
\]
Thus each $P_{\va}$ is a homogeneous polynomial of degree $n$ in the
nonnegative parameters $a_{ri}$, with one monomial for each assignment of
type $\va$:
\[
  P_{\va}
  \;=\;
  \sum_{\phi \ \text{of type } \va}
  \ \prod_{r=1}^{n} a_{r,\,\phi(r)} .
\]
\end{definition}

\begin{definition}[{\rm \defn{The cone $\BR_{\mathring{Q}}(n,k)$}}]\label{def:BR}
For $A \in \Rpos^{\,n \times k}$, write
$\vP(A) = \bigl(P_{\va}(A)\bigr)_{\va \in H(n,k)}$, a vector with strictly
positive entries, and let
\[
  \mathcal{P}(n,k)
  \;=\;
  \bigl\{\, \vP(A) \;:\; A \in \Rpos^{\,n \times k} \,\bigr\}
  \;\subseteq\;
  \Rpos^{H(n,k)} .
\]
In the notation of Definition~\ref{def:BR-general}, we set
$\BR_{\mathring{Q}}(n,k) \coloneqq \BR\bigl(\mathcal{P}(n,k)\bigr)$.
\end{definition}

\subsection{Assignment valuations}

Fix $D = (d_{ri}) \in \mathbb{R}^{\,n\times k}$ and $0 < t < 1$. Under the
substitution $a_{ri} = t^{d_{ri}}$, the monomial of $P_{\va}$ indexed by an
assignment $\phi$ of type $\va$ becomes
$t^{\,\sum_r d_{r,\phi(r)}}$. Consequently, the lowest-order term of
$P_{\va}(t^D)$ as $t \to 0^{+}$ is determined by the minimum of these
exponents. Define the \defn{assignment valuation}
$\mu_D \colon H(n,k) \to \mathbb{R}$ by
\begin{equation}\label{eq:assignment-valuation}
 \mu_D(\va)
 \;=\;
 \min_{\phi \ \text{of type } \va} \ \sum_{r=1}^{n} d_{r,\phi(r)} ,
\end{equation}
Thus $\mu_D(\va)$ is the least exponent of $t$ occurring in
$P_{\va}(t^{D})$. We also set
\begin{equation}\label{eq:Ank}
 \mathcal A_{n,k}
 \;=\;
 \operatorname{cone}\bigl\{[\mu_D] : D \in \mathbb R^{\,n\times k}\bigr\}
 \;\subseteq\; \mathbb R^{H(n,k)}/\operatorname{Aff}(H(n,k)).
\end{equation}

A \defn{choice of minimizing assignments} is a tuple
$\bm\phi = (\phi_{\va})_{\va \in H(n,k)}$ in which $\phi_{\va}$ has type
$\va$. The matrices realizing a given choice form the polyhedral cone
\[
  C(\bm\phi)
  \;=\;
  \Bigl\{\, D \in \mathbb{R}^{\,n\times k} \;:\;
  \textstyle\sum_r d_{r,\phi_{\va}(r)} \le \sum_r d_{r,\psi(r)}
  \ \text{ for all } \va \in H(n,k) \text{ and all } \psi \text{ of type }
  \va \,\Bigr\}.
\]
The resulting finite collection of rational polyhedral cones covers
$\mathbb{R}^{\,n\times k}$. On $C(\bm\phi)$, each $\mu_D(\va)$ agrees with
the linear functional $D \mapsto \sum_r d_{r,\phi_{\va}(r)}$. We will use
two consequences of this description. First,
$D \mapsto \mu_D$ is rational linear on each $C(\bm\phi)$, so
$\mathcal{A}_{n,k}$ is a finite Minkowski sum of rational polyhedral cones and
is therefore closed and rational polyhedral.  Second, for a fixed exponent
vector $\vg$, the function $D \mapsto \langle \vg, \mu_D \rangle$ is linear on
each $C(\bm\phi)$, hence nonnegative there if and only if it is nonnegative at
every extreme ray of $C(\bm\phi)$.

\subsection{The polar description}

\begin{proposition}[Polar description of $\BR_{\mathring{Q}}$]
\label{prop:BRQ-polar}
The cone of bounded ratios for products of linear forms is the positive dual of
the assignment-valuation cone. More precisely,
\[
 \BR_{\mathring Q}(n,k)
 =\mathcal A_{n,k}^{\vee}
 =\Bigl\{\vg\in V_{n,k}:
   \langle\vg,\mu_D\rangle\ge0
   \text{ for every }D\in\mathbb R^{\,n\times k}\Bigr\}.
\]
\end{proposition}

\begin{proof}
We prove the two inclusions separately.

\smallskip
\noindent
\emph{Necessity.}
Suppose that $R_{\vg}$ is bounded on products of positive linear forms.  We
first show that $\vg$ is balanced.  Fix $i\in\{1,\dots,k\}$ and multiply the
$i$th column of a positive parameter matrix $A$ by a scalar $s>0$.  Every
assignment of type $\va$ uses that column exactly $\alpha_i$ times, so this
rescaling sends
\[
 P_{\va}(A)\longmapsto s^{\alpha_i}P_{\va}(A).
\]
Consequently, it multiplies the ratio by
\[
 s^{\sum_{\va}\gamma_{\va}\alpha_i}.
\]
The rescaled matrix is positive for every $s>0$.  If the exponent in this
display were positive, the ratio would become unbounded as $s\to\infty$.  If
it were negative, the ratio would become unbounded as $s\to0^+$.  Hence
\[
 \sum_{\va\in H(n,k)}\gamma_{\va}\alpha_i=0
 \qquad(i=1,\dots,k),
\]
which is precisely the condition $\vg\in V_{n,k}$.

Now fix $D=(d_{ri})\in\mathbb R^{\,n\times k}$ and $0<t<1$.  Substitute
$a_{ri}=t^{d_{ri}}$.  For an assignment $\phi$ of type $\va$, the
corresponding exponent and monomial are
\[
 c_D(\phi)=\sum_{r=1}^n d_{r,\phi(r)},
 \qquad
 \prod_{r=1}^n t^{d_{r,\phi(r)}}=t^{c_D(\phi)}.
\]
By definition, the smallest exponent among these monomials is
$\mu_D(\va)=\min_\phi c_D(\phi)$.  Factoring out the contribution of this
smallest exponent gives
\[
 P_{\va}(t^D)
 =t^{\mu_D(\va)}q_{\va,D}(t),
 \qquad
 q_{\va,D}(t)
 =\sum_{\phi\text{ of type }\va}
   t^{c_D(\phi)-\mu_D(\va)}.
\]
Every exponent in the last sum is nonnegative and at least one is zero.
Since $0<t<1$ and there are $\binom n\va$ assignments of type $\va$, we have
\[
 1\le q_{\va,D}(t)\le\binom n\va.
\]
Equivalently,
\begin{equation}\label{eq:mu-sandwich}
 t^{\mu_D(\va)}
 \le P_{\va}(t^D)
 \le\binom n\va t^{\mu_D(\va)}.
\end{equation}
Then
\[
 R_{\vg}(t^D)
 =t^{\langle\vg,\mu_D\rangle}
  \prod_{\va\in H(n,k)}q_{\va,D}(t)^{\gamma_{\va}}.
\]
The second factor is bounded above and below by positive constants independent
of $t$:
\[
 \prod_{\va}\binom n\va^{\min\{\gamma_{\va},0\}}
 \le
 \prod_{\va}q_{\va,D}(t)^{\gamma_{\va}}
 \le
 \prod_{\va}\binom n\va^{\max\{\gamma_{\va},0\}}.
\]
If $\langle\vg,\mu_D\rangle<0$, the first factor tends to infinity as
$t\to0^+$, contradicting boundedness.  Therefore
$\langle\vg,\mu_D\rangle\ge0$ for every $D$.  Since $\vg$ is balanced and
$\mathcal A_{n,k}$ is generated by the classes $[\mu_D]$, this says exactly
that $\vg\in\mathcal A_{n,k}^{\vee}$.

\smallskip
\noindent
\emph{Sufficiency.}
Conversely, suppose that $\vg\in V_{n,k}$ and
$\langle\vg,\mu_D\rangle\ge0$ for every $D$.  Let
$A=(a_{ri})\in\mathbb R_{>0}^{\,n\times k}$ be arbitrary and set
\[
 d_{ri}=-\log a_{ri},
 \qquad\text{so that}\qquad
 a_{ri}=e^{-d_{ri}}.
\]
For an assignment $\phi$ of type $\va$, its monomial in $P_{\va}(A)$ is
\[
 \prod_{r=1}^n a_{r,\phi(r)}
 =e^{-c_D(\phi)}.
\]
The map $u\mapsto e^{-u}$ is decreasing.  Hence the largest assignment
monomial is obtained when the sum in the exponent is smallest, and its value
is
\[
 \exp\bigl(-\mu_D(\va)\bigr).
\]
Factoring this largest monomial out of the coefficient gives
\[
 P_{\va}(A)=e^{-\mu_D(\va)}q_{\va},
 \qquad
 q_{\va}
 =\sum_{\phi\text{ of type }\va}
   e^{-(c_D(\phi)-\mu_D(\va))}.
\]
Every difference $c_D(\phi)-\mu_D(\va)$ is nonnegative, with equality for at
least one minimizing assignment. Hence
\[
 1\le q_{\va}\le\binom n\va.
\]
Equivalently,
\[
 e^{-\mu_D(\va)}
 \le P_{\va}(A)
 \le\binom n\va e^{-\mu_D(\va)}.
\]
With $\gamma_{\va}^{+}=\max\{\gamma_{\va},0\}$, it follows that
\begin{align}
 R_{\vg}(A)
 &=\exp\bigl(-\langle\vg,\mu_D\rangle\bigr)
   \prod_{\va}q_{\va}^{\gamma_{\va}}\notag\\
 &\le
 \exp\bigl(-\langle\vg,\mu_D\rangle\bigr)
 \prod_{\va}\binom n\va^{\gamma_{\va}^{+}}
 \le\prod_{\va}\binom n\va^{\gamma_{\va}^{+}}.
 \label{eq:BRQ-bound}
\end{align}
The final bound depends only on $\vg$, not on $A$.  Hence $R_{\vg}$ is
bounded on all products of positive linear forms, so
$\vg\in\BR_{\mathring Q}(n,k)$.
\end{proof}

\subsection{The algorithm}

\begin{algorithm}[Computation of $\BR_{\mathring{Q}}(n,k)$]\label{alg:BR}
\leavevmode
\begin{enumerate}
  \item Substitute $a_{ri} = t^{d_{ri}}$, so that every monomial of every
        $P_{\va}$ becomes a power of $t$ with exponent linear in $D$.
  \item Enumerate the choices of minimizing assignments
        $\bm\phi = (\phi_{\va})$ and form the corresponding systems of linear
        inequalities.  Discard the infeasible ones.
  \item For each feasible $\bm\phi$, compute the extreme rays $V = (v_{ri})$ of
        $C(\bm\phi)$ and record the linear inequality
        $\sum_{\va} \gamma_{\va} \sum_r v_{r,\phi_{\va}(r)} \ge 0$ in the
        unknowns $\gamma_{\va}$.
  \item By Proposition~\ref{prop:BRQ-polar}, the collected equations and
        inequalities cut out $\BR_{\mathring{Q}}(n,k)$.
\end{enumerate}
\end{algorithm}

\section{Computational results}\label{sec:computational}

\subsection{\texorpdfstring{The case $n=3$, $k=3$}{The case n=3, k=3}}

\begin{theorem}\label{thm:BRQ33}
The cone $\BR_{\mathring Q}(3,3)=\BR_{\mathring L}(3,3)$ is generated by
the exponent vectors of the nine triangular ratios
\[
 R_{\vrho(\ve_s;i\mid jk)}
 =\frac{P_{\ve_s+2\ve_i}P_{\ve_s+\ve_j+\ve_k}}
 {P_{\ve_s+\ve_i+\ve_j}P_{\ve_s+\ve_i+\ve_k}},
 \qquad s,i\in\{1,2,3\},
\]
where $\{j,k\}=\{1,2,3\}\setminus\{i\}$. It has dimension seven, nine
extreme rays, and eight facets. Each ratio is subtraction-free with $C=1$
and has supremum $1$ on $\mathcal P(3,3)$, not attained there.
\end{theorem}

\begin{proof}
Algorithm~\ref{alg:BR} with $n=k=3$ gives the nine displayed generators,
dimension seven, and eight facets. Corollary~\ref{cor:threevariable}
identifies their conical hull with $\BR_{\mathring L}(3,3)$ and shows
that all nine rays are extreme.

Variable permutations give two orbits of sizes six and three, according
as $s\ne i$ or $s=i$, with representatives
\[
  R_6^*:=\frac{P_{021}P_{102}}{P_{012}P_{111}},
  \qquad
  R_3^*:=\frac{P_{003}P_{111}}{P_{012}P_{102}}.
\]
Write each representative as $N/D$. Exact expansion in the parameters
$a_{ri}$ of the three linear factors shows that $D-N$ is a nonzero
polynomial with nonnegative coefficients. Thus both representatives are
subtraction-free with $C=1$ and strictly less than $1$ on
$\mathcal P(3,3)$. By symmetry, the same holds for all nine ratios.

To prove sharpness, let $t>0$ and consider
\begin{align*}
  F_{6,t}
  &=(tx_1+x_2+x_3)(t^2x_1+x_2+x_3)(x_1+x_2+tx_3),\\
  F_{3,t}
  &=(x_1+tx_2+x_3)(tx_1+x_2+x_3)(tx_1+tx_2+x_3).
\end{align*}
By symmetry, sharpness follows from the direct coefficient computations
\begin{align*}
  R_6^*(F_{6,t})
  &=\frac{1+t^2+t^3}{1+2t+t^2+2t^3}\longrightarrow1,\\
  R_3^*(F_{3,t})
  &=\frac{1+2t+3t^2}{(1+2t)^2}\longrightarrow1
  \qquad (t\to0^+).\qedhere
\end{align*}
\end{proof}

\subsection{\texorpdfstring{The case $n = 3$, $k = 4$}{The case n=3, k=4}}

\begin{theorem}\label{thm:BRQ34}
$\BR_{\mathring{Q}}(3,4)$ is a pointed rational polyhedral cone of dimension sixteen with $80$
extreme rays and $38$ facets. The extreme rays form five $S_4$-orbits of sizes
$24$, $8$, $24$, $12$, $12$, with representatives
\[
  R_A = \frac{P_{0030} P_{0201} P_{1002} P_{1110} P_{2001}}
             {P_{0012} P_{0210} P_{1011} P_{1101} P_{2010}},
  \qquad
  R_B = \frac{P_{0030} P_{0201} P_{1002} P_{2100}}
             {P_{0012} P_{0210} P_{1101} P_{2010}},
\]
\[
  R_C = \frac{P_{0021} P_{1002}}{P_{0012} P_{1011}},
  \qquad
  R_D = \frac{P_{0030} P_{1011}}{P_{0021} P_{1020}},
  \qquad
  R_E = \frac{P_{0210} P_{1011}}{P_{0111} P_{1110}} .
\]
Each of the $80$ ratios satisfies $R_{\vg} < 1$ on $\mathcal{P}(3,4)$, and
$\sup_{\mathcal{P}(3,4)} R_{\vg} = 1$.
\end{theorem}

\begin{proof}
Applying Algorithm~\ref{alg:BR} with $(n,k)=(3,4)$ and then carrying out exact
ray and facet enumeration yields the stated cone data and the five
$S_4$-orbits.

For $X\in\{A,B,C,D,E\}$, write $R_X=N_X/D_X$. In each case, $D_X-N_X$ is a
nonzero polynomial in the parameters of the three linear factors, and all its
coefficients are nonnegative. Hence $R_X<1$ on $\mathcal P(3,4)$. By symmetry,
the same holds for all $80$ ratios.

It remains to prove sharpness. For $\vd=(d_1,\ldots,d_6)$ and $t>0$, set
\[
  F_{\vd,t}
  =(x_1+x_2+x_3+x_4)
   (t^{d_1}x_1+t^{d_2}x_2+t^{d_3}x_3+x_4)
   (t^{d_4}x_1+t^{d_5}x_2+t^{d_6}x_3+x_4).
\]
Take
\[
\begin{array}{lll}
 \vd_A=(-2,2,1,2,-1,1),&
 \vd_B=(3,-4,1,-3,-4,-2),&
 \vd_C=(-4,-3,-2,3,-4,-1),\\
\vd_D=(3,2,1,-3,2,-1),&
 \vd_E=(4,3,5,-5,-1,2).&
\end{array}
\]
For $t>0$, each $F_{\vd_X,t}$ belongs to $\mathcal P(3,4)$.
For each $X$, the lowest powers of $t$ appearing in the numerator and
denominator of $R_X(F_{\vd_X,t})$ have the same exponent and the same
coefficient. Therefore
\[
  R_X(F_{\vd_X,t})\longrightarrow1
  \qquad(t\to0^+).
\]
Variable permutations give the same limit for every ratio in the
corresponding orbit. Thus all $80$ ratios have supremum $1$.
\end{proof}

\begin{theorem}\label{thm:BRLor34}
The cone $\BR_{\mathring L}(3,4)$ has dimension $16$ and exactly $48$
extreme rays, namely all triangular-ratio rays
\[
 \vrho(\ve_s;i\mid jk),
 \qquad s,i,j,k\in\{1,2,3,4\},
\]
with $i,j,k$ pairwise distinct and $\{j,k\}$ unordered. They form the
$R_D$-, $R_C$-, and $R_E$-orbits of Theorem~\ref{thm:BRQ34}, of sizes
$12$, $24$, and $12$, with sharp constants $4/3$, $1$, and $1$,
respectively. None of these suprema is attained on $\mathring L(3,4)$.

The $32$ ratios in the $R_A$- and $R_B$-orbits are unbounded on
$\mathring L(3,4)$. Thus
$\BR_{\mathring L}(3,4)\subsetneq\BR_{\mathring Q}(3,4)$.
\end{theorem}

\begin{proof}
Definition~\ref{def:triangular} already shows that every displayed ratio
is bounded on $\mathring L(3,4)$. It remains to show that their exponent
vectors generate the entire cone. In the quotient by affine functions,
let $K$ be the cone defined by
\[
 \langle\vrho(\ve_s;i\mid jk),\nu\rangle\ge0
\]
for all indices as in the statement. Theorem~\ref{thm:BRL-duality}
therefore gives $\mathcal M_{3,4}\subseteq K$.

The reverse inclusion is the special computation needed here.
Choose the representative of each affine class satisfying
$\nu(3\ve_r)=0$ for $r=1,2,3,4$, leaving $20-4=16$ coordinates.
Exact ray enumeration for the $48$ inequalities defining $K$ gives a
pointed, full-dimensional cone with $42$ extreme rays in six
$S_4$-orbits. Each ray has an M-convex representative, as checked by
the local tree-metric criterion of Proposition~\ref{prop:treemetric}.
Consequently every extreme ray of $K$ belongs to $\mathcal M_{3,4}$,
so $K\subseteq\mathcal M_{3,4}$. Taking polars and using
Theorem~\ref{thm:BRL-duality} yields
\[
 \BR_{\mathring L}(3,4)
 =K^\vee
 =\operatorname{cone}\bigl\{\vrho(\ve_s;i\mid jk)\bigr\}.
\]
Since $K$ is pointed, its polar has dimension $16$.
The displayed generators are exactly the $R_D$-, $R_C$-, and $R_E$-orbits,
according as $s=i$, $s\in\{j,k\}$, or $s,i,j,k$ are pairwise distinct.
By Theorem~\ref{thm:BRQ34}, these are $48$ distinct extreme rays of
$\BR_{\mathring Q}(3,4)$, so they remain extreme in their conical hull.
The same theorem shows that none of the other $32$ extreme rays belongs
to this hull. This proves their unboundedness and the strict inclusion.

The constants follow from earlier results. For $s=i$,
Theorem~\ref{weighted-star-optimal-constant} with weights $(1,0,0)$
gives the sharp constant $4/3$ and nonattainment. For $s\ne i$ and
a strictly Lorentzian cubic $F$, the
principal submatrix of $\operatorname{Hess}(\partial_sF)$ on $i,j,k$
has positive entries and signature $(1,2)$ by eigenvalue interlacing.
Lemma~\ref{lem:one-hessian-region} gives $T_{(\ve_s;i\mid jk)}<2$:
the triple of logarithmic triangular ratios lies in the open region
$\mathcal C$, whose coordinate suprema equal
$\log\mathfrak m(1,0,0)=\log2$.
The conversion in Definition~\ref{def:triangular} is
$T_{(\ve_s;i\mid jk)}=2R_{\vrho(\ve_s;i\mid jk)}$, so the latter ratio
is strictly less than $1$. Sharpness follows from the product families
of Theorem~\ref{thm:BRQ34}. They are Lorentzian by
Theorem~\ref{thm:Q-lorentzian}, and coefficientwise approximation by
strictly Lorentzian forms preserves their limiting ratio, since their
denominator coefficients are positive. Thus the supremum is $1$ and
is not attained.
\end{proof}

\section{\texorpdfstring{When $\BR_{\mathring L}(n,k)\stackrel{?}{=}
\BR_{\mathring Q}(n,k)$}{BRL(n,k) =? BRQ(n,k)}}
\label{sec:comparison}

By Theorem~\ref{thm:Q-lorentzian}, every product of linear forms with
nonnegative coefficients is Lorentzian, and its positive coefficient vector
lies in the coefficientwise closure of the strictly Lorentzian locus.  Thus
one always has
\[
 \BR_{\mathring L}(n,k)\subseteq \BR_{\mathring Q}(n,k).
\]
We determine exactly when this inclusion is an equality. The classification
rests on the polar reduction below, a persistence lemma for the number of
variables, and two families of explicit witnesses.

The resulting classification is summarized in the following table.
\[
\begin{gathered}
 \BR_{\mathring L}(n,k)\stackrel{?}{=}\BR_{\mathring Q}(n,k)
 \\[4pt]
\begin{array}{c|ccccccc}
 k\backslash n & 2 & 3 & 4 & 5 & 6 & 7 & \cdots \\ \hline
 2 & = & = & = & = & = & = & \cdots \\
 3 & = & = & = & = & \ne & \ne & \cdots \\
 4 & = & \ne & \ne & \ne & \ne & \ne & \cdots \\
 5 & = & \ne & \ne & \ne & \ne & \ne & \cdots \\
 6 & = & \ne & \ne & \ne & \ne & \ne & \cdots \\
 7 & = & \ne & \ne & \ne & \ne & \ne & \cdots \\
 \vdots & \vdots & \vdots & \vdots & \vdots & \vdots & \vdots & \ddots
\end{array}
\end{gathered}
\]

\subsection{The comparison criterion}

\begin{proposition}[The two polar descriptions]
\label{prop:comparison-polars}
One has
\begin{equation}\label{eq:two-polars}
 \BR_{\mathring Q}(n,k)=\mathcal A_{n,k}^{\vee},
 \qquad
 \BR_{\mathring L}(n,k)=\mathcal M_{n,k}^{\vee},
 \qquad
 \mathcal A_{n,k}\subseteq\mathcal M_{n,k},
\end{equation}
and consequently
\begin{equation}\label{eq:equality-criterion}
 \BR_{\mathring L}(n,k)=\BR_{\mathring Q}(n,k)
 \quad\Longleftrightarrow\quad
 \mathcal A_{n,k}=\mathcal M_{n,k}.
\end{equation}
\end{proposition}

\begin{proof}
The first two identities in \eqref{eq:two-polars} follow from
Proposition~\ref{prop:BRQ-polar} and Theorem~\ref{thm:BRL-duality},
respectively. It remains
to prove that $\mathcal A_{n,k}\subseteq\mathcal M_{n,k}$.

Fix $D\in\mathbb R^{n\times k}$. We verify the M-convex exchange axiom for
$\mu_D$. Let $\va,\vb\in H(n,k)$, and suppose that $\alpha_i>\beta_i$.
Choose minimizing assignments $\phi$ and $\psi$ of types $\va$ and $\vb$.
Form a directed multigraph on $\{1,\ldots,k\}$ by drawing the edge
$\phi(r)\to\psi(r)$ for each row $r$. At a vertex $h$, the number of outgoing
edges is $\alpha_h$, and the number of incoming edges is $\beta_h$.

Let $U$ be the set of vertices reachable from $i$ by directed edges, including
$i$ itself. We claim that $U$ contains an index $j$ with
$\alpha_j<\beta_j$. Otherwise every
$\alpha_h-\beta_h$ with $h\in U$ would be nonnegative, and the term with
$h=i$ would be positive. Hence
\[
 \sum_{h\in U}(\alpha_h-\beta_h)>0.
\]
On the other hand, no edge leaves $U$, since its endpoint would then also be
reachable from $i$. When the outdegrees minus the indegrees are summed over
$U$, the edges inside $U$ cancel. The only remaining edges enter $U$, so
\[
 \sum_{h\in U}(\alpha_h-\beta_h)
 =-\#\{r:\phi(r)\notin U,\ \psi(r)\in U\}\le0,
\]
a contradiction. This proves the claim.

Choose a shortest directed path from $i$ to such an index $j$, and label its
edges by their rows:
\[
 i=i_0\xrightarrow{\,r_1\,}i_1
 \xrightarrow{\,r_2\,}\cdots
 \xrightarrow{\,r_m\,}i_m=j.
\]
Thus $\phi(r_s)=i_{s-1}$ and $\psi(r_s)=i_s$ for $1\le s\le m$. The rows
$r_1,\ldots,r_m$ are distinct. Define
$\phi'$ and $\psi'$ by interchanging the values of $\phi$ and $\psi$ on the
rows $r_1,\ldots,r_m$, and leaving all other rows unchanged. In $\phi'$, the
changes at every intermediate vertex cancel. Only one occurrence of $i$ is
lost and one occurrence of $j$ is gained. Therefore $\phi'$ has type
$\va-\ve_i+\ve_j$. Similarly, $\psi'$ has type
$\vb+\ve_i-\ve_j$.

The interchange preserves the sum of the two assignment costs in each row.
Since $\phi$ and $\psi$ are minimizing, the definition of $\mu_D$ gives
\begin{align*}
 \mu_D(\va)+\mu_D(\vb)
 &=\sum_{r=1}^n d_{r,\phi(r)}+\sum_{r=1}^n d_{r,\psi(r)}\\
 &=\sum_{r=1}^n d_{r,\phi'(r)}+\sum_{r=1}^n d_{r,\psi'(r)}\\
 &\ge \mu_D(\va-\ve_i+\ve_j)
       +\mu_D(\vb+\ve_i-\ve_j).
\end{align*}
This is the M-convex exchange inequality. Thus every assignment valuation is
M-convex, and $\mathcal A_{n,k}\subseteq\mathcal M_{n,k}$.

If $\mathcal A_{n,k}=\mathcal M_{n,k}$, then their positive dual cones are
equal. Conversely, if their positive dual cones are equal, the bipolar theorem
for the two closed convex cones gives
\[
 \mathcal A_{n,k}
 =\mathcal A_{n,k}^{\vee\vee}
 =\mathcal M_{n,k}^{\vee\vee}
 =\mathcal M_{n,k}.
\]
This proves \eqref{eq:equality-criterion}.
\end{proof}

\subsection{Variable persistence}

\begin{lemma}[Variable persistence]
\label{lem:variable-persistence}
If $\mathcal A_{n,k}\subsetneq\mathcal M_{n,k}$, then
$\mathcal A_{n,K}\subsetneq\mathcal M_{n,K}$ for every $K\ge k$.
\end{lemma}

\begin{proof}
It suffices to prove the assertion for $K=k+1$ and then iterate. By
Proposition~\ref{prop:comparison-polars}, the assumption is equivalent to
\[
 \BR_{\mathring L}(n,k)\subsetneq\BR_{\mathring Q}(n,k).
\]
Choose
\[
 \vg\in\BR_{\mathring Q}(n,k)\setminus\BR_{\mathring L}(n,k).
\]
Identify $H(n,k)$ with the face of $H(n,k+1)$ on which the last coordinate is
zero. Extend $\vg$ by zero outside this face, and continue to denote the
extended vector by $\vg$.

Let $G$ be a product of $n$ positive linear forms in $k+1$ variables. The
coefficients of $G$ indexed by $(\va,0)$ are exactly the coefficients of
\[
 G(x_1,\ldots,x_k,0),
\]
which is a product of $n$ positive linear forms in $k$ variables. Since
$\vg\in\BR_{\mathring Q}(n,k)$, it follows that $R_{\vg}$ is bounded on all
such $G$. Therefore
\[
 \vg\in\BR_{\mathring Q}(n,k+1).
\]

We claim that $\vg\notin\BR_{\mathring L}(n,k+1)$. Otherwise, let $C$ be an
upper bound for $R_{\vg}$ on $\mathring L(n,k+1)$. Take any strictly
Lorentzian polynomial $F$ in $k$ variables and regard it as a polynomial in
$k+1$ variables independent of $x_{k+1}$. By
\cite[Theorem~2.10]{BH20}, if $f$ is Lorentzian in $k$ variables and $A$ is a
$k\times m$ matrix with nonnegative entries, then $f(A\vx)$ is Lorentzian in
$m$ variables. Apply this with $m=k+1$ and $A=(I_k\ \ 0)$. It follows that $F$,
viewed as independent of $x_{k+1}$, is Lorentzian in $k+1$ variables. Hence
$F$ is a coefficientwise limit of strictly Lorentzian polynomials in $k+1$
variables. Since all coefficients of $F$ occurring in $R_{\vg}$ are positive,
continuity gives
\[
 R_{\vg}(F)\le C.
\]
This would bound $R_{\vg}$ on $\mathring L(n,k)$, contrary to the choice of
$\vg$. Thus
\[
 \vg\in
 \BR_{\mathring Q}(n,k+1)\setminus\BR_{\mathring L}(n,k+1).
\]
Proposition~\ref{prop:comparison-polars} now gives
$\mathcal A_{n,k+1}\subsetneq\mathcal M_{n,k+1}$. Iterating proves the result
for every $K\ge k$.
\end{proof}

\subsection{The equality cases}

We begin with an explicit formula for the restriction of an assignment
valuation to an edge of the simplex.

Let $D=(d_{rh})\in\mathbb R^{n\times k}$, and fix distinct indices $i,j$. For
each row $r$, put
\[
 s_r=d_{ri}-d_{rj},
\]
and let $s_{(1)}\le\cdots\le s_{(n)}$ be these $n$ numbers listed in
nondecreasing order. We use the convention $\sum_{t=1}^{0}s_{(t)}=0$.

\begin{lemma}[Restriction to an edge]
\label{lem:edge-restriction}
For every $a\in\{0,\ldots,n\}$,
\begin{equation}\label{eq:edge-restriction}
 \mu_D\bigl(a\ve_i+(n-a)\ve_j\bigr)
 =\sum_{r=1}^{n}d_{rj}+\sum_{t=1}^{a}s_{(t)}.
\end{equation}
\end{lemma}

\begin{proof}
Fix $a\in\{0,\ldots,n\}$. An assignment of type
$a\ve_i+(n-a)\ve_j$ is determined by the set $S$ of rows that it sends to
$i$. The set $S$ has cardinality $a$, and every row outside $S$ is sent to
$j$. Therefore
\begin{align*}
 \mu_D\bigl(a\ve_i+(n-a)\ve_j\bigr)
 &=\min_{\substack{S\subseteq\{1,\ldots,n\}\\ |S|=a}}
   \left(\sum_{r\in S}d_{ri}+\sum_{r\notin S}d_{rj}\right)\\
 &=\sum_{r=1}^{n}d_{rj}
   +\min_{\substack{S\subseteq\{1,\ldots,n\}\\ |S|=a}}
    \sum_{r\in S}(d_{ri}-d_{rj}).
\end{align*}
The last minimum is the sum of the $a$ smallest numbers among
$s_1,\ldots,s_n$. This proves \eqref{eq:edge-restriction}. Subtracting the
instance of \eqref{eq:edge-restriction} for $a-1$ from the instance for $a$
shows that the $a$th successive difference along the edge joining $n\ve_i$
and $n\ve_j$ is $s_{(a)}$.
\end{proof}

\begin{proposition}[Equality]
\label{prop:comparison-equality-cases}
The equality $\BR_{\mathring L}(n,k)=\BR_{\mathring Q}(n,k)$ holds when $k=2$,
when $n=2$, and when $k=3$ and $n\le5$.
\end{proposition}

\begin{proof}
Suppose that $k=2$, and let $\nu$ be M-convex. Write
\[
 \nu_r=\nu(r,n-r),
 \qquad
 s_r=\nu_r-\nu_{r-1}
 \qquad (1\le r\le n).
\]
We first verify that these successive differences are nondecreasing. For
$1\le r\le n-1$, apply the exchange axiom to
\[
 (r+1,n-r-1)
 \qquad\text{and}\qquad
 (r-1,n-r+1).
\]
The first vector has the larger first coordinate, and the second has the
larger second coordinate. Since there are only two coordinates, the exchange
moves one unit from the first coordinate of the first vector to its second
coordinate and makes the reverse move in the other vector. Both resulting
vectors are $(r,n-r)$. Hence
\[
 \nu_{r+1}+\nu_{r-1}\ge2\nu_r.
\]
This inequality is equivalent to $s_r\le s_{r+1}$. Therefore
\[
 s_1\le s_2\le\cdots\le s_n.
\]

Let the $r$th row of $D$ be $(s_r,0)$. Its row differences
$d_{r1}-d_{r2}=s_r$ are already nondecreasing.
Lemma~\ref{lem:edge-restriction} therefore gives, for every $0\le a\le n$,
\[
 \mu_D(a,n-a)
 =\sum_{r=1}^{a}s_r
 =\sum_{r=1}^{a}(\nu_r-\nu_{r-1})
 =\nu_a-\nu_0.
\]
Thus $\nu-\mu_D$ is the constant function with value $\nu_0$. A constant
function is affine on $H(n,2)$, so $[\nu]=[\mu_D]$. It follows that every
M-convex class belongs to $\mathcal A_{n,2}$. By
Lemma~\ref{lem:Mpolyhedral}, every element of $\mathcal M_{n,2}$ is a finite
sum of such classes. Hence
\[
 \mathcal M_{n,2}\subseteq\mathcal A_{n,2}.
\]
The reverse inclusion is Proposition~\ref{prop:comparison-polars}. Therefore
$\mathcal A_{n,2}=\mathcal M_{n,2}$, and
\eqref{eq:equality-criterion} gives the asserted equality of bounded-ratio
cones.

For $n=2$, the asserted equality follows from
\cite[Propositions~5.3 and~5.4]{HHSW}.

It remains to consider $k=3$ and $n\in\{3,4,5\}$. By
Corollary~\ref{cor:threevariable},
\[
 \BR_{\mathring L}(n,3)
 =\operatorname{cone}\bigl\{
 \vrho(\vb;i\mid jk):
 \vb\in H(n-2,3),\ i\in\{1,2,3\},\
 \{j,k\}=\{1,2,3\}\setminus\{i\}
 \bigr\}.
\]
The displayed rays are precisely the distinct extreme rays of this cone. Their
numbers for $n=3,4,5$ are, respectively, $9$, $18$, and $30$.

We computed $\BR_{\mathring Q}(n,3)$ for these three values of $n$ in exact
rational arithmetic using Algorithm~\ref{alg:BR}. In each case, the
computation gives the same conical description. Therefore
$\BR_{\mathring L}(n,3)=\BR_{\mathring Q}(n,3)$ for $n=3,4,5$.
\end{proof}

\subsection{The two base cases}

\begin{proposition}[Four variables in degrees three, four, and five]
\label{prop:base-34}
Let $R_A$ be the ratio of Theorem~\ref{thm:BRQ34}, and let $R_4$ and $R_5$
be obtained from it by adding $\ve_3$ and $2\ve_3$, respectively, to each of
its ten multi-indices:
\[
 R_4=\frac{P_{0040}P_{0211}P_{1012}P_{1120}P_{2011}}
          {P_{0022}P_{0220}P_{1021}P_{1111}P_{2020}},
 \qquad
 R_5=\frac{P_{0050}P_{0221}P_{1022}P_{1130}P_{2021}}
          {P_{0032}P_{0230}P_{1031}P_{1121}P_{2030}}.
\]
Then $\BR_{\mathring L}(n,4)\subsetneq\BR_{\mathring Q}(n,4)$ for $n=3,4,5$.
\end{proposition}

\begin{proof}
For $n=3$, this is Theorem~\ref{thm:BRLor34}. Let $\vg_A$ be the exponent
vector of $R_A$. Let $n\in\{4,5\}$ and write $R_n$ for the corresponding
ratio, with exponent vector $\vg_n$. Adding
$\ve_3$ or $2\ve_3$ to every multi-index preserves balance, with index sums
$(4,3,9,4)$ and $(4,3,14,4)$, respectively.

\emph{Boundedness on products.} Expanding the numerator and denominator in the
parameters $a_{ri}$ of Definition~\ref{def:Q} gives
\[
 D_n-N_n\in\mathbb Z_{\ge0}[a_{11},\dots,a_{n4}],
\]
where the difference has $9244$ nonzero monomials for $n=4$ and $89380$ for
$n=5$. All of them have positive coefficients, and the largest coefficients
are $72$ and $93$, respectively. Thus $R_n$ is subtraction-free with $C=1$ in
the sense of
Definition~\ref{def:sf}, so $R_n\le1$ on $\mathcal P(n,4)$ and
$\vg_n\in\BR_{\mathring Q}(n,4)$.

\emph{Unboundedness on the Lorentzian locus.}
By Theorem~\ref{thm:BRLor34}, $\vg_A\notin\BR_{\mathring L}(3,4)$.
Theorem~\ref{thm:BRL-duality} therefore gives an M-convex function $\nu_A$
with $\langle\vg_A,\nu_A\rangle<0$. After positive rescaling, assume that
$\langle\vg_A,\nu_A\rangle=-1$. By Definition~\ref{def:Fq}, the associated
cubic is
\[
 F_q^{\nu_A}
 =\sum_{\va\in H(3,4)}
   q^{\nu_A(\va)}\frac{\vx^{\va}}{\va!},
 \qquad 0<q\le1.
\]
Thus its ordinary and normalized coefficients are, respectively,
\[
 P_{\va}(F_q^{\nu_A})
 =\frac{q^{\nu_A(\va)}}{\va!},
 \qquad
 \widehat P_{\va}(F_q^{\nu_A})
 =q^{\nu_A(\va)}.
\]
Since $\nu_A$ is M-convex, \cite[Theorem~3.14]{BH20} shows that
$F_q^{\nu_A}$ is Lorentzian.

Put $G_q=x_3^{\,n-3}F^{\nu_A}_q$. This polynomial is Lorentzian because
$x_3^{\,n-3}$ is Lorentzian and Lorentzian polynomials are closed under
products \cite[Corollary~2.32]{BH20}. Every index occurring in $R_n$ exceeds
the corresponding index of $R_A$ by $(n-3)\ve_3$. Hence
$P_{\va+(n-3)\ve_3}(G_q)=P_{\va}(F^{\nu_A}_q)$ at all ten indices.

Evaluating $R_A$ on $F_q^{\nu_A}$ now gives
\[
\begin{aligned}
 R_A(F_q^{\nu_A})
 &=\prod_{\va\in H(3,4)}
   \left(\frac{q^{\nu_A(\va)}}{\va!}\right)^{(\gamma_A)_{\va}}\\
 &=\left(\prod_{\va\in H(3,4)}
   (\va!)^{-(\gamma_A)_{\va}}\right)
   q^{\langle\vg_A,\nu_A\rangle}\\
 &=\frac16q^{-1}
 =\frac1{6q}.
\end{aligned}
\]
Here the factor $1/6$ follows directly from the ten multi-indices in $R_A$.
Indeed, the products of the factorials of the numerator and denominator
multi-indices are $48$ and $8$, respectively, so the factorial correction is
$8/48=1/6$.
Consequently,
\[
 R_n(G_q)=R_A(F^{\nu_A}_q)=\frac1{6q}
 \ \longrightarrow\ \infty
 \qquad(q\to0^{+}).
\]
Although $G_q$ vanishes outside its translated cubic support, the ten
coefficients occurring in $R_n$ are positive. Thus $R_n$ is continuous there.
Moreover, $G_q$ is a coefficientwise limit of strictly Lorentzian polynomials
by Definition~\ref{def:strictly-lorentzian}. Hence $R_n$ is unbounded on
$\mathring L(n,4)$ and $\vg_n\notin\BR_{\mathring L}(n,4)$.
\end{proof}

\begin{proposition}[A uniform ternary obstruction]
\label{prop:base-63}
For every $N\ge6$, let
\begin{equation}\label{eq:RN}
 R_N=
 \frac{P_{(N-2,2,0)}P_{(N-2,0,2)}P_{(N-4,4,0)}P_{(N-4,0,4)}P_{(N-6,3,3)}}
      {P_{(N-2,1,1)}P_{(N-3,3,0)}P_{(N-3,0,3)}P_{(N-5,4,1)}P_{(N-5,1,4)}},
\end{equation}
with exponent vector $\vg_N$. Then the following statements hold.
\begin{enumerate}
 \item $R_N$ is bounded on products of $N$ positive linear forms, so
 $\vg_N\in\BR_{\mathring Q}(N,3)$.
 \item $R_N$ is unbounded on strictly Lorentzian ternary forms of degree $N$,
 so $\vg_N\notin\BR_{\mathring L}(N,3)$.
\end{enumerate}
Consequently, $\mathcal A_{N,3}\subsetneq\mathcal M_{N,3}$ for every
$N\ge6$.
\end{proposition}

\begin{proof}
The numerator and denominator multi-indices in \eqref{eq:RN} both sum to
$(5N-18,9,9)$. Thus $\vg_N$ is balanced. We prove the two assertions
separately.

\smallskip
\noindent
\emph{Step 1: boundedness on products.}
Fix a matrix $D=(d_{ri})\in\mathbb R^{N\times3}$, whose columns correspond
to $x,y,z$, and abbreviate
\[
 \mu_D(b,c)=\mu_D(N-b-c,b,c).
\]
An assignment contributing to $\mu_D(b,c)$ selects exactly one entry from
each row of $D$. Replace $d_{ri}$ by
\[
 d'_{ri}=d_{ri}-d_{r1}.
\]
For every assignment $\phi$, this replacement gives
\[
 \sum_{r=1}^{N}d'_{r,\phi(r)}
 =\sum_{r=1}^{N}d_{r,\phi(r)}-\sum_{r=1}^{N}d_{r1}.
\]
Consequently,
\[
 \mu_{D'}(b,c)=\mu_D(b,c)-\sum_{r=1}^{N}d_{r1}
\]
for every $b,c$. The ratio
\eqref{eq:RN} has five numerator factors and five denominator factors, so
this common shift cancels in $\langle\vg_N,\mu_D\rangle$. Relabeling $D'$ as
$D$, we may therefore assume that
\[
 d_{r1}=0\qquad\text{for every }r\in[N].
\]
In other words, the first column of $D$, corresponding to $x$, is zero.

Write $y_r=d_{r2}$ and $z_r=d_{r3}$ for the other two entries in row $r$,
and set
\[
 y(U)=\sum_{r\in U}y_r,
 \qquad
 z(V)=\sum_{r\in V}z_r.
\]
To form an assignment with $b$ copies of $y$ and $c$ copies of $z$, choose
disjoint row sets $U,V$ of sizes $b,c$ and assign all remaining rows to $x$.
Hence
\begin{equation}\label{eq:mu-bc}
 \mu_D(b,c)=
 \min_{\substack{U\cap V=\varnothing\\ |U|=b,\ |V|=c}}
 \bigl(y(U)+z(V)\bigr).
\end{equation}

Let $A$ and $C$ be the sets of indices of the two and four smallest values
$y_r$, respectively, where ties are resolved using one fixed order. Thus
$A\subset C$. Define $B\subset G$ in the same way using the values $z_r$.
Then
\[
 y(A)=\mu_D(2,0),\quad y(C)=\mu_D(4,0),\qquad
 z(B)=\mu_D(0,2),\quad z(G)=\mu_D(0,4).
\]
Choose disjoint three-element sets $Y,Z$ that attain the finite minimum
defining $\mu_D(3,3)$. Thus
\[
 y(Y)+z(Z)=\mu_D(3,3).
\]
The five assignments just chosen realize the five minima associated with the
numerator of $R_N$. Counting occurrences with multiplicity, they use $y$ on
$A\uplus C\uplus Y$ and $z$ on $B\uplus G\uplus Z$. We shall redistribute
these occurrences among assignments of the five denominator types
\[
 (1,1),\quad(3,0),\quad(0,3),\quad(4,1),\quad(1,4).
\]
This redistribution will preserve their total sum.

\smallskip
\noindent
\emph{Step 2: constructing the denominator assignments.}
For $\eta\in Y$ and $\zeta\in Z$, let
\[
 S=(Y\setminus\{\eta\})\cup\{\zeta\},
 \qquad
 T=(Y\cup Z)\setminus S.
\]
Thus $S$ contains two rows of $Y$ and one of $Z$, while $T$ contains one row
of $Y$ and two of $Z$.  We claim that $\eta,\zeta$ can be chosen so that
\begin{equation}\label{eq:good-split}
 A\not\subseteq S,\qquad |C\cap S|\le2,\qquad
 B\not\subseteq T,\qquad |G\cap T|\le2.
\end{equation}
In other words, $S$ leaves available at least one row of $A$ and two rows of
$C$, and $T$ does the same for $B$ and $G$.

To prove the claim, arrange the nine choices $(\eta,\zeta)$ in a $3\times3$
grid whose rows are indexed by $\eta\in Y$ and whose columns are indexed by
$\zeta\in Z$. Call a choice $y$-forbidden if $A\subseteq S$ or
$|C\cap S|=3$, and put $p=|C\cap Y|$.
\begin{itemize}
 \item If $p\le1$, then
 removing $\eta$ from $Y$ cannot increase the number of elements in $C$, and
 adding $\zeta$ can increase that number by at most one. Hence
 \[
  |C\cap S|\le |C\cap Y|+1=p+1\le2.
 \]
 Thus $|C\cap S|=3$ never occurs. Since $A\subset C$, we also have
 $|A\cap Y|\le1$. If $A\cap Y=\varnothing$, then $A\not\subseteq S$ because
 $S$ contains only one element outside $Y$. Suppose that $A\cap Y$ has one
 element. Then $A\subseteq S$ exactly when the other element of $A$ belongs
 to $Z$, $\zeta$ equals that element, and $\eta$ is not the element of
 $A\cap Y$. Thus $\zeta$ is fixed, while $\eta$ has two possible values.
 Hence at most two cells are $y$-forbidden.

 \item Suppose that $p=2$, and let $\eta_0$ be the unique element of
 $Y\setminus C$. If $\eta\ne\eta_0$, then removing $\eta$ from $Y$ leaves
 only one element of $C$ in $Y\setminus\{\eta\}$. Adding $\zeta$ contributes
 at most one further element of $C$. Therefore $|C\cap S|\le2$. If
 $\eta=\eta_0$, then $Y\setminus\{\eta\}$ contains both elements of
 $C\cap Y$. In this case $|C\cap S|=3$ exactly when $\zeta\in C$. Since
 $C$ already contains two elements of $Y$, at most two elements of $Z$ can
 belong to $C$. Thus $|C\cap S|=3$ holds in at most two cells, all in the
 grid row indexed by $\eta_0$.

 We next determine when $A\subseteq S$. If both elements of $A$ lie in $Y$,
 then $A=C\cap Y$, and $A\subseteq S$ exactly when $\eta=\eta_0$. These
 cells form the grid row indexed by $\eta_0$. If exactly one element of $A$
 lies in $Y$, then $A\subseteq S$ exactly when the other element of $A$
 belongs to $Z$, $\zeta$ equals that element, and $\eta$ is not the element
 of $A\cap Y$. When this is possible, it gives two cells in one grid column.
 One of the two choices is $\eta=\eta_0$, so this column meets the row indexed
 by $\eta_0$. At the intersection, $\zeta$ is the other element of $A$, which
 belongs to $C$ because $A\subset C$. Hence this intersection also satisfies
 $|C\cap S|=3$. The two cells in the column and the at most two cells in the
 row therefore have at most three cells in their union. If neither element of
 $A$ lies in $Y$, then $A\not\subseteq S$ because $S$ contains only one
 element outside $Y$. In every case, at most three cells are $y$-forbidden.

 \item Suppose that $p=3$. Then every element of $Y$ belongs to $C$.
 Removing $\eta$ leaves two elements of $C$ in $Y\setminus\{\eta\}$.
 Therefore $|C\cap S|=3$ exactly when $\zeta\in C$. Since $C$ contains the
 three elements of $Y$ and only one further element, at most one element of
 $Z$ belongs to $C$. Thus the cells satisfying $|C\cap S|=3$ lie in at most
 one grid column.

 Since $A\subset C$, at least one element of $A$ lies in $Y$. If both
 elements of $A$ lie in $Y$, then $A\subseteq S$ exactly when $\eta$ is the
 unique element of $Y\setminus A$. These cells form one grid row. If exactly
 one element of $A$ lies in $Y$, then $A\subseteq S$ exactly when the other
 element of $A$ belongs to $Z$, $\zeta$ equals that element, and $\eta$ is
 not the element of $A\cap Y$. When this is possible, the two resulting cells
 lie in the column indexed by the other element of $A$.
 Because this element belongs to $A\subset C$, this is the same column in
 which $|C\cap S|=3$. Consequently, all $y$-forbidden cells lie in one grid
 row together with one grid column.
\end{itemize}
The same argument applies to
\[
 T=(Z\setminus\{\zeta\})\cup\{\eta\}.
\]
It shows that the choices for which $B\subseteq T$ or $|G\cap T|=3$ either
consist of at most three cells or lie in one grid row together with one grid
column. Call these choices $z$-forbidden. The two forbidden sets cannot fill
the grid. If both contain at most three cells, their union contains at most
six cells. If only one is contained in a row together with a column, the
other has at most
three cells and cannot cover the four cells left outside that row and column.
If both are contained in a row together with a column, choose a grid row
different from the two specified rows and a grid column different from the
two specified columns. The cell at their intersection belongs to neither
forbidden set.
Hence a choice satisfying \eqref{eq:good-split} exists.

Fix such a choice. Choose $a\in A\setminus S$ and $b\in B\setminus T$, and
denote the other elements of $A$ and $B$ by $a'$ and $b'$, respectively. We
may choose $a$ and $b$ so that $a'\ne b'$. Indeed, if both choices are unique,
then $a'\in S$ and $b'\in T$, so $a'\ne b'$ because $S\cap T=\varnothing$.
If at least one choice is not unique, make that choice so that the two
remaining elements are distinct.

Choose a two-element set $C_{\rm out}\subseteq C\setminus S$ containing
$a$, and let $C_{\rm rest}=C\setminus C_{\rm out}$.  Similarly, choose a
two-element set $G_{\rm out}\subseteq G\setminus T$ containing $b$, and let
$G_{\rm rest}=G\setminus G_{\rm out}$.  The five required assignments are
now explicit:
\[
\begin{array}{c|c|c|c}
 &\text{type}&\text{rows assigned to $y$}&\text{rows assigned to $z$}\\ \hline
 1&(1,1)&\{a'\}&\{b'\}\\
 2&(3,0)&\{a\}\cup C_{\rm rest}&\varnothing\\
 3&(0,3)&\varnothing&\{b\}\cup G_{\rm rest}\\
 4&(4,1)&C_{\rm out}\cup(S\cap Y)&S\cap Z\\
 5&(1,4)&T\cap Y&G_{\rm out}\cup(T\cap Z).
\end{array}
\]
All rows not listed in a line of the table are assigned to $x$. Each
line therefore gives an admissible assignment of the indicated type. Indeed,
its two displayed sets are disjoint and have the required sizes. Counting
occurrences with multiplicity, the sets assigned to $y$ are exactly
$A\uplus C\uplus Y$, and the sets assigned to $z$ are exactly
$B\uplus G\uplus Z$. Hence the sum of the values of these five assignments
is
\[
 \mu_D(2,0)+\mu_D(0,2)+\mu_D(4,0)+\mu_D(0,4)+\mu_D(3,3).
\]
For each line of the table, the minimum in \eqref{eq:mu-bc} is at most the
value of the displayed assignment. Adding these five inequalities gives
\begin{align*}
 &\mu_D(1,1)+\mu_D(3,0)+\mu_D(0,3)
   +\mu_D(4,1)+\mu_D(1,4)\\
 &\qquad\le
 \mu_D(2,0)+\mu_D(0,2)+\mu_D(4,0)
   +\mu_D(0,4)+\mu_D(3,3).
\end{align*}
This says exactly that $\langle\vg_N,\mu_D\rangle\ge0$.  Since $D$ was
arbitrary, Proposition~\ref{prop:comparison-polars} gives
$\vg_N\in\BR_{\mathring Q}(N,3)$. This proves assertion~(i).

\smallskip
\noindent
\emph{Step 3: a Lorentzian family in degree six.}
For $0<t<\tfrac12$, let
\begin{equation}\label{eq:F6}
 F_t=\bigl(xy+xz+yz+t(x^2+y^2+z^2)\bigr)
 (x+y+tz)\bigl(t(x+y)+z\bigr)(x+z+ty)\bigl(t(x+z)+y\bigr).
\end{equation}
The Hessian of the quadratic factor is
\[
 \begin{pmatrix}
  2t&1&1\\
  1&2t&1\\
  1&1&2t
 \end{pmatrix}.
\]
It has eigenvalue $2t+2$ in the direction $(1,1,1)$ and eigenvalue $2t-1$
on the two-dimensional plane $a+b+c=0$. Thus it has exactly one positive
eigenvalue when $0<t<\tfrac12$, and the quadratic factor is strictly
Lorentzian. The other four factors are positive linear forms. Closure under
products \cite[Corollary~2.32]{BH20} therefore shows that $F_t$ is
Lorentzian. All its coefficients are positive.

After expanding the product in powers of $t$, the term independent of $t$ is
\[
 yz(x+y)(x+z)(xy+xz+yz).
\]
The coefficient of $t$ is obtained by choosing the $t$-part from exactly one
factor. Together, these two terms determine the ten coefficients needed
below:
\[
\begin{array}{c|ccccc}
 \va&420&402&240&204&033\\ \hline
 P_{\va}(F_t)&t+O(t^2)&t+O(t^2)&t+O(t^2)&t+O(t^2)&1+O(t^2)
\end{array}
\]
and
\[
\begin{array}{c|ccccc}
 \va&411&330&303&141&114\\ \hline
 P_{\va}(F_t)&3t+O(t^2)&2t+O(t^2)&2t+O(t^2)&4t+O(t^2)&4t+O(t^2).
\end{array}
\]

\smallskip
\noindent
\emph{Step 4: the same family in every degree $N\ge6$.}
Put $r=N-6$ and define
\[
 F_{N,t}=F_t\cdot(x+ty+t^2z)^r.
\]
For $r=0$ this is simply $F_t$. The extra factor is a product of positive
linear forms, so $F_{N,t}$ is Lorentzian and all its coefficients are
positive.

Only the first-order expansion of the extra factor is needed:
\[
 (x+ty+t^2z)^r=x^r+rtx^{r-1}y+O(t^2),
\]
where the second term is omitted when $r=0$. Consequently,
for each of the ten pairs $(b,c)$ occurring in \eqref{eq:RN},
\begin{equation}\label{eq:first-order-lift}
 \begin{split}
 P_{(N-b-c,b,c)}(F_{N,t})
 ={}&P_{(6-b-c,b,c)}(F_t)\\
 &+rtP_{(7-b-c,b-1,c)}(F_t)+O(t^2),
 \end{split}
\end{equation}
where the second term is also omitted when $b=0$.

There are only two shifted coefficients in \eqref{eq:first-order-lift} with
a nonzero constant term. Indeed, the displayed constant term shows directly
that
\[
 P_{123}(F_t)=2+O(t),
 \qquad
 P_{231}(F_t)=1+O(t),
\]
while every other shifted coefficient appearing in
\eqref{eq:first-order-lift} is $O(t)$. The first of these contributes only to
$P_{(N-6,3,3)}(F_{N,t})$ and gives
\[
 P_{(N-6,3,3)}(F_{N,t})=1+2rt+O(t^2)=1+O(t).
\]
The second changes the first-order expansion of
$P_{(N-5,4,1)}(F_{N,t})$ as follows:
\[
 \begin{split}
 P_{(N-5,4,1)}(F_{N,t})
 &=4t+rt+O(t^2)\\
 &=(N-2)t+O(t^2).
 \end{split}
\]
All other corrections are $O(t^2)$. Thus the five numerator coefficients
in \eqref{eq:RN} have leading terms
\[
 t,\quad t,\quad t,\quad t,\quad 1,
\]
and the five denominator coefficients have leading terms
\[
 3t,\quad2t,\quad2t,\quad(N-2)t,\quad4t.
\]
It follows that
\[
 R_N(F_{N,t})
 =\frac1{48(N-2)}\,t^{-1}\bigl(1+O(t)\bigr)
 \longrightarrow\infty
 \qquad(t\to0^+).
\]
All error terms here are for fixed $N$.

\smallskip
\noindent
\emph{Step 5: passage to strictly Lorentzian forms.}
Let $M>0$ and choose $t$ so small that $R_N(F_{N,t})>2M$. Since
$F_{N,t}$ is Lorentzian, it is a coefficientwise limit of strictly
Lorentzian polynomials. The five denominator coefficients in \eqref{eq:RN}
are positive at $F_{N,t}$, so $R_N$ is continuous there. A sufficiently close
strictly Lorentzian approximation therefore has ratio greater than $M$.
Since $M$ was arbitrary, $R_N$ is unbounded on $\mathring L(N,3)$, and hence
$\vg_N\notin\BR_{\mathring L}(N,3)$. Together with assertion~(i) and
Proposition~\ref{prop:comparison-polars}, this proves
$\mathcal A_{N,3}\subsetneq\mathcal M_{N,3}$.
\end{proof}

\subsection{The classification}

\begin{theoremC}
\label{thm:comparison-classification}
For all $n,k\ge2$,
\[
 \BR_{\mathring L}(n,k)=\BR_{\mathring Q}(n,k)
\]
if and only if $n=2$, or $k=2$, or $k=3$ and $n\le5$.  Equivalently, the
inclusion is strict precisely when $n\ge3$ and $k\ge4$, or $n\ge6$ and
$k\ge3$.
\end{theoremC}

\begin{proof}
Proposition~\ref{prop:comparison-equality-cases} gives equality in the listed
cases.

For strictness, suppose first that $n\ge6$. Then
Proposition~\ref{prop:base-63} gives
$\BR_{\mathring L}(n,3)\subsetneq\BR_{\mathring Q}(n,3)$, and
Lemma~\ref{lem:variable-persistence} propagates this to every $k\ge3$.
Suppose next that $3\le n\le5$ and $k\ge4$. Then
Proposition~\ref{prop:base-34} gives
$\BR_{\mathring L}(n,4)\subsetneq\BR_{\mathring Q}(n,4)$, with the three
degrees covered by $R_A$, $R_4$, and $R_5$, respectively.
Lemma~\ref{lem:variable-persistence} then propagates this to every $k\ge4$.
These cases constitute precisely the complement of the equality cases.
Equation~\eqref{eq:equality-criterion} then translates the statements back to
the two cones of bounded ratios.
\end{proof}

\section{A necessary majorization condition}

We conclude with a necessary condition that is particularly quick to check.
Recall the \defn{majorization order} on partitions. For
$\lambda = (\lambda_1 \ge \lambda_2 \ge \cdots \ge \lambda_l \ge 0)$ and
$\mu = (\mu_1 \ge \mu_2 \ge \cdots \ge \mu_l \ge 0)$,
\[
  \lambda \preceq \mu
  \quad\text{if and only if}\quad
  \sum_{j \le r} \lambda_j \;\le\; \sum_{j \le r} \mu_j \ \ \text{for all } r,
  \qquad\text{and}\qquad
  \sum_{j \le l} \lambda_j \;=\; \sum_{j \le l} \mu_j .
\]

\begin{definition}[{\rm \defn{Partitions $f_i^{\pm}$}}]\label{def:fpm}
Let $\vg \in \mathbb{Z}^{H(n,k)}$ be an integral exponent vector, and fix
$1 \le i \le k$. Form the partition $f_i^{+}(\vg)$ by taking, for each $\va$
with $\gamma_{\va}>0$, the part $\alpha_i$ with multiplicity
$\gamma_{\va}$. Define $f_i^{-}(\vg)$ analogously from the $\va$ with
$\gamma_{\va}<0$, taking $\alpha_i$ with multiplicity $-\gamma_{\va}$.
Pad both partitions with parts equal to $0$ so that they have the same number
of parts.
\end{definition}

\begin{proposition}\label{prop:majorization}
Let $\vg \in \BR_{\mathring{Q}}(n,k)$. Then, for every $1 \le i \le k$,
\[
  \sum_{\va \in H(n,k)} \gamma_{\va} \min(\alpha_i, r) \;\le\; 0
  \quad\text{for } r = 1, \dots, n-1,
  \qquad\text{and}\qquad
  \sum_{\va \in H(n,k)} \gamma_{\va}\, \alpha_i \;=\; 0 .
\]
If $\vg$ is integral, these conditions say exactly that
$f_i^{-}(\vg) \preceq f_i^{+}(\vg)$ for every $i$.
\end{proposition}

\begin{proof}
Fix $i$ and $r$. Choose a subset $R \subseteq \{1, \dots, n\}$ with $|R| = r$
and specialize
\[
  a_{si} = t \ \ (s \in R),
  \qquad
  a_{si'} = 1 \ \ \text{in all other cases},
  \qquad t > 0.
\]
This specialization lies in $\Rpos^{\,n \times k}$. A monomial of $P_{\va}$ is
indexed by an assignment $\phi$ of type $\va$, as in Definition~\ref{def:P},
and evaluates to $t^{|\{ s \in R \,:\, \phi(s) = i \}|}$. The largest such
exponent is $\min(\alpha_i, r)$ and is attained by at least one $\phi$. Hence
$P_{\va} \sim c_{\va} t^{\min(\alpha_i, r)}$ with $c_{\va}$ a positive integer,
and
\[
  R_{\vg} \;\sim\; C\, t^{\,\sum_{\va} \gamma_{\va} \min(\alpha_i, r)},
  \qquad C = \prod_{\va} c_{\va}^{\,\gamma_{\va}} > 0,
  \qquad t \to \infty ,
\]
so boundedness forces the exponent to be nonpositive. Only $|R|$ matters
because the linear forms may be permuted. For $r = n$, every monomial of
$P_{\va}$ has the same degree $\alpha_i$, and therefore
$R_{\vg} = C t^{\sum_{\va} \gamma_{\va} \alpha_i}$ exactly. Letting $t \to 0^{+}$
as well gives the stated equality.

To identify these inequalities with majorization, let $\lambda^{T}$ denote
the conjugate partition of $\lambda$. Recall that $\lambda \preceq \mu$ if and only if
$\mu^{T} \preceq \lambda^{T}$. Since
$\sum_{j} \min(\lambda_j, r) = \sum_{c \le r} \lambda^{T}_{c}$, the displayed
inequality for a given $r$ is precisely
$\sum_{c \le r} f_i^{+}(\vg)^{T}_{c} \le \sum_{c \le r} f_i^{-}(\vg)^{T}_{c}$,
and the displayed equality says that the two partitions have equal size.
\end{proof}

\begin{example}\label{ex:majorization-holds}
Let $n = k = 3$ and consider
\[
  R_{\vg} \;=\; \frac{P_{300}\, P_{030}\, P_{003}}{P_{111}\, P_{210}\, P_{012}} .
\]
Reading the coordinate $\alpha_i$ from each of the six multi-indices gives the
partitions
\[
  \begin{array}{lll}
    f_1^{+} = (3,0,0), &\qquad f_2^{+} = (3,0,0), &\qquad f_3^{+} = (3,0,0), \\[3pt]
    f_1^{-} = (2,1,0), &\qquad f_2^{-} = (1,1,1), &\qquad f_3^{-} = (2,1,0).
  \end{array}
\]
For instance, $f_1^{-}$ reorders the first entries $1$, $2$, $0$ of
$(1,1,1)$, $(2,1,0)$, and $(0,1,2)$. Hence
$f_i^{-} \preceq f_i^{+}$ for every $i$, strictly in each case. Equivalently,
in the linear form of
Proposition~\ref{prop:majorization}, the quantity
$\sum_{\va} \gamma_{\va} \min(\alpha_i, r)$ equals
\[
  -1, \ -1 \quad (i = 1), \qquad
  -2, \ -1 \quad (i = 2), \qquad
  -1, \ -1 \quad (i = 3), \qquad\text{for } r = 1, 2,
\]
and vanishes for $r = 3$, which is the balancing equation.
\end{example}

Thus majorization imposes $k(n-1)$ linear inequalities and $k$ linear
equations on $\vg$. These conditions define a cone containing
$\BR_{\mathring{Q}}(n,k)$, but the containment can be strict.

\begin{example}\label{ex:majorization-not-sufficient}
Let $n = 2$ and $k = 4$, and consider
\[
  R_{\vg} \;=\;
  \frac{P_{2000}\, P_{0200}\, P_{0020}\, P_{0002}\, P_{0110}\, P_{1001}}
       {P_{1100}^{\,3}\, P_{0011}^{\,3}} .
\]
For every $i$ one has
\[
  f_i^{+}(\vg) = (2,1,0,0,0,0)
  \qquad\text{and}\qquad
  f_i^{-}(\vg) = (1,1,1,0,0,0),
\]
so $f_i^{-} \preceq f_i^{+}$ strictly at every index. Nevertheless, $R_{\vg}$
is unbounded. Indeed, for the two linear forms with coefficient vectors
$(t,t,1,1)$ and $(1,1,1,1)$, one has
\[
  P_{2000} = P_{0200} = t, \quad
  P_{0020} = P_{0002} = 1, \quad
  P_{0110} = P_{1001} = t+1, \quad
  P_{1100} = 2t, \quad
  P_{0011} = 2,
\]
so that
\[
  R_{\vg} \;=\; \frac{t^{2}(t+1)^{2}}{(2t)^{3} \cdot 2^{3}}
  \;=\; \frac{(t+1)^{2}}{64\,t}
  \;\longrightarrow\; \infty,
  \qquad t \to \infty .
\]
\end{example}

\begin{corollary}\label{cor:maj-lorentzian}
Every bounded ratio on the strictly Lorentzian polynomials of degree $n$ in $k$
variables satisfies the majorization conditions. More precisely, if
$\vg \in \BR_{\mathring{L}}(n,k)$, then, for every $1 \le i \le k$,
\[
  \sum_{\va \in H(n,k)} \gamma_{\va} \min(\alpha_i, r) \;\le\; 0
  \quad\text{for } r = 1, \dots, n-1,
  \qquad\text{and}\qquad
  \sum_{\va \in H(n,k)} \gamma_{\va}\, \alpha_i \;=\; 0 .
\]
\end{corollary}

\begin{proof}
Let $C$ be a positive constant such that
$\prod_{\va} P_{\va}^{\,\gamma_{\va}} \le C$ for all
$(P_{\va}) \in \mathring{L}(n,k)$. Let $Q_t$ be the product of linear forms
specialized as in the proof of Proposition~\ref{prop:majorization}. All its
coefficients are positive. Each nonzero linear factor has nonnegative
coefficients and is stable. Their product $Q_t$ is therefore stable and hence
Lorentzian \cite[Proposition~2.2]{BH20}. By
Definition~\ref{def:strictly-lorentzian}, it is a limit of strictly Lorentzian
polynomials. Since $R_{\vg}$ is continuous at coefficient vectors with
positive entries, it follows
that $R_{\vg}(Q_t) \le C$ for every $t > 0$. The proof of
Proposition~\ref{prop:majorization} shows that $R_{\vg}(Q_t)$ grows like
$t^{\sum_{\va} \gamma_{\va} \min(\alpha_i, r)}$ as $t \to \infty$, so the
exponent is nonpositive. The case $r = n$ with $t \to 0^{+}$ gives the stated
equality.
\end{proof}

\begin{remark}\label{rem:fast-check}
The majorization condition consists of $k(n-1)$ inequalities and $k$ equations,
one for each pair $(i,r)$ with $1 \le i \le k$ and $1 \le r \le n$. Each is a
single linear functional of $\vg$. The condition can be tested for a given
$\vg$ in
$O(ks \log s)$ arithmetic operations, where $s$ is the number of nonzero
entries of $\vg$, that is, the number of factors of the ratio.

No comparable test is available for boundedness itself. For $n=2$, write
$a_{ij}=\gamma_{\ve_i+\ve_j}$. As explained in
Remark~\ref{rem:recover-HHSW}, membership in
$\BR_{\mathring L}(2,k)$ is equivalent to
\[
 \sum_{i<j}a_{ij}d_{ij}\le0
 \qquad\text{for every }d\in\mathrm{Cut}_k.
\]
Thus membership amounts to testing the sign of the associated linear
functional on all cuts, which is a weighted cut optimization problem. See
\cite{DL97} for background on cut and metric cones.
\end{remark}

\section{Concluding remarks}

We finish by recording a subtraction-free strengthening suggested by the
computations above.

\begin{definition}[{\rm \defn{Subtraction-free ratio}}]\label{def:sf}
Let $\vg \in \mathbb{Z}^{H(n,k)}$ and write
\[
  N_{\vg} \;=\; \prod_{\gamma_{\va} > 0} P_{\va}^{\,\gamma_{\va}},
  \qquad
  D_{\vg} \;=\; \prod_{\gamma_{\va} < 0} P_{\va}^{\,-\gamma_{\va}},
\]
so that $R_{\vg} = N_{\vg} / D_{\vg}$. The ratio $R_{\vg}$ is
\defn{subtraction-free} if there exists a positive constant $C$ such that
$C\, D_{\vg} - N_{\vg}$, expanded as a polynomial in the parameters $a_{ri}$,
has only nonnegative coefficients. This extends the definition $C=1$ used
in \cite[Definition~2.7]{SG25}.
\end{definition}

By definition, every subtraction-free ratio is bounded, with
$f(\vg) \le C$, because
$C\, D_{\vg} - N_{\vg} \ge 0$ on $\Rge^{\,n \times k}$.

Conjecture~\ref{conj:sf} extends to arbitrary $n$ the corresponding conjecture
of \cite[Conjecture~5.5]{HHSW} for $n = 2$ in a weaker form. That conjecture is stated in terms
of the normalized coefficients $\va!\,P_{\va}$. The required constant is the
power of two $2^{\sum_i \gamma_{2\ve_i}}$, which is exactly the factorial
correction $\prod_{\va} (\va!)^{\gamma_{\va}}$. Indeed, in degree two,
$\va! = 2$ for $\va = 2\ve_i$ and $\va! = 1$ otherwise. In degree $n$, the
numbers $\va! = \alpha_1! \cdots \alpha_k!$ are no longer powers of two, so no
such constant is available. The correction is absorbed into the coefficients
$P_{\va}$ used here. This is why Definition~\ref{def:sf} allows an unspecified
constant $C$.

\begin{conjecture}\label{conj:sf}
If
$\vg \in \BR_{\mathring{Q}}(n,k) \cap \mathbb{Z}^{H(n,k)}$, then $R_{\vg}$ is
subtraction-free.
\end{conjecture}

The conjecture holds with $C = 1$ for all nine extreme rays of
$\BR_{\mathring{Q}}(3,3)$ and for all eighty extreme rays of
$\BR_{\mathring{Q}}(3,4)$.

\subsection*{Acknowledgements}

We thank the UCLA Olga Radko Endowed Math Circle (ORMC) for its financial and
logistical support of our research through the Vertical Research Integration
(VRI) program. We are grateful to Oleg Gleizer, Igor Pak, and Dimitri
Shlyakhtenko for organizing a research group for high school students in ORMC.

Special thanks to June Huh for stating and motivating the problem while the
authors were working on \cite{HHSW}, during the thematic year at IAS. We thank
Igor Pak for countless discussions and for his generous guidance on both
research and mathematical writing. We also thank Swee Hong Chan, Daoji Huang,
Mateusz Micha\l{}ek and Botong Wang for many fruitful discussions and their
support.

Claude Opus 5, Anthropic and ChatGPT Sol5.6 were used for calculations, proof
ideas, and editorial assistance. Some of their suggestions were helpful, while
others were misleading. The authors independently verified all computations
and take full responsibility for the contents of this paper.

\end{document}